\documentclass[11pt]{article}
\usepackage{benstyle}
\usepackage{longtable}
\usepackage{amsmath}
\usepackage[ruled,linesnumbered]{algorithm2e}
\usepackage{hyperref}
\usepackage[capitalise]{cleveref}
\usepackage[percent]{overpic}
\usepackage{multirow}
\newcounter{protocol}


\makeatletter
\newcommand{\pushright}[1]{\ifmeasuring@#1\else\omit\hfill$\displaystyle#1$\fi\ignorespaces}
\newcommand{\pushleft}[1]{\ifmeasuring@#1\else\omit$\displaystyle#1$\hfill\fi\ignorespaces}
\makeatother

\usepackage{comment}
\usepackage{makecell}
\usepackage{arydshln}
\providecommand{\keywords}[1]
{
\textbf{{Keywords:}} #1
}

\providecommand{\AMSclass}[1]
{
\textbf{{Mathematics Subject Classification:}} #1
}

\usepackage{cite}

\title{Restarts subject to approximate sharpness:\\A parameter-free and optimal scheme for first-order methods}
\author{Ben Adcock, Matthew J. Colbrook, Maksym Neyra-Nesterenko}
\date{}

\begin{document}

\maketitle

\footnotetext[2]{\textit{Corresponding author:} m.colbrook@damtp.cam.ac.uk\\\indent{}DAMTP, Centre for Mathematical Sciences, University of Cambridge, UK}

\begin{abstract}
Sharpness is an almost generic assumption in continuous optimization that bounds the distance from minima by objective function suboptimality. It facilitates the acceleration of first-order methods through \textit{restarts}. However, sharpness involves problem-specific constants that are typically unknown, and restart schemes typically reduce convergence rates. Moreover, these schemes are challenging to apply in the presence of noise or with approximate model classes (e.g., in compressive imaging or learning problems), and they generally assume that the first-order method used produces feasible iterates. We consider the assumption of \textit{approximate sharpness}, a generalization of sharpness that incorporates an unknown constant perturbation to the objective function error. This constant offers greater robustness (e.g., with respect to noise or relaxation of model classes) for finding approximate minimizers. By employing a new type of search over the unknown constants, we design a restart scheme that applies to general first-order methods and does not require the first-order method to produce feasible iterates. Our scheme maintains the same convergence rate as when the constants are known. The convergence rates we achieve for various first-order methods match the optimal rates or improve on previously established rates for a wide range of problems. We showcase our restart scheme in several examples and highlight potential future applications and developments of our framework and theory.
\end{abstract}

\vspace{2mm}


\AMSclass{65K0, 65B99, 68Q25, 90C25, 90C60}

\section{Introduction}

First-order methods are the workhorse of much of modern continuous optimization \cite{chambolle2016introduction,beck2017first,aster2018parameter,nesterov2018lectures}. These methods are widely used to solve large-scale problems due to their excellent scalability and ease of implementation. However, standard first-order methods often converge slowly, for instance, when applied to non-smooth objective functions or functions lacking strong convexity. This limitation has motivated extensive research aimed at speeding up these methods \cite{nesterov1983method,nesterov2003introductory,beck2009fast,kim2016optimized,renegar2016efficient,renegar2019accelerated,d2021acceleration}.

Recently, there has been significant interest in using \textit{restarts} to accelerate the convergence of first-order methods \cite{becker2011templates,nesterov2013gradient,iouditski2014primal,lin2014adaptive,su2014differential,giselsson2014monotonicity,o2015adaptive,fercoq2016restarting,freund2018new,necoara2019linear,fercoq2019adaptive,kerdreux2019restarting,roulet2020sharpness,roulet2020computational,renegar2021simple,colbrook2021warpd,adcock2022efficient,neyra-nesterenko2022nestanets,kerdreux2022restarting}. A restart scheme involves repeatedly using the output of an optimization algorithm as the initial point for a new instance, or ``restart''. Before executing this new instance, the scheme may also adjust the algorithm's parameters. Under the right conditions, the objective function error and the feasibility gap decay faster with the restarted scheme than with the standard (non-restarted) first-order method.

In this work, we relax previous assumptions and introduce a new restart scheme with optimal rates. This scheme applies to a broad class of convex optimization problems, generalizes and improves upon various existing schemes, and achieves optimal complexity bounds for a wide range of problems. Moreover, our scheme is parameter-free, up to the parameters used by the first-order method employed in our scheme.

\subsection{The problem}
\label{sec:prob_setup_intro}

We consider the general convex optimization problem
\ea{
\min_{x \in Q} f(x), \label{eqn:cvx-problem}
}
where $f : D \rightarrow \bbR$ is a proper, closed convex function with non-empty effective domain $D \subseteq \bbC^n$, and $Q \subseteq \bbC^n$ is a closed, convex set with $Q\subset D$. Let $\hat{f}$ denote the optimal value of \eqref{eqn:cvx-problem} and $\widehat{X}$ denote the set of minimizers of $f$ over $Q$, where we assume that $\widehat{X}$ is non-empty.

Our key assumption is that $f$ satisfies the following \textit{approximate sharpness condition}
\ea{
d(x, \widehat{X}) \leq \left( \frac{f(x) - \hat{f} + g_Q(x) + \eta}{\alpha} \right)^{1/\beta} \qquad \forall x \in D, \label{eqn:sharpness}
}
for a metric $d$ on $\bbC^n$ and some constants $\alpha > 0$, $\beta \geq 1$, $\eta \geq 0$.  We slightly abuse notation by defining $d(x, S) := \inf_{z \in S} d(x, z)$ for a set $S \subseteq \bbC^n$. Here, $g_Q : D \rightarrow \bbR_+$ is a function with
\eas{
    g_Q(x) = 0 \iff x \in Q
}
such that if a sequence $\{x_m\} \subset D$ has $\lim_{m\rightarrow\infty}d(x_m, Q) = 0$, then $\lim_{m\rightarrow\infty}g_{Q}(x_m) =0$. We assume that the function $g_{Q}$ is known but that suitable constants $\eta$, $\alpha$, and $\beta$ (or a subset thereof) are unknown. We refer to $g_Q$ as the \textit{feasibility gap} function and $f - \hat{f}$ as the \textit{objective (function) error}. 

To develop a restart scheme that accelerates an optimization algorithm for solving \eqref{eqn:cvx-problem}, we assume that $f$ satisfies \eqref{eqn:sharpness}. We also assume access to an optimization algorithm, $\Gamma : \bbR_{++} \times \bbR_{++} \times D \rightarrow D$, which maps $(\delta, \epsilon,  x_0)$ to $x$ such that
\be{
\label{Gamma-delta-to-eps}
    d(x_0, \widehat{X}) \leq \delta \quad \implies \quad f(x) - \hat{f} + g_Q(x) \leq \epsilon,\ \text{where }x = \Gamma(\delta,\epsilon,x_0).
}
In essence, if the initial value $x_0$ is within a distance $\delta$ of an optimal solution, the algorithm outputs an $x$ that is $\epsilon$-suboptimal, meaning $f(x) - \hat{f} \leq \epsilon$, and $\epsilon$-feasible, meaning $g_Q(x) \leq \epsilon$, for \eqref{eqn:cvx-problem}. The assumption \eqref{Gamma-delta-to-eps} is a standard condition found in typical convergence analyses of first-order methods. The algorithms $\Gamma$ considered in this paper are iterative. We define the \textit{cost function} 
\bes{
\cC_\Gamma : \bbR_{++}  \times \bbR_{++} \rightarrow \bbN,
}
where $\cC_\Gamma (\delta, \epsilon)$ represents an upper bound on the number of iterations required by $\Gamma$ to compute $x = \Gamma(\delta, \epsilon, x_0)$ for any initial value $x_0$ satisfying $d(x_0,\widehat{X}) \leq \delta$. This framework can be extended to include cost in terms of floating-point operations or other measures of time complexity. It is assumed that $\cC_\Gamma$ is non-decreasing in its first argument and non-increasing with respect to its second. In~\cref{sec:FOM_examples}, we describe various examples of first-order optimization methods that satisfy \cref{Gamma-delta-to-eps} and analyze their cost functions. See also \cite{renegar2021simple}.

Our restart scheme decreases the sum of the objective and feasibility gap functions after each restart. Moreover, we only use \R{eqn:sharpness} in our analysis each time we restart. As a result, \R{eqn:sharpness} can be relaxed and we only need \R{eqn:sharpness} to hold within the sublevel set \textit{sublevel} set
$
\{ x \in D : f(x) + g_Q(x) \leq f(x_0) + g_Q(x_0) \}
$
for a starting vector $x_0 \in D$.

\subsection{Motivations} \label{sec:motivations}

The assumption \eqref{eqn:sharpness} is considerably weaker than typical assumptions for acceleration, such as strong convexity. It can be considered an \textit{approximate} version of the sharpness condition considered in \cite{roulet2020sharpness} (see \eqref{loja}). We discuss its connections to other error bounds in \cref{sec:prev_work}. There are two principal distinctions between \eqref{eqn:sharpness} and sharpness. First, we do not assume that the sharpness condition is exact; instead, we include an additional term $\eta\geq 0$ that controls the approximation. This adjustment is crucial in many applications and dealing with noisy data. It also provides greater \textit{robustness} of our results. For instance, in the context of sparse recovery, \eqref{eqn:sharpness} covers both \textit{noisy} measurements and \textit{approximately} sparse vectors \cite{colbrook2021warpd}, which is more realistic than exact sparse recovery from noiseless measurements. We further discuss this problem in \cref{sec:sparse-recovery}. Second, our approach does not require that the iterates of our algorithm be feasible, thanks to the additional feasibility gap function $g_Q$. This provides added flexibility and efficiency in selecting the first-order method for the restart scheme, such as the primal-dual algorithm considered in \cref{sec:PD_constrained_alg}.

The other key motivation for this work is that we do not assume prior knowledge of the constants $\alpha$, $\beta$, and $\eta$. When these parameters are known, deriving a restart scheme is relatively straightforward. However, these constants are rarely known in practice. For instance, although sharpness is established for general subanalytic convex functions \cite{bolte2007lojasiewicz}, the proof relies on topological arguments that are not constructive. Another example can be seen in the sparse recovery problem, which we will discuss next. In some cases, approximate bounds for one or more of these constants may be available. However, if these bounds are loose -- particularly global ones, which can be overly pessimistic near minimizers -- they can lead to inefficient schemes. Our method eliminates the necessity for such precise bounds but still accommodates the inclusion of prior information about the constants (e.g., exact values or ranges) if available.

\subsection{Example: sparse recovery}\label{sec:sparse-recovery}

To illustrate these motivations, consider the classical sparse recovery problem of reconstructing an approximately sparse vector $x \in \bbR^n$ from a small collection of noisy linear measurements $y = A x + e \in \bbR^m$, where $m \ll n$. As discussed in \cref{sec:motivations}, in practice, $x$ is not exactly sparse, but \textit{approximately sparse}. This is usually quantified by the best $s$-term approximation error
$$
\sigma_s(x)_{\ell^1} = \min \{ \nmu{u - x}_{\ell^1} : u \in \bbR^n, \text{$u$ $s$-sparse} \},
$$
where $1 \leq s \leq n$ is the sparsity.
It is also typical to consider noisy measurements, where the noise vector $e$ satisfies $\nmu{e}_{\ell^2} \leq \varsigma$ for some noise level $0 < \varsigma \ll \infty$. In this practical setting of approximate sparsity and noisy measurements, it is generally impossible to recover $x$ exactly from its measurements $y$. Rather, the goal is to recovery $x$ \textit{accurately} and \textit{stably}, i.e., up to an error scaling linearly in $\sigma_{s}(x)_{\ell^1}$ and $\varsigma$ \cite{adcock2021compressive,foucart2013invitation}. A standard means to do this involves solving the Quadratically-Constrained Basis Pursuit (QCBP) problem
\ea{
\min_{z \in \bbR^n} \nmu{z}_{\ell^1} \ \text{subject to} \ \nmu{Az-y}_{\ell^2} \leq \varsigma.
\label{eqn:QCBP}
}
The objective of compressed sensing theory is then to derive conditions on $A$ (in terms of $m$ and $s$) that ensure any minimizer $\hat{x}$ of \eqref{eqn:QCBP} recovers $x$ accurately and stably.
One such condition is the \textit{robust Null Space Property (rNSP)} in Definition \ref{def:rNSP} (this condition is implied by, and is, therefore, weaker than, the well-known \textit{Restricted Isometry Property (RIP)}). Later, in Proposition \ref{prop:sharpness_of_QCBP}, we show that the rNSP implies accurate and stable recovery and approximate sharpness for \eqref{eqn:QCBP}. If $A$ has the rNSP of order $s$ with constants $0 < \rho < 1$ and $\gamma > 0$, then, firstly, any minimizer $\hat{x}$ of \eqref{eqn:QCBP} satisfies
\be{
\label{acc-stable-rec}
\nm{x - \hat{x}}_{\ell^2} \leq \hat{c}_1 \frac{\sigma_s(x)_{\ell^1}}{\sqrt{s}} + \hat{c}_2 \varsigma,
}
(see, e.g., \cite[Thm.\ 5.17]{adcock2021compressive})
and secondly, the approximate sharpness condition \eqref{eqn:sharpness} holds for \eqref{eqn:QCBP} with
$$
g_{Q}(z ) = \sqrt{s} \max \{ \nm{A z - y}_{\ell^2} - \varsigma , 0 \}, \quad \alpha = \hat{c}_3\sqrt{s}, \quad \beta = 1, \quad \eta = \hat{c}_4 \sigma_s(x)_{\ell^1} + \hat{c}_5 \varsigma\sqrt{s} ,
$$
where $\hat{c}_1,\hat{c}_2,\hat{c}_3,\hat{c}_4,\hat{c}_5 > 0$ are constants depending on $\rho$ and $\gamma$ only (see \cite[Theorem 3.3]{colbrook2021warpd}, as well as \cref{prop:sharpness_of_QCBP}).

This simple example illustrates the main motivations for this paper. First, \eqref{eqn:QCBP} satisfies approximate sharpness under \textit{exactly the same} conditions that imply accurate and stable recovery \eqref{acc-stable-rec} via its minimizers. Second, the approximate sharpness parameter $\eta$ is the same (up to possible constants) as the error bound \eqref{acc-stable-rec}. Therefore, it is acceptable to solve \eqref{eqn:QCBP} only down to an error proportional to $\eta$. Third, and most crucially, the approximate sharpness constants are typically unknown. In this example, $\alpha$ and $\eta$ depend on the rNSP constants $\rho$ and $\gamma$, amongst other factors. However, in general, given $s$ and $A$, computing $\rho$ and $\gamma$ is well-known to be NP-hard. The constant $\eta$ also depends on $\sigma_{s}(x)_{\ell^1}$, which is generally also unknown.

\subsection{Contributions}

Our main contribution is a restart scheme, detailed in \cref{alg:restart-unknown-constsC}, which operates under the approximate sharpness condition \eqref{eqn:sharpness} when the constants $\alpha$, $\beta$ and $\eta$ (or any subset thereof) are unknown. In the most general case where all three constants are unknown, our approach relies on several parameters: \textit{bases} $a,b > 1$, a \textit{scale factor} $r \in (0,1)$, \textit{estimates} $\alpha_0 > 0$, $\beta_0 \geq 1$, and a so-called \textit{schedule criterion function}. The method employs a logarithmic grid search over $\alpha$ and $\beta$ using the bases $a,b$, estimates $\alpha_0,\beta_0$, and the schedule criterion function to determine the order through which the grid is searched.  The scale factor $r$ is used to adjust the parameters of the first-order method at each restart, ensuring rapid convergence. We will provide full details of this scheme in \cref{sec:restart-unknown-constants}, but now let us present our main results for it.

\begin{theorem}
\label{thm:MAIN}
Suppose that $f$ satisfies \eqref{eqn:sharpness} for some unknown constants $\alpha$, $\beta$ and $\eta$. 
Consider \cref{alg:restart-unknown-constsC} for fixed $a, b > 1$, $0 < r<1$, $\alpha_0 > 0$, $\beta_0 \geq 1$ and schedule criterion function as in \cref{rad_ord_cor1} (unknown $\alpha$ and $\beta$), \cref{known_alpha} (known $\alpha$, unknown $\beta$) or \cref{known_beta} (unknown $\alpha$, known $\beta$). Then running \cref{alg:restart-unknown-constsC} with
$$
t \gtrsim K(\varepsilon),\quad \varepsilon \rightarrow 0^{+},
$$
(total inner) iterations, where $K(\varepsilon)$ is given in \eqref{L-eps-def},
implies that
$$
f(x^{(t)}) - \hat{f} + g_Q(x^{(t)}) \leq \max \{ \eta , \varepsilon \}.
$$
Let $\beta_*=b^{\lceil \log_b(\beta/\beta_0) \rceil}\beta_0$. If, in addition, $\cC_\Gamma$ satisfies
\begin{equation}
\label{known_const_C_assump}
\cC_\Gamma(\delta, \epsilon) \leq {C \delta^{d_1}}/{\epsilon^{d_2}}+1, \qquad C,d_1,d_2 > 0,
\end{equation}
for all $\delta, \epsilon > 0$, then
\begin{equation}
\label{sum_asy_rat}
K(\varepsilon)\leq\hat{C}
\begin{cases}
\epsilon_0^{d_1/\beta_*-d_2}\left\lceil{\log(\epsilon_0/\varepsilon)} \right\rceil,\quad &\text{if $d_2\leq d_1/\beta_*$},\\
\varepsilon^{d_1/\beta_*-d_2}\left\lceil{\log(\epsilon_0/\varepsilon)} \right\rceil,\quad &\text{if $d_2>d_1/\beta_*$},
\end{cases}
\end{equation}
where $\hat{C}$ is independent of $\varepsilon$ (but depends on $r,a,b,\alpha,\beta_*,\alpha_0,\beta_0,d_1$ and $d_2$). Explicit forms for $\hat{C}$ in~\eqref{sum_asy_rat} are given in \cref{sec:restart-unknown-constants}.
\end{theorem}

\cref{sec:expansion_functions} contains the proof of this theorem.
A few comments are warranted:
\begin{itemize}
	\item \textbf{The role of $\varepsilon$:} Note that $\varepsilon$ is not an actual parameter of the algorithm but is used to describe the algorithm’s behavior as the number of iterations increases.
	\item \textbf{Non-uniqueness of constants:} It is possible for a problem \eqref{eqn:cvx-problem} to satisfy the approximate sharpness condition \eqref{eqn:sharpness} with different values for $\alpha, \beta,$ and $\eta$. This can lead to varied convergence rates and the constant $\hat{C}$ in \eqref{sum_asy_rat}. In such cases, \cref{thm:MAIN} says that for a given accuracy threshold $\varepsilon \geq \eta$, we can take the best rate of convergence/iteration bound over different approximate sharpness constants. Hence, non-uniqueness is beneficial. Furthermore, it is possible to take advantage of local approximate sharpness near minimizers.
	\item \textbf{Convergence to order $\eta$:} \cref{thm:MAIN} does not guarantee a decrease of the objective function error below $\eta$ as $\varepsilon \rightarrow 0^{+}$. This is reasonable in practice. For example, in the case of sparse recovery, $\eta$ is a combination of the noise level and the best $s$-term approximation error (recall \cref{sec:sparse-recovery}). Therefore, there is little benefit in decreasing the objective function error below $\eta$ since the error in the recovered vector will generally be at least $\eta$ in magnitude.
	\item \textbf{Assumption on convergence rates:} The assumption in \eqref{known_const_C_assump} is generic for convergence rates of first-order methods. Examples of these rates are provided in \cref{sec:FOM_examples}. The $+1$ term is included in \eqref{known_const_C_assump} since we often have a bound of the form
$
\cC_\Gamma(\delta, \epsilon) \leq \lceil{C \delta^{d_1}}/{\epsilon^{d_2}}\rceil.
$
\item \textbf{Initial estimates and scale factor:} The parameters $\alpha_0 > 0$ and $\beta_0 \geq 1$ in \cref{alg:restart-unknown-constsC} are estimates for the true $\alpha$, $\beta$. Setting $\alpha_0 = \beta_0 = 1$ is advisable if no prior estimates are available. Regarding the scale factor $r$, as discussed in \cref{sec:bestr}, a good choice is $r = \E^{-1/d_2}$.
\end{itemize}

As mentioned, our scheme conducts a grid search over the parameters $\alpha$, $\beta$ using the bases $a,b > 1$ and estimates $\alpha_0,\beta_0$. The order of this search is determined by a schedule criterion function, which is detailed in \cref{def:grid-search-bijection} and discussed subsequently. This new idea offers flexibility depending on which parameters are known or unknown and facilitates a unified framework for proving convergence results (e.g., using \cref{thm:restart-unknown-consts-err}). We will provide further details in \cref{sec:restart-unknown-constants}. Notably, this framework enables us to conduct searches over a nonuniform grid (see \cref{rad_ord_cor1}) that searches more in iteration than parameter index space. This approach is crucial for developing a search method for unknown parameters that does not suffer from reduced convergence rates.

\begin{table}[t]
\centering
\begin{tabular}{|c|rc|c|}
\hline
\multirow{2}{*}{Objective function class/structure}& \multicolumn{2}{c|}{\multirow{2}{*}{Asymptotic bound for $K(\varepsilon)$}} & \multirow{2}{*}{Example method}\\
&\multicolumn{2}{c|}{}&\\
\hline
\hline
\multirow{4}{*}{\thead{$L-$smooth \\ See \cref{L_smooth} \\(NB: must have $\beta\geq 2$)}}&
\multirow{2}{*}{$\beta=2$:\!\!\!\!}&
\multirow{2}{*}{$\sqrt{L/\alpha}\cdot\log(1/ \varepsilon)$}&
\multirow{4}{*}{\thead{Nesterov's method\\$d_1=1$, $d_2=1/2$\\See \cref{sec:alg_NESTA_smoothing}}}\\
&&&\\\cdashline{2-3}
&\multirow{2}{*}{$\beta>2$:\!\!\!\!}&
\multirow{2}{*}{$\frac{\sqrt{L}}{\alpha^{1/\beta_*}}\cdot\frac{1}{\varepsilon^{1/2-1/\beta_*}}$}&\\
&&&\\\hline
\multirow{4}{*}{\thead{$(u,v)-$smoothable\\See \cref{def:uv-smoothable}}}&
\multirow{2}{*}{$\beta=1$:\!\!\!\!}&
\multirow{2}{*}{$\frac{\sqrt{ab}}{\alpha}\cdot\log(1/ \varepsilon)$}&
\multirow{4}{*}{\thead{Nesterov's method\\with smoothing\\$d_1=1$, $d_2=1$\\See \cref{Nesta_smooth_sec}}}\\
&&&\\\cdashline{2-3}
&\multirow{2}{*}{$\beta>1$:\!\!\!\!}&
\multirow{2}{*}{$\frac{\sqrt{ab}}{\alpha^{1/\beta_*}}\cdot\frac{1}{\varepsilon^{1-1/\beta_*}}$}&\\
&&&\\\hline
\multirow{6}{*}{\thead{H\"older smooth, parameter $\nu\in[0,1]$\\
See \cref{def:holder-smooth22}\\(NB: must have $\beta\geq 1+\nu$)
}}&
\multirow{3}{*}{$\beta=1\!+\!\nu$:\!\!\!\!}&
\multirow{3}{*}{$\frac{M_\nu^{\frac{2}{1+3\nu}}}{\alpha^{\frac{2}{(1+3\nu)}}}\cdot\log(1/ \varepsilon)$}&
\multirow{6}{*}{\thead{Universal fast\\gradient method\\$d_1=(2+2\nu)/(1+3\nu)$\\$d_2=2/(1+3\nu)$\\
See \cref{sec:FGM}}}\\
&&&\\
&&&\\\cdashline{2-3}
&\multirow{3}{*}{$\beta>1\!+\!\nu$:\!\!\!\!}&
\multirow{3}{*}{$\frac{M_\nu^{\frac{2}{1+3\nu}}}{\alpha^{\frac{2+2\nu}{\beta_*(1+3\nu)}}}\cdot\frac{1}{\varepsilon^{\frac{2(\beta_*-1-\nu)}{\beta_*(1+3\nu)}}}$}&\\
&&&\\
&&&\\\hline
\multirow{4}{*}{\thead{$f(x)\!=\!q(x)\! +\! g(x)\! +\! h(Bx)$, $q$ is $L_q-$smooth,\\$\sup_{z\in \mathrm{dom}(h)}\inf_{y\in\partial h(z)}\nm{y}\leq L_h$,\\$\|B\|\leq L_B$}}&
\multirow{2}{*}{$\beta=1$:\!\!\!\!}&
\multirow{2}{*}{$\frac{L_BL_h+L_q}{\alpha}\cdot{\log(1/ \varepsilon)}$}&
\multirow{4}{*}{\thead{Primal-dual algorithm\\$d_1=1$, $d_2=1$\\See \cref{sec:PD_unconstrained_desc}}}\\
&&&\\\cdashline{2-3}
&\multirow{2}{*}{$\beta>1$:\!\!\!\!}&
\multirow{2}{*}{$\frac{L_BL_h+L_q}{\alpha^{1/\beta_*}}\cdot\frac{1}{\varepsilon^{1-1/\beta_*}}$}&\\
&&&\\\hline
\multirow{4}{*}{\thead{$f(x)\!=\! q(x)\! +\! g(x)\! +\! h(Bx)$, $q$ is $L_q-$smooth,\\
$\sup_{z\in \mathrm{dom}(h)}\inf_{y\in\partial h(z)}\nm{y}\leq L_h$,\\$\|A\|\leq L_A$, $\|B\|\leq L_B$,\\
$Q\!=\!\{x:Ax\in C\},$ $g_Q(x)\!=\!\kappa\inf_{z\in C} \!\nm{Ax-z}$}}&
\multirow{2}{*}{$\beta=1$:\!\!\!\!}&
\multirow{2}{*}{$\frac{\kappa L_A+L_BL_h+L_q}{\alpha}\cdot{\log(1/ \varepsilon)}$}&
\multirow{4}{*}{\thead{Primal-dual algorithm\\with constraints\\$d_1=1$, $d_2=1$\\See \cref{sec:PD_constrained_alg}}}\\
&&&\\\cdashline{2-3}
&\multirow{2}{*}{$\beta>1$:\!\!\!\!}&
\multirow{2}{*}{$\frac{\kappa L_A+L_BL_h+L_q}{\alpha^{1/\beta_*}}\cdot\frac{1}{\varepsilon^{1-1/\beta_*}}$}&\\
&&&\\\hline
\end{tabular}
\caption{Asymptotic cost bounds (as $\varepsilon\downarrow0$ for $\eta\lesssim\varepsilon$) and suitable first-order methods for Algorithm 2 when applied to different classes of objective functions. Note that whenever the bound is a polynomial in $\log(1/\varepsilon)$, we have $\beta_*=\beta$.\label{rates_table}}
\end{table}

\subsection{Complexity bounds}\label{sec:complexity-bounds}

Suppose now that $\eta \lesssim \varepsilon$. Note that this case includes $\eta = 0$, in which case sharpness holds. When \cref{alg:restart-unknown-constsC} is applied with a suitable first-order method, it leads to near-optimal\footnote{By optimal, we mean optimal in the number of oracle calls to $f$, its gradient (where appropriate) or suitable proximal maps. For the first-order methods we discuss, this number will always be bounded by a small multiple of the number of iterations.} complexity bounds for a wide range of different convex optimization problems, \textit{without knowledge of $\alpha$ and $\beta$}. \cref{rates_table} summarizes some of these bounds, and the following correspond to an example for each row:

\begin{itemize}
	\item For $L$-smooth functions (\cref{L_smooth}) with $\beta = 2$, a well-known lower bound for the subclass of strongly convex smooth functions is $\mathcal{O}(\sqrt{L/\alpha}\log(1/\varepsilon))$ \cite{nemirovskij1983problem}. If $\beta> 2$ then the optimal lower bound is $\mathcal{O}(\sqrt{L}\alpha^{-1/\beta}/\varepsilon^{1/2-1/\beta})$ \cite[page 26]{nemirovskii1985optimal}. In both cases, we achieve these optimal bounds with our algorithm (provided $\beta_* = \beta$; see below) using, for example, Nesterov's method.
	\item Suppose that the objective function $f$ is $L_f$-Lipschitz and has linear growth. Such functions are $(1,L_f^2/2)$-smoothable (\cref{def:uv-smoothable}). When $\beta = 1$, the combination of our algorithm and Nesterov's method with smoothing has complexity $\ord{\log(1/\varepsilon)}$.
	\item For H\"older smooth functions (see \cref{def:holder-smooth22}), the bound in \cref{rates_table} matches (again, provided $\beta_* = \beta$) the optimal bound from \cite[page 26]{nemirovskii1985optimal}.
	\item There is not much work on optimal rates for saddle point problems, a challenge being that there are different measures of error (see \cite{optimal_saddle}). Hence, we cannot claim that the final two rows of \cref{rates_table} yield optimal rates. Nevertheless, they yield significantly faster convergence rates than non-restarted first-order methods for saddle point problems.
\end{itemize}

The above optimality depends on $\beta=\beta_*$ (which occurs whenever $\beta$ lies on the grid). If this does not hold, then there is an additional algebraic factor in \cref{rates_table}, making it slightly suboptimal. This can be overcome by a choice of $b$ that depends logarithmically on $\varepsilon$ at the expense of an additional logarithmic term. This point and its relation to other methods is discussed in \cref{sec:getridofalgebraic}.

\subsection{Connections with previous work}
\label{sec:prev_work}

There is a large amount of recent work on adaptive first-order methods \cite{o2015adaptive,freund2018new,necoara2019linear,fercoq2016restarting,su2014differential,giselsson2014monotonicity,fercoq2019adaptive}. Adaptive methods seek to learn when to restart a first-order method by trying various values for the method's parameters and observing consequences over several iterations. Nesterov provided a catalyst for this body of work in \cite{nesterov2013gradient}, where he designed an accelerated (line search) method for $L$-smooth objective functions $f$ (see \cref{sec:alg_NESTA_smoothing}) with an optimal convergence rate $\mathcal{O}(\sqrt{L/\varepsilon})$ without needing $L$ as an input. In the same paper, Nesterov considered strongly convex objective functions with a grid search to approximate the strong convexity parameter. By narrowing the class of objective functions, this led to an adaptive method with a dramatically improved convergence rate ($\mathcal{O}(\log(1/\varepsilon))$, compared to $\mathcal{O}(1/\sqrt{\varepsilon})$), even without having to know the Lipschitz constant or strong convexity parameter.

The complexity of first-order methods is usually controlled by smoothness assumptions on the objective function, such as Lipschitz continuity of its gradient. Additional assumptions on the objective function, such as strong and uniform convexity, provide respectively, linear and faster polynomial rates of convergence \cite{nesterov2003introductory}. Restart schemes for strongly convex or uniformly convex functions have been studied in \cite{nemirovskii1985optimal,nesterov2013gradient,iouditski2014primal,lin2014adaptive}. However, strong or uniform convexity is often too restrictive an assumption in many applications.

An assumption more general than strong or uniform convexity is sharpness:
\begin{equation}
d(x, \widehat{X}) \leq \left( \frac{f(x) - \hat{f}}{\alpha} \right)^{1/\beta} \qquad \forall x \in Q, \label{loja}
\end{equation}
also known as a H\"olderian growth/error bound or a \L{}ojasiewicz-type inequality. For example, Nemirovskii and Nesterov \cite{nemirovskii1985optimal} linked a ``strict minimum'' condition similar to \eqref{loja} (with known constants) with faster convergence rates using restart schemes for smooth objective functions. For further use of \L{}ojasiewicz-type inequalities for first-order methods, see \cite{bolte2014proximal,bolte2017error,attouch2010proximal,frankel2015splitting,karimi2016linear}. H\"olderian error bounds were first introduced by Hoffman \cite{hoffman1952approximate} to study systems of linear inequalities, and extended to convex optimization in \cite{robinson1975application,mangasarian1985condition,auslender1988global,burke1993weak,burke2002weak}. \L{}ojasiewicz showed that \eqref{loja} holds generically for real analytic and subanalytic functions \cite{lojasiewicz1963propriete}, and Bolte, Daniilidis, and Lewis extended this result to non-smooth subanalytic convex functions \cite{bolte2007lojasiewicz}. However, the proofs of these results use non-constructive topological arguments. Hence, without further case-by-case analysis of problems and outside of some particular cases (e.g., strong convexity), we cannot assume that suitable constants in \eqref{loja} are known.

An example of \eqref{loja} for $\beta=1$ was considered in \cite{roulet2020computational} (see also \cite{benlectures}), where the authors use a restarted NESTA algorithm \cite{becker2011nesta} for the exact recovery of sparse vectors from noiseless measurements. The approximate sharpness condition \eqref{eqn:sharpness} was first considered in \cite{colbrook2021warpd} for the case of $\beta=1$, and known $\alpha$ and $\eta$, to allow the recovery of approximately sparse vectors from noisy measurements and further related examples. Here, the parameter $\eta>0$ is crucial, both in practice and to allow analysis. See also \cite{adcock2022efficient,neyra-nesterenko2022nestanets}. Though similar to the sharpness condition in \eqref{loja}, our more general assumption in \eqref{eqn:sharpness} differs in two essential ways, discussed above. First, we do not assume that the sharpness condition is exact ($\eta>0$), and second, we do not require iterates of our algorithm to be feasible (the function $g_Q$). It is also important to re-emphasize that, in this paper, we do not assume that the approximate sharpness constants are known.

The $\eta$ term in \eqref{eqn:sharpness} is expected and natural, for example, in the sparse recovery example of \cref{sec:sparse-recovery}. In \cite{colbrook2022difficulty}, it was shown that there are well-conditioned recovery problems for which stable and accurate neural networks exist, but no training algorithm can obtain them. The existence of a training algorithm depends on the amount/type of training data and the accuracy required. However, under certain conditions, one can train an appropriate neural network: \cite{colbrook2022difficulty} links trainability to a particular case of \eqref{eqn:sharpness}, and links the accuracy possible via training to the corresponding $\eta$ term. In the setting of inexact input, the noise parameter appears as a limitation on the ability of an algorithm \cite{opt_big}. These phenomena occur even if the algorithm is only expected to work on a restricted class of inputs that are `nice' or `natural' for the problem under consideration. The results of \cite{colbrook2022difficulty,opt_big} lead to the phenomenon of generalized hardness of approximation, where it is possible to obtain solutions up to some threshold, but beyond that threshold it becomes impossible. This threshold is strongly related to $\eta$ in the standard cases.

Most restart schemes are designed for a narrow family of first-order methods and typically rely on learning approximations of the parameter values characterizing functions in a particular class, e.g., learning the Lipschitz constant $L$ when $f$ is assumed to be $L$-smooth, or the constants $\alpha$ and $\beta$ in \eqref{loja} (e.g., see the discussion in \cite{renegar2021simple}). Two notable exceptions related to the present paper particularly inspired us.

First, Roulet and d'Aspremont \cite{roulet2020sharpness} consider all $f$ possessing sharpness and having H\"older continuous gradient with exponent $0 <\nu\leq 1$. The restart schemes of \cite{roulet2020sharpness} result in optimal complexity bounds when particular algorithms are employed in the schemes, assuming scheme parameters are set to appropriate values that, however, are generally unknown in practice. The limitations of a grid search for these parameters are explicitly discussed in Appendix C of \cite{roulet2020sharpness} due to an additional algebraic factor, similar in spirit to the case of $\beta\neq\beta_*$ in \cref{rates_table}. A critical difference between our scheme and that of \cite{roulet2020sharpness} is our use of a schedule criterion function (see \cref{sec:sch_defs}). For example, this flexibility leads us to overcome these algebraic losses, as outlined in \cref{sec:getridofalgebraic}, up to a logarithmic factor. However, for smooth $f$ (i.e., $\nu=1$), \cite{roulet2020sharpness} develops an adaptive grid search procedure for a certain condition number and sharpness parameters within the scheme to accurately approximate the required values, leading to an overall complexity that is optimal up to a squared logarithmic term. Some other differences between our scheme and that of \cite{roulet2020sharpness} is that we deal with approximate sharpness ($\eta>0$), allow infeasible iterates (captured by $g_Q$), and consider general classes of functions, some of which are listed in \cref{rates_table}.

Second, Renegar and Grimmer \cite{renegar2021simple} provide a simple scheme for restarting (generic) first-order methods. Multiple instances are run that communicate their improvements in objective value to one another, possibly triggering restarts. Their restart scheme only depends on how much the objective value has been decreased and does not attempt to learn parameter values. The scheme in \cite{renegar2021simple} leads to nearly optimal complexity bounds for quite general classes of functions. This method differs quite significantly from ours in that it does not assume an underlying sharpness condition \eqref{loja} (although such a condition is used in the analysis to obtain explicit complexity bounds). In contrast to \cite{renegar2021simple}, our method is independent of the total number of iterations, and we do not need to specify the total number of iterations in advance. Further, we also address the practical case of approximate sharpness and allow the case of infeasible iterates (the convergence analysis of \cite{renegar2021simple} relies on $\eta=0$ and that iterates are feasible). In \cref{sec:comparison_sec}, we compare our scheme to that of Renegar and Grimmer.

Restart schemes can also be employed for non-deterministic algorithms. For example, Fercoq and Qu \cite{fercoq2016restarting} study an objective function that can be written as a sum of differentiable and separable functions under the assumption of a local quadratic bound ($\beta=2$ in our notation). They introduce a sequence of variable restart periods that allows several restart periods to be tried. This allows them to force a decrease in function value when a restart occurs. This leads to a scheme with a nearly linear rate, even without knowledge of the local quadratic error bound.

\begin{table}[t]\label{notation_table}
\centering
\begin{tabular}{|l|l|}
\hline
Notation& Meaning\\
\hline
\hline
$f$ & Proper convex function \\
$D$ & Effective domain of $f$ \\
$Q$ & Closed, convex subset of $\bbR^n$ or $\bbC^n$\\
$g_Q$ & Sharpness feasibility gap function, identically zero on $Q$\\
$\hat{f}$ & Minimum value of objective function over $Q$\\
$\widehat{X}$ & Set of minimizers of $f$ \\
$d$ & Metric on $\bbR^n$ or $\bbC^n$\\
$\eta$ & Sharpness gap constant \\
$\alpha$ & Sharpness scaling constant \\
$\beta$ & Sharpness exponentiation constant \\
$\delta$ & Distance bound between initial point to optimum points \\
$\varepsilon$ & Bound on sum of objective function error and feasibility gap\\
$\epsilon_j$ & Sum of objective function error and feasibility gap at $j$th restart initial point \\
$\Gamma$ & Optimization algorithm \\
$\cC_\Gamma$ & Cost function that outputs the number of iterates \\
$\phi$ & Mapping of current algorithm step to parameter subscripts $(i,j,k)$ \\
$h$ & Function defining classes of maps $\phi$ as abstract execution order of restart scheme \\
$\chi_C$ & Indicator function of a set $C$ ($\chi_C(x)=0$ if $x\in C$, $\chi_C(x)=\infty$ otherwise) \\
$\nm{\cdot}$ & Unless otherwise stated, the Euclidean norm on $\bbC^n$ or the induced $2$-norm on $\bbC^{m \times n}$ \\
$\ip{\cdot}{\cdot}$ & Unless otherwise stated, the Euclidean inner product on $\bbC^n$ \\
$\ip{\cdot}{\cdot}_{\bbR}$ & Unless otherwise stated, $\ip{x}{y}_{\bbR} = \mathrm{Re}\left(\ip{x}{y}\right)$ for $x,y \in \bbC^n$ \\
$\bbR_+$ & Non-negative real numbers \\
$\bbR_{++}$ & Positive real numbers \\
$\bbN_{0}$ & Non-negative integers $\{0\}\cup\mathbb{N}$ \\
\hline
\end{tabular}
\caption{Notation used throughout the paper.}
\end{table}

\subsection{Notation and outline}\label{blah}

Table 2 outlines the notation used throughout the paper for ease of reference. The remainder of this paper is organized as follows. In \cref{sec:restart-known-constants}, we introduce a restart scheme where $\eta$ is unknown, but $\alpha$ and $\beta$ are known. This transpires to be significantly more straightforward than the general case. Next, in \cref{sec:restart-unknown-constants}, we introduce and analyze the complete restart scheme when all three constants are potentially unknown. In \cref{sec:FOM_examples}, we apply this restart scheme to different problems with various first-order methods, leading, in particular, to the results described in \cref{rates_table}. Next, in \cref{sec:num-exp}, we present a series of numerical experiments illustrating the restart schemes in different applications. Finally, we end in \cref{sec:conclusion} with conclusions and open problems.

\section{Restart scheme for unknown \texorpdfstring{$\eta$ but known $\alpha$ and $\beta$}{Lg}}
\label{sec:restart-known-constants}

To formulate a restart scheme within the setup of \cref{sec:prob_setup_intro}, observe that the approximate sharpness condition \eqref{eqn:sharpness} relates the distance $d(x,\widehat{X})$ to the objective function error $f(x) - \hat{f}$ and feasibility gap $g_Q(x)$. The upper bound in the approximate sharpness condition can be used as the input $\delta$ for the algorithm $\Gamma$. At the same time, $\epsilon$ is set as a rescaling of the previous sum of objective error and feasibility gap $f(x) - \hat{f} + g_Q(x)$ with rescaling parameter $r\in(0,1)$. However, in practical scenarios, the exact values of the objective error $f(x) - \hat{f}$ and feasibility gap $g_Q(x)$ might not be known. Instead, it is sufficient to have upper bounds for these quantities.

\begin{algorithm}[t]
\SetKwInOut{Input}{Input} 
\SetKwInOut{Output}{Output}
\SetKwComment{Comment}{// }{}
\Input{Optimization algorithm $\Gamma$ for \eqref{eqn:cvx-problem}, initial vector $x_0 \in D$, upper bound $\epsilon_0$ such that $f(x_0) - \hat{f} + g_Q(x_0) \leq \epsilon_0$, constants $\alpha > 0$ and $\beta \geq 1$ such that \eqref{eqn:sharpness} holds (for possibly unknown $\eta \geq 0$), $r \in (0,1)$, and number of restart iterations $t \in \bbN$.}
\Output{Final iterate $x_t$ approximating a solution to \eqref{eqn:cvx-problem}}
\For{$k = 0, 1, \dots, t-1$}{
$\epsilon_{k+1} \gets r \epsilon_k$ \;
$\delta_{k+1} \gets \left(\frac{2\epsilon_k}{\alpha} \right)^{1/\beta}$ \;
$z \gets \Gamma \left( \delta_{k+1} , \epsilon_{k+1}, x_k \right)$\;
$x_{k+1} \gets \argmin{} \left\{ f(x) + g_Q(x) : x = x_k \mbox{ or } x = z \right\}$\;
}
\caption{Restart scheme for unknown $\eta$.\label{alg:restart-known-consts}}
\end{algorithm}

As a warmup, we first consider the scenario where the constants $\alpha$ and $\beta$ are known, but $\eta$ remains unknown. The restart scheme for this case is outlined in \cref{alg:restart-known-consts} and is similar in spirit to other known-constant restart schemes but with the added accommodation of any general $\eta > 0$. For instance, in \cite[Sec. 2]{roulet2020sharpness}, the authors consider a restart scheme for H\"older smooth functions (with $\nu=1$), which reduces the objective function by a specific factor for functions satisfying \eqref{eqn:sharpness} with $\eta=0$. The more straightforward case with known $\alpha$ and $\beta$ in \cref{alg:restart-known-consts} provides insights into solving the more comprehensive problem addressed in \cref{sec:restart-unknown-constants}. A significant distinction of our method in \cref{sec:restart-unknown-constants} from previous approaches is our use of a general search method over the parameters $\alpha,\beta$, with a flexible choice of base. This method can be applied in any situation described by \eqref{eqn:sharpness} and is compatible with any first-order method that meets the conditions specified in \cref{sec:prob_setup_intro}.

Using the approximate sharpness condition \eqref{eqn:sharpness} and the algorithm $\Gamma$, we see inductively that for any $t$ with $\epsilon_t\geq \eta$, \cref{alg:restart-known-consts} produces iterates $x_0, x_1, \ldots, x_t \in D$ satisfying the following two bounds:
\begin{equation}
\label{large_error_regime}
\begin{split}
&f(x_k) - \hat{f} + g_Q(x_k)\leq \epsilon_k, \\
&d(x_k, \widehat{X}) \leq \left(\frac{f(x_k) - \hat{f} + g_Q(x_k) + \eta}{\alpha} \right)^{1/\beta} \leq \left(\frac{\epsilon_k + \eta}{\alpha} \right)^{1/\beta} \leq \left(\frac{2\epsilon_k}{\alpha} \right)^{1/\beta}, \qquad 0 \leq k \leq t.
\end{split}
\end{equation}
In addition, the total number of inner iterations used in \cref{alg:restart-known-consts} is at most
\eas{
\sum_{k=0}^{t-1} \cC_\Gamma \left( \left( \frac{2\epsilon_{k}}{\alpha} \right)^{1/\beta} , \epsilon_{k+1} \right).
}
Under further assumptions about the function $\cC_\Gamma$, the iterates produced by the restart scheme yield linear (if $d_2=d_1\beta$) or fast algebraic (if $d_2>d_1\beta$) decay of $f(x_k) - \hat{f} + g_Q(x_k)$ in $k$ down to a finite tolerance proportional to $\eta$. Hence, this property holds for both the objective error $f(x_k) - \hat{f}$ and feasibility gap $g_Q(x_k)$. We state and prove this in the following theorem.

\begin{theorem} \label{thm:restart-known-consts}
Consider \cref{alg:restart-known-consts} and its corresponding inputs. For any $\varepsilon \in(0,\epsilon_0)$, if we run \cref{alg:restart-known-consts} with $t \geq \lceil \log(\epsilon_0 / \varepsilon) / \log(1/r)\rceil$, then 
\ea{
f(x_t) - \hat{f} + g_Q(x_t) \leq \max \{ \eta , \varepsilon \}. \label{eqn:known-consts-eps-error}
}
Suppose, in addition, that for all $\delta, \epsilon > 0$, $\cC_\Gamma$ satisfies
$$
\cC_\Gamma(\delta, \epsilon) \leq {C \delta^{d_1}}/{\epsilon^{d_2}}+1, \qquad C,d_1,d_2 > 0.
$$
Then the total number of iterations of $\Gamma$ needed to compute an $x_t$ with \R{eqn:known-consts-eps-error} is at most
\begin{equation}
\left\lceil\frac{\log(\epsilon_0 /\varepsilon)}{\log(1/r)} \right\rceil+\frac{C2^{d_1/\beta}}{\alpha^{d_1/\beta} r^{d_2}} \cdot 
\begin{cases}
\frac{1-r^{\lceil{\log(\epsilon_0 /\varepsilon)}/{\log(1/r)} \rceil|d_2-d_1/\beta|}
}{1-r^{|d_2-d_1/\beta|}}\cdot \frac{1}{\epsilon_0^{d_2-d_1/\beta}},\quad &\text{if $d_2<d_1/\beta$},\\
\left\lceil\frac{\log(\epsilon_0 /\varepsilon)}{\log(1/r)} \right\rceil,\quad &\text{if $d_2=d_1/\beta$},\\
\frac{1-r^{\lceil{\log(\epsilon_0 /\varepsilon)}/{\log(1/r)} \rceil|d_2-d_1/\beta|}
}{1-r^{|d_2-d_1/\beta|}}\cdot\frac{1}{\varepsilon^{d_2-d_1/\beta}},\quad &\text{if $d_2>d_1/\beta$}.
\end{cases}\label{eqn:known-consts-total-iters}
\end{equation}
Asymptotically as $\varepsilon\downarrow 0$, these can be written as
$$
\sim \begin{cases}
\log(1/\epsilon),\quad &\text{if $d_2\leq d_1/\beta$},\\
\frac{1}{\varepsilon^{d_2-d_1/\beta}},\quad &\text{if $d_2>d_1/\beta$}.
\end{cases}
$$
\end{theorem}

Note that the cases in \eqref{eqn:known-consts-total-iters} match in the limit $d_2-d_1/\beta\rightarrow 0$.

\begin{proof}[Proof of \cref{thm:restart-known-consts}]
The theorem statement is unchanged if we assume that $\varepsilon\geq\eta$. Hence, we may assume without loss of generality that $\varepsilon\geq\eta$. Let $s=\lceil \log(\epsilon_0 / \varepsilon) / \log(1/r)\rceil$, then $\epsilon_{s-1}=r^{s-1}\epsilon_0\geq \varepsilon \geq \eta$. It follows that we are in the regime where \eqref{large_error_regime} holds, and hence
$$
f(x_{s-1}) - \hat{f} + g_Q(x_{s-1})\leq \epsilon_{s-1},\qquad d(x_{s-1}, \widehat{X}) \leq \left(\frac{2\epsilon_{s-1}}{\alpha} \right)^{1/\beta}.
$$
Then by line 4 of \cref{alg:restart-known-consts} and the choice of $s$, we have
$$
f(z) - \hat{f} + g_Q(z) \leq \epsilon_{s} \leq \varepsilon,\quad z = \Gamma(\delta_{s},\epsilon_s,x_{s-1}).
$$
Due to the $\mathrm{argmin}$ taken in \cref{alg:restart-known-consts},~\eqref{eqn:known-consts-eps-error} follows. The total number of iterations, $T$, needed to reach such an $x_s$ is bounded by
$$
T\leq\sum_{k=0}^{s-1} \cC_\Gamma \left( \left( \frac{2\epsilon_{k}}{\alpha} \right)^{1/\beta} , \epsilon_{k+1} \right)
 \leq s+C\sum_{k=0}^{s-1} \frac{(2\epsilon_{k})^{d_1/\beta}}{ \alpha^{d_1/\beta}\epsilon_{k+1}^{d_2}}
= s+ \frac{C2^{d_1/\beta}}{\alpha^{d_1/\beta} r^{d_2}} \sum_{k=0}^{s-1} \frac{1}{\epsilon_{k}^{d_2 - d_1/\beta}}.
$$
In the case that $d_2=d_1/\beta$, then $\epsilon_{k}^{d_2 - d_1/\beta}=1$ and we obtain
$$
T\leq s+ \frac{C2^{d_1/\beta}}{\alpha^{d_1/\beta} r^{d_2}} s=\left(1+\frac{C2^{d_1/\beta}}{\alpha^{d_1/\beta} r^{d_2}}\right)\left\lceil\frac{\log(\epsilon_0 /\varepsilon)}{\log(1/r)} \right\rceil.
$$
If $d_2\neq d_1/\beta$, we use that $\epsilon_k=r^k\epsilon_0$ and sum the geometric series to obtain
\begin{equation}
\label{new_casess}
T\leq \left\lceil\frac{\log(\epsilon_0 /\varepsilon)}{\log(1/r)} \right\rceil+\frac{C2^{d_1/\beta}}{\alpha^{d_1/\beta} r^{d_2}}\frac{1-r^{\lceil{\log(\epsilon_0 /\varepsilon)}/{\log(1/r)} \rceil(d_1/\beta-d_2)}
}{1-r^{d_1/\beta-d_2}}\frac{1}{\epsilon_0^{d_2-d_1/\beta}}.
\end{equation}
If $d_2>d_1/\beta$, then since $\epsilon_0 \geq \varepsilon/r^{s-1}$, we have $\epsilon_0^{d_2-d_1/\beta}\geq \varepsilon^{d_2-d_1/\beta}r^{d_2-d_1/\beta}/r^{s(d_2-d_1/\beta)}$. Substituting this into \eqref{new_casess} and rearranging yields
$$
T\leq \left\lceil\frac{\log(\epsilon_0 /\varepsilon)}{\log(1/r)} \right\rceil+ \frac{C 2^{d_1/\beta}}{\alpha^{d_1/\beta} r^{d_2}}   \frac{1-r^{\lceil{\log(\epsilon_0 /\varepsilon)}/{\log(1/r)} \rceil(d_2-d_1/\beta)}
}{1-r^{d_2-d_1/\beta}}\frac{1}{\varepsilon^{d_2 - d_1/\beta}}.
$$
The result follows by considering the three separate cases in \eqref{eqn:known-consts-total-iters}.
\end{proof}

\section{Restart scheme for unknown $\alpha$, $\beta$ and $\eta$}
\label{sec:restart-unknown-constants}

In the event that the constants $\alpha$, $\beta$ of \eqref{eqn:sharpness} are unknown, we introduce a logarithmic grid search for each of $\alpha$ and $\beta$, running multiple instances of \cref{alg:restart-known-consts}, and aggregating results that minimize the objective error and feasibility gap. Even if suitable \textit{global} $\alpha$ and $\beta$ are known, the following algorithm is useful since it also takes advantage of sharper versions of \eqref{eqn:sharpness} that only hold \textit{locally} around optimal points. Moreover, an objective function can satisfy the approximate sharpness condition in \eqref{eqn:sharpness} for multiple values of $\alpha,\beta$ and $\eta$. In such scenarios, our convergence theorems apply to the optimal values of these constants for any given $\epsilon$.

\subsection{Schedule criterion functions, $h$-assignments and grid searches}
\label{sec:sch_defs}

To introduce the algorithm, suppose that $\alpha$ and $\beta$ are both unknown and let $a,b > 1$ (the \textit{bases}). Our algorithm employs logarithmic search grids for $\alpha$ and $\beta$. Specifically, we consider the values $\alpha_i = a^i\alpha_0$ for $i \in \bbZ$ and $\beta_j = b^j\beta_0$ for $j \in \bbN_0$, where we assume that $\alpha_0,\beta_0$ are additional inputs with $\alpha_0>0$ and $\beta\geq \beta_0\geq 1$. Note that the lower bound $\beta_0$ in the definition of the $\beta_j$ is to capture additional knowledge that may be available (see, e.g., the examples in \cref{sec:FOM_examples}), and may be set to $1$ if no such knowledge is available. Similarly, the constant $\alpha_0$ centers the search grid for $\alpha$ and can be set to $1$ or a scaling factor that captures the magnitude of $f$.

Our algorithm applies the restart scheme described in \cref{alg:restart-known-consts} with the values $\alpha_i$ and $\beta_j$ for each $i$ and $j$. However, it does so according to a particular schedule or order. To capture this order, we make the following definition.

\begin{definition}\label{def:grid-search-bijection}
Consider an infinite subset $S \subseteq \bbZ \times \bbN_0 \times \bbN$. Let $h : \bbR_{+} \times \bbR_+ \times \bbR_{++} \rightarrow \bbR_{++}$ be a function that is non-decreasing in its first and second arguments, and strictly increasing in its third argument. We call such an $h$ a \textit{schedule criterion function}, or simply a \textit{schedule criterion}.
Given a schedule criterion $h$, an \textit{$h$-assignment over $S$} is a bijection $\phi : \bbN \rightarrow S$ satisfying
\ea{
  h(|i'|,j',k') \leq h(|i|,j,k) \quad \iff \quad \phi^{-1}(i',j',k') \leq \phi^{-1}(i,j,k),\label{key_h_prop}
}
for all $(i,j,k), (i',j',k') \in S$.\hfill$\blacktriangle$
\end{definition}

The schedule criterion $h$ and assignment $\phi$ together control the execution order of \cref{alg:restart-known-consts} instances for each triple $(i,j,k)$, where $k\in\bbN$ is a counter, which is an upper bound for the total number of iterations used by the algorithm for the parameter values $(i,j)$. The schedule criterion $h$ weights the importance of the indices $(i,j,k)$, and the assignment $\phi$ allows us to order the triples $(i,j,k)$ according to $h$. In the general case of unknown $\alpha$ and $\beta$, we take 
$$
S = \bbZ \times \bbN_0 \times \bbN,\qquad \text{($\alpha,\beta$ unknown).}
$$
Definition \ref{def:grid-search-bijection} also permits the case where either $\alpha$ or $\beta$ is known. Indeed, suppose that $\beta = \beta_0$ is known, but $\alpha$ is unknown. Then, we define the set $S$ as 
\be{
\label{S-alpha-unknown}
S = \bbZ \times \{0 \} \times \bbN,\qquad \text{($\alpha$ unknown, $\beta$ known)}
}
and let $\alpha_i = a^i \alpha_0$ as before. In this case, we may ignore the second set and consider the schedule criterion and $h$-assignment functions as mappings
\be{
\label{h-assignment-alpha-unknown}
h : \bbR_{+} \times \bbR_{++} \rightarrow \bbR_{++},\quad \phi : \bbN \rightarrow \bbZ \times \bbN,\qquad \text{($\alpha$ unknown, $\beta$ known)}.
}
Similarly, if $\alpha = \alpha_0$ is known and $\beta$ is unknown, then we let 
\be{
\label{S-beta-unknown}
S = \{ 0 \} \times \bbN_0 \times \bbN,\qquad \text{($\alpha$ known, $\beta$ unknown)}
}
and
\be{
\label{h-assignment-beta-unknown}
h : \bbR_{+} \times \bbR_{++} \rightarrow \bbR_{++}, \quad \phi : \bbN \rightarrow \bbN_0 \times \bbN,\qquad \text{($\alpha$ known, $\beta$ unknown)},
}
with $\beta_j = b^j \beta_0$. Finally, if both $\alpha$ and $\beta$ are known, then we set
\be{
\label{S-both-known}
S = \{ 0 \} \times \{ 0 \} \times \bbN,\qquad \text{($\alpha$, $\beta$ known)}
}
and
\be{
\label{h-assignment-known}
h :  \bbR_{++} \rightarrow \bbR_{++}, \quad \phi : \bbN \rightarrow \bbN,\qquad \text{($\alpha$, $\beta$ known)}.
}

For example, \cref{fig:schedule_examples} shows the level curves of the function $h$ defined by
\begin{align*}
h(x_1,x_2,x_3)&=(x_1+1)^2(x_2+1)^2x_3\qquad&&\text{($\alpha$ unknown, $\beta$ unknown)},\\
h(x_2,x_3)&=(x_2+1)^2x_3\qquad&&\text{($\alpha$ known, $\beta$ unknown)},\\
h(x_1,x_3)&=(x_1+1)^2x_3\qquad&&\text{($\alpha$ unknown, $\beta$ known)}.
\end{align*}
These specific choices are motivated by the theoretical results given later in Section \ref{sec:expansion_functions}.
The level curves in the figure describe the search order, and the algorithm performs instances of \cref{alg:restart-known-consts} according to the sublevel set of indices shown by the red dots.

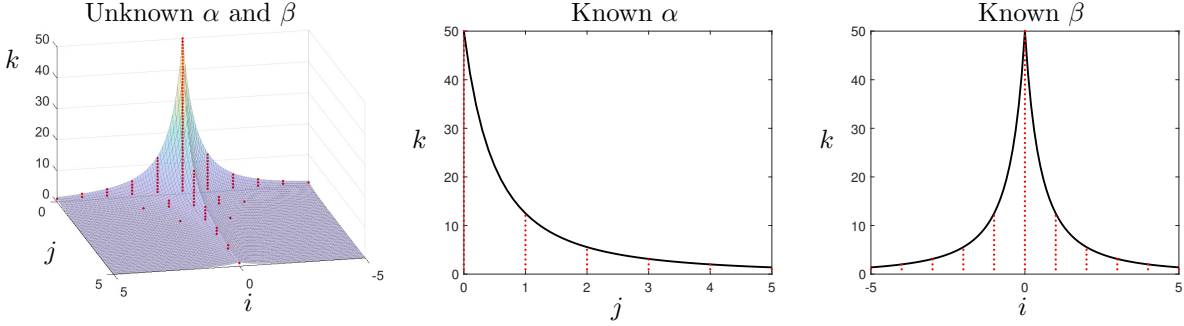
\begin{figure}
\centering
\begin{minipage}[b]{1\textwidth}
\centering
\begin{overpic}[width=0.32\textwidth,trim={0mm 0mm 0mm 0mm},clip]{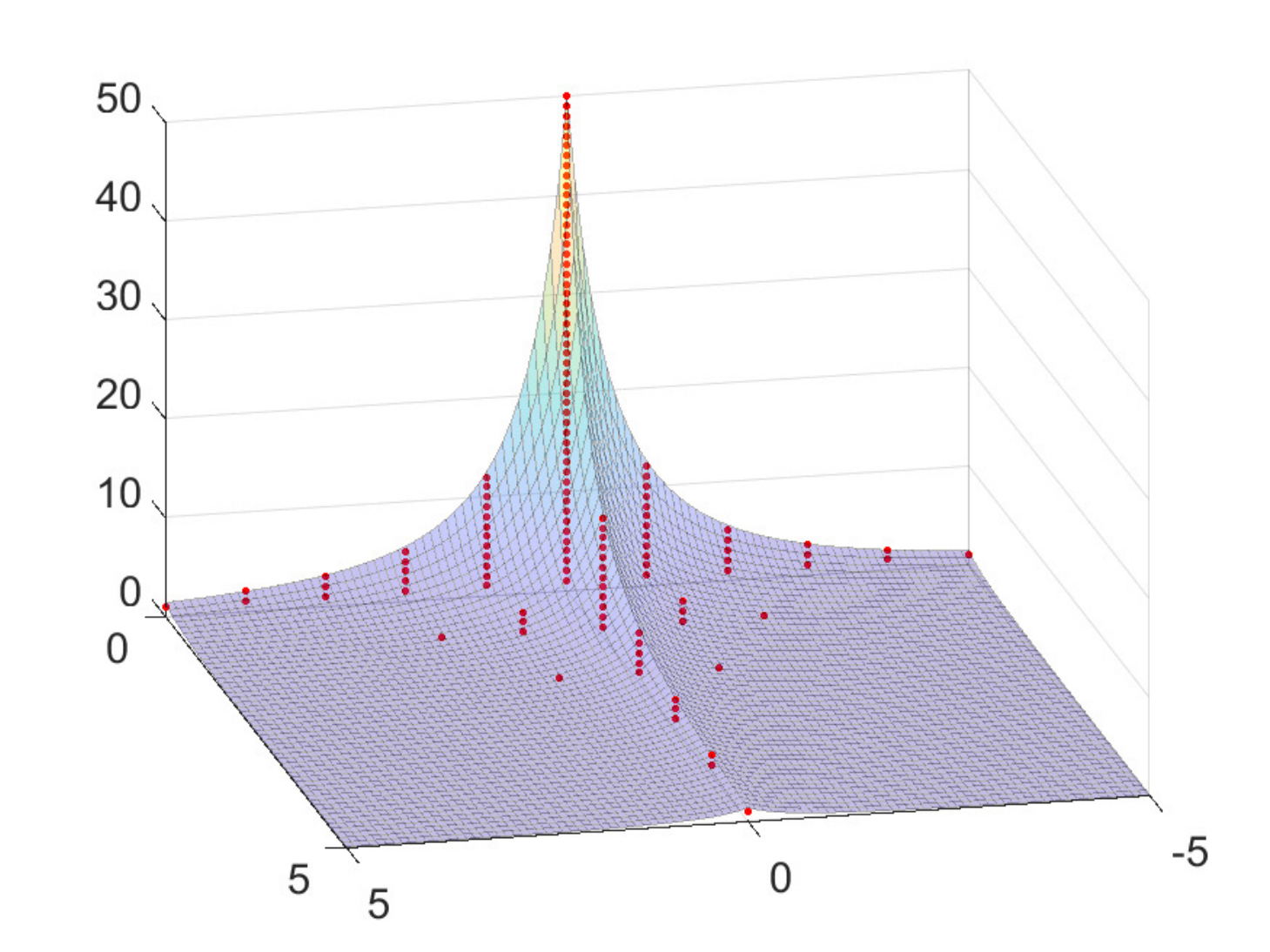}
\put (60,-1) {\small $i$}
\put (10,12) {\small $j$}
\put (0,60) {\small $k$}
\put (20,72) {\small Unknown $\alpha$ and $\beta$}
\end{overpic}
\begin{overpic}[width=0.32\textwidth,trim={0mm 0mm 0mm 0mm},clip]{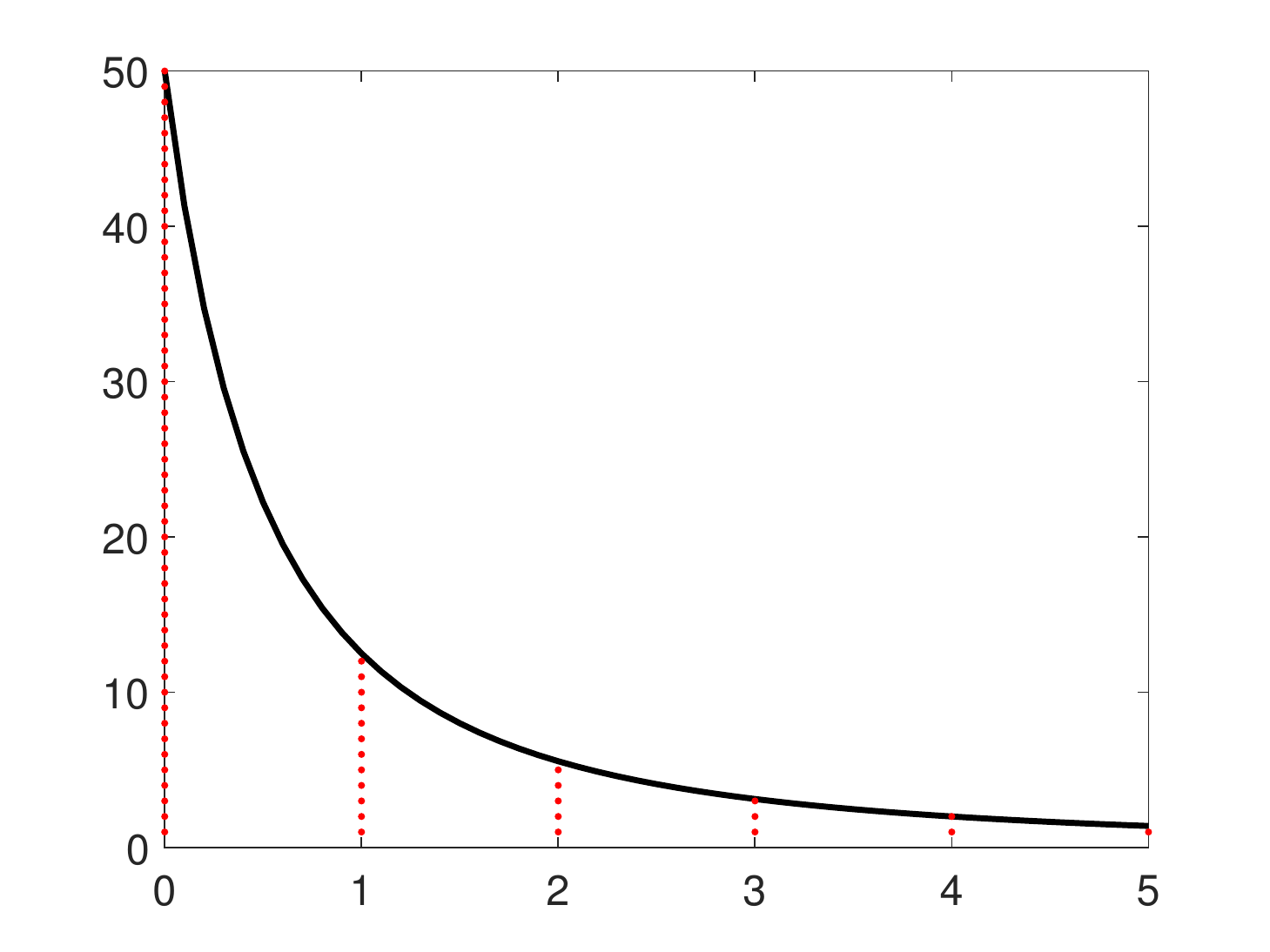}
\put (50,-2) {\small $j$}
\put (0,40) {\small $k$}
\put (40,72) {\small Known $\alpha$}
\end{overpic}
\begin{overpic}[width=0.32\textwidth,trim={0mm 0mm 0mm 0mm},clip]{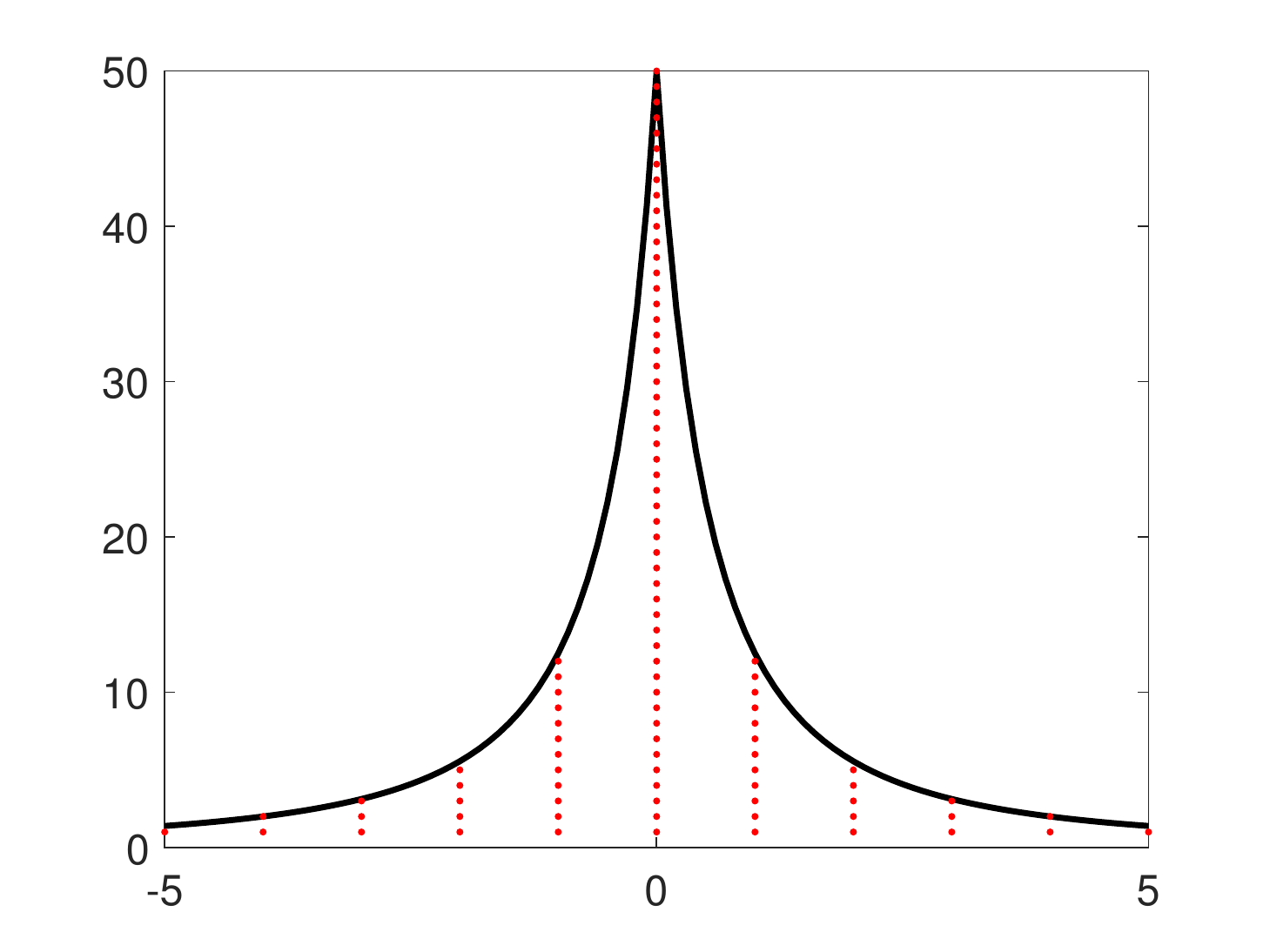}
\put (50,-2) {\small $i$}
\put (0,40) {\small $k$}
\put (40,72) {\small Known $\beta$}
\end{overpic}
\end{minipage}
\caption{Level curves of $h=50$ for the schedule criterion functions $h$ in \cref{rad_ord_cor1} (left panel), \cref{known_alpha} (middle panel) and \cref{known_beta} (right panel) with $c_1=c_2=2$. The level curves describe the search order. The red dots show the corresponding indices $(i,j,k)$ in the set defined in \eqref{eqn:restart-unknown-consts-t-lb}. The index $i$ indicates the parameter search value $a^i\alpha_0$ for $\alpha$. The index $j$ indicates the parameter search value $b^j\beta_0$ for $\beta$. The height (i.e., $k$) indicates the total number of inner iterations for a fixed $(i,j)$.}
\label{fig:schedule_examples}
\end{figure}

\subsection{The algorithm}

\begin{algorithm}[t]
\SetKwInOut{Input}{Input} 
\SetKwInOut{Output}{Output}
\SetKwComment{Comment}{// }{}
\Input{Optimization algorithm $\Gamma$ for \eqref{eqn:cvx-problem}, bijection $\phi$ as in \cref{def:grid-search-bijection}, initial vector $x^{(0)} \in D$, upper bound $\epsilon_0$ such that $f(x^{(0)}) - \hat{f} + g_Q(x^{(0)}) \leq \epsilon_0$, constants $a,b > 1$, $r\in (0,1)$, $\alpha_0>0$, $\beta_0\geq 1$ and total number of inner iterations $t \in \bbN$.}
\Output{Final iterate $x^{(t)}$ approximating a solution to \eqref{eqn:cvx-problem}.}
Initialize $x^{(0)}=x_0$, $U_{i,j} = 0$, $V_{i,j} = 0 $, $\epsilon_{i,j,0} = \epsilon_0$ for all $i\in \bbZ,j\in\mathbb{N}_{0}$\;
\For{$m = 0, 1, \ldots, t-1$}{
$(i,j,k) \gets \phi(m+1)$ \;
$\alpha_i \gets a^i\alpha_0, \ \beta_j \gets b^j\beta_0$, $U \gets U_{i,j}, \ V \gets V_{i,j}$\;
$\epsilon_{i,j,U+1} \gets r\epsilon_{i,j,U}$\;
\uIf{$2\epsilon_{i,j,U}>\alpha_i$}{
$\delta_{i,j,U+1}\gets\left(\frac{2\epsilon_{i,j,U}}{\alpha_i}\right)^{\min\left\{b/\beta_j,1/\beta_0\right\}}$\;
}
\Else{$\delta_{i,j,U+1}\gets\left(\frac{2\epsilon_{i,j,U}}{\alpha_i}\right)^{1/\beta_j}$\;}
\uIf{$V+\cC_\Gamma\left(\delta_{i,j,U+1}, \epsilon_{i,j,U+1} \right)\leq k$}{
$z^{(m)} \gets \Gamma \left(\delta_{i,j,U+1}, \epsilon_{i,j,U+1}, x^{(m)} \right)$\;
$x^{(m+1)} \gets \argmin{}\left\{ f(x) + g_Q(x) : x = z^{(m)} \ \mbox{or} \ x = x^{(m)} \right\}$\;
$V_{i,j}\gets V +\cC_\Gamma\left(\delta_{i,j,U+1}, \epsilon_{i,j,U+1} \right)$\;
$U_{i,j} \gets U+1$\;
}
\Else{
$x^{(m+1)} = x^{(m)}$\;
}
}
\caption{Restart scheme for unknown $\alpha$, $\beta$ and $\eta$ in \eqref{eqn:sharpness} via grid search.\label{alg:restart-unknown-constsC}}
\end{algorithm}

Given a set $S$, our general algorithm is presented in \cref{alg:restart-unknown-constsC}. It proceeds as follows. At step $m \in \{0,\ldots,t-1\}$ it first applies the bijection $\phi$ to obtain the tuple $(i,j,k) = \phi(m+1)$. The first two entries give the approximate sharpness parameter values $\alpha_i = a^i\alpha_0$ and $\beta_j = b^j\beta_0$. The final entry $k$ is a counter, which is an upper bound for the total number of iterations used by the algorithm for these parameter values. We also have two further counters for each double $(i,j)$. The counter $V_{i,j}$ counts the total number of inner iterations of $\Gamma$ used by the restart scheme with these parameters. The second counter $U_{i,j}$ counts the number of completed restarts (outer iterations) corresponding to these parameters.

Having obtained a tuple $(i,j,k) = \phi(m+1)$, the algorithm proceeds as follows. First,  much as in line 2 of \cref{alg:restart-known-consts}, it updates the first scaling parameter in line 5. Then, reminiscent of line 3 of \cref{alg:restart-known-consts}, it updates the other scaling parameter in lines 6-10. This step is more involved, a complication that arises because the true parameter $\beta$ is unknown.

The following lines, lines 11-16, are similar to lines 4-5 of \cref{alg:restart-known-consts}. The main difference is the inclusion of the if statement, which is done to control the computational cost. It stipulates that a restart be performed (line 12) if the total cost (including the proposed restart) does not exceed the counter $k$ (line 11). If this is not the case, no restart is performed, and the algorithm moves on to the next step.

Note that \cref{alg:restart-unknown-constsC} is sequential. However, one can readily devise a parallel implementation that runs Algorithm 1 in parallel over each pair $(i,j)$ and then minimizes $f + g_Q$ over all instances at the end of the process.

Before analyzing the cost of the algorithm, it is worth considering the special cases where either $\alpha$ or $\beta$ is known. Suppose first that $\beta = \beta_0$ is known, but $\alpha$ is unknown and let $S$ and $\phi$ be as in \eqref{S-alpha-unknown}-\eqref{h-assignment-alpha-unknown}. Then, we may eliminate the $j$-index from \cref{alg:restart-unknown-constsC} and write the algorithm more simply as \cref{alg:restart-unknown-alpha}. The first if-else statement from \cref{alg:restart-unknown-constsC} disappears in this case. Similarly, if $\alpha$ is known but $\beta$ is unknown, then we can employ $S$ and $\phi$ as in \eqref{S-beta-unknown}-\eqref{h-assignment-beta-unknown} and eliminate the $i$-index from \cref{alg:restart-unknown-constsC}. We present the result in \cref{alg:restart-unknown-beta}.

\begin{algorithm}[t]
\SetKwInOut{Input}{Input} 
\SetKwInOut{Output}{Output}
\SetKwComment{Comment}{// }{}
\Input{Optimization algorithm $\Gamma$ for \eqref{eqn:cvx-problem}, bijection $\phi$ as in \eqref{h-assignment-alpha-unknown}, initial vector $x^{(0)} \in D$, upper bound $\epsilon_0$ such that $f(x^{(0)}) - \hat{f} + g_Q(x^{(0)}) \leq \epsilon_0$ constants $a> 1$, $r\in (0,1)$, $\alpha_0>0$ and constant $\beta \geq 1$ such that \eqref{eqn:sharpness} holds, and total number of inner iterations $t \in \bbN$.}
\Output{Final iterate $x^{(t)}$ approximating a solution to \eqref{eqn:cvx-problem}.}
Initialize $x^{(0)}=x_0$, $U_{i} = 0$, $V_{i} = 0 $, $\epsilon_{i,0} = \epsilon_0$ for all $i\in \bbZ$\;
\For{$m = 0, 1, \ldots, t-1$}{
$(i,k) \gets \phi(m+1)$ \;
$\alpha_i \gets a^i\alpha_0$, $U \gets U_{i}, \ V \gets V_{i}$\;
$\epsilon_{i,U+1} \gets r\epsilon_{i,U}$\;
$\delta_{i,U+1} \gets\left(\frac{2\epsilon_{i,U}}{\alpha_i}\right)^{1/\beta}$\;
\uIf{$V+\cC_\Gamma\left(\delta_{i,U+1}, \epsilon_{i,U+1} \right)\leq k$}{
$z^{(m)} \gets \Gamma \left(\delta_{i,U+1}, \epsilon_{i,U+1}, x^{(m)} \right)$\;
$x^{(m+1)} \gets \argmin{}\left\{ f(x) + g_Q(x) : x = z^{(m)} \ \mbox{or} \ x = x^{(m)} \right\}$\;
$V_{i}\gets V +\cC_\Gamma\left(\delta_{i,U+1}, \epsilon_{i,U+1} \right)$\;
$U_{i} \gets U+1$\;
}
\Else{
$x^{(m+1)} = x^{(m)}$\;
}
}
\caption{Restart scheme for unknown $\alpha$ and $\eta$, known $\beta$ in \eqref{eqn:sharpness} via grid search. \label{alg:restart-unknown-alpha}}
\end{algorithm}

\begin{algorithm}[t]
\SetKwInOut{Input}{Input} 
\SetKwInOut{Output}{Output}
\SetKwComment{Comment}{// }{}
\Input{Optimization algorithm $\Gamma$ for \eqref{eqn:cvx-problem}, bijection $\phi$ as in \cref{def:grid-search-bijection}, initial vector $x^{(0)} \in D$, upper bound $\epsilon_0$ such that $f(x^{(0)}) - \hat{f} + g_Q(x^{(0)}) \leq \epsilon_0$, constant $\alpha > 0$ such that \eqref{eqn:sharpness} holds, constants $b > 1$, $r\in (0,1)$, $\beta_0\geq 1$ and total number of inner iterations $t \in \bbN$.}
\Output{Final iterate $x^{(t)}$ approximating a solution to \eqref{eqn:cvx-problem}.}
Initialize $x^{(0)}=x_0$, $U_{j} = 0$, $V_{j} = 0 $, $\epsilon_{j,0} = \epsilon_0$ for all $j\in\mathbb{N}_{0}$\;
\For{$m = 0, 1, \ldots, t-1$}{
$(j,k) \gets \phi(m+1)$ \;
$\beta_j \gets b^j\beta_0$, $U \gets U_{j}, \ V \gets V_{j}$\;
$\epsilon_{j,U+1} \gets r\epsilon_{j,U}$\;
\uIf{$2\epsilon_{j,U}>\alpha$}{
$\delta_{j,U+1}\gets\left(\frac{2\epsilon_{j,U}}{\alpha}\right)^{\min\left\{b/\beta_j,1/\beta_0\right\}}$\;
}
\Else{$\delta_{j,U+1}\gets\left(\frac{2\epsilon_{j,U}}{\alpha}\right)^{1/\beta_j}$\;}
\uIf{$V+\cC_\Gamma\left(\delta_{j,U+1}, \epsilon_{j,U+1} \right)\leq k$}{
$z^{(m)} \gets \Gamma \left(\delta_{j,U+1}, \epsilon_{j,U+1}, x^{(m)} \right)$\;
$x^{(m+1)} \gets \argmin{}\left\{ f(x) + g_Q(x) : x = z^{(m)} \ \mbox{or} \ x = x^{(m)} \right\}$\;
$V_{j}\gets V +\cC_\Gamma\left(\delta_{j,U+1}, \epsilon_{j,U+1} \right)$\;
$U_{j} \gets U+1$\;
}
\Else{
$x^{(m+1)} = x^{(m)}$\;
}
}
\caption{Restart scheme for known $\alpha$, unknown $\beta$ and $\eta$ in \eqref{eqn:sharpness} via grid search.\label{alg:restart-unknown-beta}}
\end{algorithm}

\rem{[\cref{alg:restart-unknown-constsC} reduces to \cref{alg:restart-known-consts} when $\alpha$ and $\beta$ are known]

Suppose that $\alpha$ and $\beta$ are both known. We may now eliminate the $i$- and $j$-indices from \cref{alg:restart-unknown-constsC}. Then the main for loop subsequently reduces to

\vspace{2mm}
\RestyleAlgo{plain}
\begin{algorithm*}[H]
\For{$m = 0, 1, \dots , t-1$}{
$k \gets \phi(m+1)$\;
$\epsilon_{U+1} \gets r \epsilon_U$\;
$\delta_{U+1} \gets \left ( \frac{2 \epsilon_U }{\alpha} \right )^{1/\beta}$\;
\uIf{$V+\cC_\Gamma\left(\delta_{U+1}, \epsilon_{U+1} \right)\leq k$}{
$z^{(m)} \gets \Gamma \left(\delta_{U+1}, \epsilon_{U+1}, x^{(m)} \right)$\;
$x^{(m+1)} \gets \argmin{}\left\{ f(x) + g_Q(x) : x = z^{(m)} \ \mbox{or} \ x = x^{(m)} \right\}$\;
$V\gets V +\cC_\Gamma\left(\delta_{U+1}, \epsilon_{U+1} \right)$\;
$U \gets U+1$\;
}
\Else{
$x^{(m+1)} = x^{(m)}$\;
}
}
\end{algorithm*}
\RestyleAlgo{ruled}
\vspace{2mm}

\noindent
We now observe that the algorithm performs a restart whenever the counter $k$ is sufficiently large (i.e., line 5). Thus, up to re-labeling, this is identical to \cref{alg:restart-known-consts}. In particular, the choice of the $h$-assignment $\phi$ does not influence the algorithm in this case.
}

\subsection{Cost analysis of the algorithm}

We now present a general result on this algorithm. It relates the total number of inner iterations of $\Gamma$ used by \cref{alg:restart-unknown-constsC} (to produce a solution within a desired error) to intrinsic properties of the schedule criterion function $h$. This will allow us to derive explicit bounds for specific choices of $h$.

\begin{theorem}\label{thm:restart-unknown-consts-err}
Let $S \subseteq \bbZ \times \bbN_0 \times \bbN$ be an infinite subset, $h$ be a schedule criterion, and $\phi$ an $h$-assignment over $S$. Let $\alpha$, $\beta$ and $\eta$ be approximate sharpness constants of $f$ in \eqref{eqn:sharpness}. Consider \cref{alg:restart-unknown-constsC} for fixed $a, b > 1$. Define the (unknown) indices
$$
I=\lfloor \log_a(\alpha/\alpha_0) \rfloor,\qquad J=\lceil \log_b(\beta/\beta_0) \rceil
$$
and the corresponding constants
$$
\alpha_*=a^I\alpha_0\leq \alpha,\qquad \beta_*=b^J\beta_0\geq\beta.
$$
For $q\in\mathbb{N}$ set
\begin{equation}
\label{delta_form_f}
\delta_{I,J,q}=\left[\max\left\{1,\frac{2r^{q-1}\epsilon_0}{\alpha_*}\right\}\right]^{\min\left\{{b}/{\beta_*},1/\beta_0\right\}}\left[\min\left\{1,\frac{2r^{q-1}\epsilon_0}{\alpha_*}\right\}\right]^{{1}/{\beta_*}}
\end{equation}
Now, for any $\varepsilon \in (0, \epsilon_0)$, let
\be{
\label{L-eps-def}
K(\varepsilon) := K(\varepsilon, \alpha, \beta, \eta) = \sum_{q=1}^{\lceil \log(\epsilon_0 / \varepsilon) / \log(1/r)\rceil}\cC_\Gamma\left(\delta_{I,J,q}, r^{q}\epsilon_0\right)
}
and suppose that $(I,J,K(\varepsilon)) \in S$. Then the total number of inner iterations of $\Gamma$ needed by \cref{alg:restart-unknown-constsC} to compute $x^{(t)}$ with
\eas{
    f(x^{(t)}) - \hat{f} + g_Q(x^{(t)}) \leq \max \{ \eta , \varepsilon \},
} 
is bounded by the cardinality of the set
\begin{equation}
\left\{(i',j',k') \in S : h(|i'|,j',k') \leq h\left( |I|,J, K(\varepsilon) \right) \right\}. \label{eqn:restart-unknown-consts-t-lb}
\end{equation}
In addition, if $\cC_\Gamma$ satisfies
\begin{equation}
\label{growth_cond_c}
\cC_\Gamma(\delta, \epsilon) \leq {C \delta^{d_1}}/{\epsilon^{d_2}}+1, \qquad C,d_1,d_2 > 0,
\end{equation}
for all $\delta, \epsilon > 0$, then
\begin{equation}
\label{L_bound}
\begin{split}
K(\varepsilon)\leq \left\lceil\frac{\log(\epsilon_0 /\varepsilon)}{\log(1/r)} \right\rceil+&\max\left\{\left(\frac{2\epsilon_0}{\alpha_*}\right)^{d_1\min\left\{\frac{b-1}{\beta_*},\frac{1}{\beta_0}-\frac{1}{\beta_*}\right\}},1\right\}\times\\
&\frac{C2^{d_1/\beta_*}}{\alpha_*^{d_1/\beta_*} r^{d_2}} \cdot 
\begin{cases}
\frac{1-r^{\lceil{\log(\epsilon_0 /\varepsilon)}/{\log(1/r)} \rceil|d_2-d_1/\beta_*|}
}{1-r^{|d_2-d_1/\beta_*|}}\cdot \frac{1}{\epsilon_0^{d_2-d_1/\beta_*}},\quad &\text{if $d_2<d_1/\beta_*$},\\
\left\lceil\frac{\log(\epsilon_0 /\varepsilon)}{\log(1/r)} \right\rceil,\quad &\text{if $d_2=d_1/\beta_*$},\\
\frac{1-r^{\lceil{\log(\epsilon_0 /\varepsilon)}/{\log(1/r)} \rceil|d_2-d_1/\beta_*|}
}{1-r^{|d_2-d_1/\beta_*|}}\cdot\frac{1}{\varepsilon^{d_2-d_1/\beta_*}},\quad &\text{if $d_2>d_1/\beta_*$}.
\end{cases}
\end{split}
\end{equation}
\end{theorem}

\begin{proof}
Since $\epsilon_{i,j,q-1}=r^{q-1}\epsilon_0$ for all $q \in\mathbb{N}$, \eqref{delta_form_f} must hold by considering the two separate cases defining $\delta_{I,J,q}$. Similar to the proof of \cref{thm:restart-known-consts}, we may assume without loss of generality that $\varepsilon\geq\eta$. Note that, due to \eqref{eqn:sharpness},
\ea{
d(x, \widehat{X}) \leq \left( \frac{f(x) - \hat{f} + g_Q(x) + \eta}{\alpha_*} \right)^{1/\beta}, \qquad \forall x \in D. \label{eqn:sharpness_unknown_const}
}
Now consider the following adapted version of the iterates in \cref{alg:restart-known-consts}:

\vspace{2mm}
\RestyleAlgo{plain}
\begin{algorithm*}[H]
\For{$p = 0, 1, \dots$}{
$\epsilon_{p+1} \gets r  \epsilon_p $ \;
\uIf{$2\epsilon_p >\alpha_*$}{
$\delta_{p+1}\gets\left(\frac{2\epsilon_p}{\alpha_*}\right)^{\min\left\{b/\beta_*,1/\beta_0\right\}}$ 
\;
}
\Else{$\delta_{p+1}\gets\left(\frac{2\epsilon_p}{\alpha_*}\right)^{1/\beta_*}$\;}
$z \gets \Gamma \left( \delta_{p+1}, \epsilon_{p+1}, x_p \right)$\;
$x_{p+1} \gets \argmin{} \left\{ f(x) + g_Q(x) : x = x_p \mbox{ or } x = z \right\}$\;
}
\end{algorithm*}
\RestyleAlgo{ruled}
\vspace{2mm}

\noindent{}It is easy to see inductively that for any $l$ with $\epsilon_l\geq \eta$ the above produces iterates $\{x_0, x_1, \ldots, x_l\}\subset D$ satisfying
\eas{
&f(x_p) - \hat{f} + g_Q(x_p)\leq \epsilon_p, \qquad d(x_p, \widehat{X}) \leq \delta_{p+1}, \qquad 0 \leq p \leq l.
}
The only difference to the previous argument for \cref{alg:restart-known-consts} is the use of \eqref{eqn:sharpness_unknown_const}, and the fact that
$$
\left( \frac{f(x_p) - \hat{f} + g_Q(x_p) + \eta}{\alpha_*} \right)^{1/\beta}\leq\left( \frac{2\epsilon_p}{\alpha_*} \right)^{1/\beta}\leq\begin{cases}
\left(\frac{2\epsilon_p}{\alpha_*}\right)^{\min\left\{b/\beta_*,1/\beta_0\right\}},\quad &\text{if $2\epsilon_p>\alpha_*$}\\
\left(\frac{2\epsilon_p}{\alpha_*}\right)^{1/\beta_*},\quad &\text{otherwise.}
\end{cases}
$$
Here, we use $\beta\geq \beta_0$ in the first case.

In \cref{alg:restart-unknown-constsC}, each $U_{i,j}$ plays the role of the index $p$ in the above iterates (i.e., counting the number of restarts for a fixed $(i,j)$) and $V_{i,j}$ counts the total number of inner iterations that have been executed by the algorithm $\Gamma$ for the approximate sharpness constants given by the double index $(i,j)$. The fact that we take minimizers of $f+g_Q$ across different indices does not alter the above inductive argument since the argument only depends on bounding the value of $f-\hat{f}+g_Q$. Moreover, since $h$ is strictly increasing in its final argument and satisfies \eqref{key_h_prop}, the counter index $k$ counts successively through $\mathbb{N}$ for any fixed $(i,j)$ as the for loop in \cref{alg:restart-unknown-constsC} proceeds. It follows that if $\phi(m+1)=(I,J,k)$, $V_{I,J} +\cC_\Gamma\left(\delta_{I,J,U_{I,J}+1}, \epsilon_{I,J,U_{I,J}+1}, x^{(m)}\right)\leq k$ and $\epsilon_{I,J,U_{I,J}}\geq \eta$, then
\be{
\label{err-bd-1-IJK}
f(x^{(m+1)})- \hat{f} + g_Q(x^{(m+1)})\leq \epsilon_{I,J,U_{I,J}+1} = r^{U_{I,J}+1} \epsilon_0.
}
Hence, for \cref{alg:restart-unknown-constsC} to produce an iterate with
\be{
\label{err-bd-2-IJK}
    f(x^{(t)}) - \hat{f} + g_Q(x^{(t)}) \leq \max \{ \eta , \varepsilon \},
}
it is sufficient to reach an $m$ with $\phi(m+1)=(I,J,k)$ such that
\begin{equation}
\label{jjnvio}
k\geq\sum_{q=1}^{\lceil \log(\epsilon_0 / \varepsilon) / \log(1/r)\rceil}\cC_\Gamma\left(\delta_{I,J,q}, \epsilon_{I,J,q}\right) = K(\varepsilon)
\end{equation}
and execute the resulting restart. To see why this is the case, notice that if $k$ satisfies this inequality, then the number of restart iterations performed by the algorithm for the parameter values $(I,J)$ is at least $\lceil \log(\epsilon_0 / \varepsilon) / \log(1/r)\rceil$. Plugging this into \eqref{err-bd-1-IJK} gives the desired bound \eqref{err-bd-2-IJK}.

Now consider the set in \eqref{eqn:restart-unknown-consts-t-lb}. Due to \eqref{key_h_prop}, we notice that this set is equivalent to
\bes{
\{ (i',j',k') \in S : \phi^{-1}(i',j',k') \leq m+1 \},
}
where $\phi(m+1) = (I,J,K(\varepsilon))$. Notice that if a triple $(i',j',k')$ belongs to this set, then $(i',j',k'')$ belongs to the set for every $1 \leq k'' \leq k'$. Thus, the number of terms in this set corresponding to the pair $(i',j')$ is precisely the total number of inner iterations performed by the algorithm at the corresponding parameter values up to step $m$. We immediately deduce that the cardinality of the set \eqref{eqn:restart-unknown-consts-t-lb} is a bound for the total number of inner iterations performed by the algorithm across all parameter values up to step $m$, as required.

To finish the proof, we must show that \eqref{L_bound} holds under the additional assumption \eqref{growth_cond_c} on $\cC_\Gamma$. Suppose first that $\delta_{I,J,q}>1$, then
\begin{align*}
\cC_\Gamma\left(\delta_{I,J,q}, r^{q}\epsilon_0\right) &\leq 
C\left(\frac{2r^{q-1}\epsilon_0}{\alpha_*}\right)^{d_1\min\left\{b/\beta_*,1/\beta_0\right\}}\left(r^{q}\epsilon_0\right)^{-d_2}+1\\
&\leq C\left(\frac{2\epsilon_0}{\alpha_*}\right)^{d_1[\min\left\{b/\beta_*,1/\beta_0\right\}-1/\beta_*]}\left(\frac{2r^{q-1}\epsilon_0}{\alpha_*}\right)^{d_1/\beta_*}\left(r^{q}\epsilon_0\right)^{-d_2}+1\\
&=\frac{C}{r^{d_2}}\left(\frac{2\epsilon_0}{\alpha_*}\right)^{d_1[\min\left\{b/\beta_*,1/\beta_0\right\}-1/\beta_*]}\left(\frac{2}{\alpha_*}\right)^{d_1/{\beta_*}}\left(r^{q-1}\epsilon_0\right)^{-d_2+d_1/{\beta_*}}+1.
\end{align*}
Similarly, if $\delta_{I,J,q}\leq 1$, then 
$$
\cC_\Gamma\left(\delta_{I,J,q}, r^{q}\epsilon_0\right)\leq \frac{C}{r^{d_2}}\left(\frac{2}{\alpha_*}\right)^{d_1/{\beta_*}}\left(r^{q-1}\epsilon_0\right)^{-d_2+d_1/{\beta_*}}+1.
$$
From \eqref{jjnvio}, it follows that
$$
K(\varepsilon)\leq\left\lceil\frac{\log(\epsilon_0 /\varepsilon)}{\log(1/r)} \right\rceil+ \max\left\{\left(\frac{2\epsilon_0}{\alpha_*}\right)^{d_1[\min\left\{b/\beta_*,1/\beta_0\right\}-1/\beta_*]},1\right\}\cdot  \frac{C2^{d_1/\beta_*}}{\alpha_*^{d_1/\beta_*} r^{d_2}}\cdot \sum_{k=0}^{\lceil \frac{\log(\epsilon_0 / \varepsilon)}{\log(1/r)}\rceil-1} \frac{1}{(r^k\epsilon_0)^{d_2 - d_1/\beta_*}}.
$$
We now note that the only difference between this bound for $K(\varepsilon)$ and the bound for $T$ in the proof of \cref{thm:restart-known-consts} is the factor that maximizes over the terms in curly brackets and the replacement of $\alpha$ and $\beta$ by $\alpha_*$ and $\beta_*$, respectively. The result follows from the same arguments as in the proof of \cref{thm:restart-known-consts}.
\end{proof}

\subsection{Choices of schedule criterion functions and assignments, and the proof of \cref{thm:MAIN}}
\label{sec:expansion_functions}

As revealed by the previous theorem, the total number of inner iterations of $\Gamma$ needed for \cref{alg:restart-unknown-constsC} depends on the choice of $h$ and $\phi$. We examine some choices and state them as corollaries. These choices correspond to those shown in \cref{fig:schedule_examples}. Combining these corollaries with \cref{thm:restart-unknown-consts-err}, we immediately obtain \cref{thm:MAIN}.

\begin{corollary}[Unknown $\alpha$ and $\beta$]
\label{rad_ord_cor1}
Suppose that $S = \bbZ \times \bbN_0 \times \bbN$ and let 
$$h(x_1,x_2,x_3) = (x_1+1)^{c_1}(x_2+1)^{c_2}x_3,\quad c_1,c_2>1$$
be a schedule criterion with $h$-assignment $\phi$. Then for any $\varepsilon \in (0,\epsilon_0)$, running \cref{alg:restart-unknown-constsC} with
$$
t \geq {2c_1c_2\tau}/{[(c_1-1)(c_2-1)]},\quad
\tau = (| \lfloor \log_a (\alpha/\alpha_0) \rfloor| + 1)^{c_1}(| \lceil \log_b (\beta/\beta_0) \rceil | + 1)^{c_2} K(\varepsilon),
$$
where $K(\varepsilon)$ is as in \R{L-eps-def}, implies that $$f(x^{(t)}) - \hat{f} + g_Q(x^{(t)}) \leq \max \{ \eta , \varepsilon \}.$$ 
\end{corollary}

\begin{proof}
It suffices to prove that the stated lower bound on $t$ is an upper bound for the cardinality of the set \eqref{eqn:restart-unknown-consts-t-lb} from \cref{thm:restart-unknown-consts-err}. We do this by finding an upper bound on the number of solutions to $n_1^{c_1} n_2^{c_2} n_3 \leq \tau$ where $n_1, n_2, n_3 \in \bbN$. 
By directly counting, the number of solutions is bounded by
$$
\sum_{n_1=1}^{\tau^{1/{c_1}}} \sum_{n_2=1}^{\left(\frac{\tau}{{n_1^{c_1}}}\right)^{\frac{1}{{c_2}}}}\frac{\tau}{n_1^{c_1}n_2^{c_2}}  
\leq
\tau \sum_{n_1=1}^\infty \frac{1}{n_1^{c_1}} \sum_{n_2=1}^\infty\frac{1}{n_2^{c_2}}.
$$
We have that
$$
\sum_{n_1=1}^\infty\frac{1}{n_1^{c_1}}\leq 1+\int_1^\infty \frac{dx}{x^{c_1}}=\frac{c_1}{c_1-1}.
$$
It follows that the number of solutions is bounded by $\tau c_1 c_2/((c_1-1)(c_2-1))$. Each counted solution $(n_1,n_2,n_3)$ defines \textit{at most} two tuples $(i',j',k')$ in the set \eqref{eqn:restart-unknown-consts-t-lb}, namely $i' = \pm(n_1 - 1)$, $j' = n_2 - 1$, $k' = n_3$. In reverse, each tuple $(i', j', k')$ of the set \eqref{eqn:restart-unknown-consts-t-lb} is always associated with a single solution $(n_1,n_2,n_3)$, namely $n_1 = |i'| + 1$, $n_2 = j' + 1$, $n_3 = k'$. It then follows that that the set \eqref{eqn:restart-unknown-consts-t-lb} is bounded by $2 \tau c_1 c_2/((c_1-1)(c_2-1))$.
\end{proof}

We now consider the cases where either $\alpha$ or $\beta$ is known.

\begin{corollary}[Known $\alpha$]\label{known_alpha}
Suppose that $\alpha= a^i\alpha_0$. Let
$S = \{i\} \times \bbN_0 \times \bbN$
and $h(x_1,x_2,x_3) = (x_2+1)^{c_2}x_3$, $c_2>1$, be a schedule criterion. Then given any $h$-assignment $\phi$ and any $\varepsilon \in (0, \epsilon_0)$, running \cref{alg:restart-unknown-constsC} with 
$$
t \geq {c_2\tau}/{(c_2-1)}, \quad
\tau = (| \lceil \log_b (\beta/\beta_0) \rceil | + 1)^{c_2} K(\varepsilon),
$$
where $K(\varepsilon)$ is as in \R{L-eps-def}, implies that
$$
f(x^{(t)}) - \hat{f} + g_Q(x^{(t)}) \leq \max \{ \eta , \varepsilon \}.
$$ 
\end{corollary}

\begin{proof}
The result follows after modifying the proof of \cref{rad_ord_cor1}. First, find an upper bound to the number of solutions to $n_2^{c_2} n_3 \leq \tau$ for $n_2, n_3 \in \bbN$. Now find the correspondence between the solutions and the triples $(i',j',k')$ of \eqref{eqn:restart-unknown-consts-t-lb}, where $i'$ is now fixed.
\end{proof}

\begin{corollary}[Known $\beta$]\label{known_beta}
Suppose that $\beta= \beta_0$ is known, $S = \mathbb{Z}\times \{0\} \times \bbN$
and $h(x_1,x_2,x_3) = (x_1+1)^{c_1}x_3$, $c_1 > 1$, is a schedule criterion. Then given any $h$-assignment $\phi$ and any $\varepsilon \in (0, \epsilon_0)$, running \cref{alg:restart-unknown-constsC}
and
$$
t \geq {2c_1\tau}/{(c_1-1)}, \quad
\tau = (| \lfloor \log_a (\alpha/\alpha_0) \rfloor | + 1)^{c_1} K(\varepsilon),
$$
where $K(\varepsilon)$ is as in \R{L-eps-def}, implies that
$$
f(x^{(t)}) - \hat{f} + g_Q(x^{(t)}) \leq \max \{ \eta , \varepsilon \}.
$$
\end{corollary}

\begin{proof}
Similar to the previous proof, the result follows after modifying the proof of \cref{rad_ord_cor1}. First, find an upper bound to the number of solutions to $n_1^{c_1} n_3 \leq \tau$ for $n_1, n_3 \in \bbN$. Now find the correspondence between the solutions and the triples $(i',j',k')$ of \eqref{eqn:restart-unknown-consts-t-lb}, where $j'$ is now fixed.
\end{proof}

To compute an $h$-assignment $\phi$ for \cref{rad_ord_cor1,known_alpha,known_beta}, let us first make some observations in the most general case. From the setup of \cref{def:grid-search-bijection}, consider the equivalence relation $\sim$ over $S$ where $(i_1,j_1,k_1) \sim (i_2,j_2,k_2)$ if and only if $h(\absu{i_i}, j_1, k_1) = h(\absu{i_2}, j_2, k_2)$. The range of $h(\absu{\cdot},\cdot,\cdot)$ over $S$ is countable and has a least element since the arguments of $h$ are non-decreasing by \cref{def:grid-search-bijection}. Therefore, the equivalence classes can be ordered where $S/\sim \ = \{[\bm{g}_1],[\bm{g}_2],[\bm{g}_3],\dots\}$ where for all $i,j$, if $i < j$, then $h(|i_1|,j_1,k_1) < h(|i_2|,j_2,k_2)$ for all $(i_1,j_1,k_1) \in [\bm{g}_i]$, $(i_2,j_2,k_2) \in [\bm{g}_j]$. All $h$-assignments $\phi$ are then determined by this ordering, in the sense that
$$
\phi(n) \in [\bm{g}_m] \iff \sum_{i=1}^{m-1} \# [\bm{g}_i] < n \leq \sum_{i=1}^m \# [\bm{g}_i], \qquad \forall m,n \in \bbN.
$$
We use this observation to compute an assignment function in the context of \cref{rad_ord_cor1}. The procedure is analogous for \cref{known_alpha,known_beta} as they are special cases. As in \cref{rad_ord_cor1}, we have $S = \bbZ \times \bbN_0 \times \bbN$ and
$$
h(x_1,x_2,x_3) = (x_1+1)^{c_1}(x_2+1)^{c_2}x_3,\quad c_1,c_2 >1.
$$
If $c_1, c_2$ are positive integers, as in our numerical experiments, the range of $h(\absu{\cdot}, \cdot, \cdot)$ is precisely $\bbN$, and the equivalence classes can be described by
$$
[\bm{g}_m] = \left\{(i,j,k) \in S : m = h(\absu{i}, j, k) = (\absu{i} + 1)^{c_1} (j + 1)^{c_2} k \right\}, \qquad m \in \bbN.
$$
In this instance, every equivalence class is finite. For any $m$, computing $[\bm{g}_m]$ (and in turn $\phi$) amounts to finding all (finitely many) solutions to the nonlinear equation
$$
m = (\absu{i} + 1)^{c_1} (j + 1)^{c_2} k, \qquad (i,j,k) \in S.
$$
Using the change of variables $y_1 = \absu{i} + 1$, $y_2 = j + 1$, $y_3 = k$, we obtain the Diophantine equation
$$
m = y_1^{c_1}y_2^{c_2}y_3, \qquad y_1, y_2, y_3 \in \bbN.
$$
Our algorithmic approach to find $y_1, y_2, y_3$ is by brute force. Since we know $y_1 \leq \sqrt[c_1]{m}$, $y_2 \leq \sqrt[c_2]{m}$, and $y_3 \leq m$, there are only $m^{\frac{1}{c_1} + \frac{1}{c_2} + 1}$ candidate solutions to check. After finding the solutions, one obtains the elements of $[\bm{g}_m]$ by reverting the change of variables, yielding $i = \pm (y_1 - 1)$, $j = y_2 - 1$, $k = y_3$. The resulting solutions can be listed in arbitrary order, which yields an instance of $\phi$.

Finally, to emphasize the generality of our algorithm, we consider the case where $\alpha$ and $\beta$ are known to lie within explicit ranges. In this case, we modify the set $S$ based on these ranges and choose a schedule criterion function $h(x_1,x_2,x_3)$ depending on $x_3$ only. The following result is immediate.

\begin{corollary}[Known ranges for $\alpha$, $\beta$]\label{known_ranges_search}
Suppose we have integers
$$i_\mathrm{min} \leq i_\mathrm{max}, \quad 0 \leq j_\mathrm{min} \leq j_\mathrm{max},$$
for which 
$$\alpha \in [a^{i_\mathrm{min}}\alpha_0, a^{i_\mathrm{max}}\alpha_0], \quad \beta \in [b^{j_\mathrm{min}}\beta_0, b^{j_\mathrm{max}}\beta_0].$$
Let
$$S = \{i_\mathrm{min}, i_\mathrm{min} + 1, \dots, i_\mathrm{max}\} \times \{j_\mathrm{min}, j_\mathrm{min} + 1, \dots, j_\mathrm{max}\} \times \bbN ,$$
and $h(x_1,x_2,x_3) = x_3$ be a schedule criterion. Then given any $h$-assignment $\phi$ and any $\varepsilon \in (0, \epsilon_0)$, running \cref{alg:restart-unknown-constsC} with
$$t \geq (i_\mathrm{max} - i_\mathrm{min} + 1)(j_\mathrm{max} - j_\mathrm{min} + 1)K(\varepsilon),$$
where $K(\varepsilon)$ is as in \R{L-eps-def}, implies $f(x^{(t)}) - \hat{f} + g_Q(x^{(t)}) \leq \max \{ \eta , \varepsilon \}$.
\end{corollary}

\subsection{Comparison with the cost in~\cref{thm:restart-known-consts}}
\label{sec:getridofalgebraic}

We compare the cost in~\cref{rad_ord_cor1} to that of~\cref{thm:restart-known-consts} under the assumption \R{growth_cond_c}. Let $\hat{K}(\varepsilon)$ be the cost in \eqref{eqn:known-consts-total-iters}. Then
\begin{align}
K(\varepsilon)&\lesssim \hat{K}(\varepsilon)\begin{cases} 1,&\quad \text{if $\beta=\beta_*$ or $d_2\leq d_1/\beta_*$},\\
\frac{1}{\varepsilon^{d_1(1/\beta-1/\beta_*)}},&\quad \text{otherwise}.
\end{cases}
\end{align}
It follows that if $\beta=\beta_*$ or $d_2\leq d_1/\beta_*$, the cost of \cref{alg:restart-unknown-constsC} is of the same order as $\hat{K}(\varepsilon)$. If neither of these hold, then the cost of \cref{alg:restart-unknown-constsC} is of the order of $\varepsilon^{-d_1(1/\beta-1/\beta_*)}$ times the cost of \cref{alg:restart-known-consts}. Note that the order of this extra algebraic dependence can be made arbitrarily small by taking $b$ close to $1$, at the expense of a factor in the term $\tau$ that grows as $\log_b (\beta/\beta_0)^{c_2}$.

One can also remove this extra algebraic factor by letting the base $b$ depend on $\epsilon$ (in which case $\epsilon$ now becomes an input to the restart scheme). Specifically, let $b = 1 + 1/\log(\varepsilon^{-1})$. Corollary \ref{rad_ord_cor1} implies the iteration bound
$$
t \geq {2c_1c_2\tau}/{[(c_1-1)(c_2-1)]},\quad
\tau = (| \lfloor \log_a (\alpha/\alpha_0) \rfloor| + 1)^{c_1}(| \lceil \log_b (\beta/\beta_0) \rceil | + 1)^{c_2} K(\varepsilon),
$$
where $K(\varepsilon)$ is as in \R{L-eps-def}, and in particular, if $\cC_{\Gamma}$ satisfies \R{growth_cond_c} then $K(\varepsilon)$ satisfies \R{L_bound}. We now analyze this bound in the limit $\varepsilon \downarrow 0$. First, recall that
\bes{
\beta \leq \beta_* \leq b \beta.
}
In particular, $\beta_* \rightarrow \beta$ as $\varepsilon\downarrow 0$. Also, observe that
\bes{
\log_b(\beta / \beta_0) \sim \log(\beta/\beta_0) \log(\varepsilon^{-1}),\quad \varepsilon\downarrow0.
}
Suppose now that $d_2 \geq d_1/\beta$. Since we also have $d_2 \geq d_1 / \beta_*$, then we may apply \R{L_bound} to get that
\bes{
K(\varepsilon) \leq \hat{C} \left ( \log(\epsilon_0 / \varepsilon) + (\varepsilon^{-1})^{d_2-d_1/\beta_*} \right ),
}
for all sufficiently small $\varepsilon$, where the constant $\hat{C}$ depends on $C,r,\alpha_0,\alpha,\beta_0,\beta,d_1,d_2$. Since
\bes{
d_2 - d_1 / \beta_* \leq d_2 - d_1 / (b \beta),
}
we obtain
\eas{
(\varepsilon^{-1})^{ d_2-d_1/\beta_*} &\leq (\varepsilon^{-1})^{d_2-d_1/(b \beta) }
\\
& = \exp \left [ \log(\varepsilon^{-1}) \left ( d_2 - \frac{d_1}{(1+1/\log(\varepsilon^{-1})) \beta} \right ) \right ]
\\
& = \exp \left [ \frac{ \log(\varepsilon^{-1})}{1+1/\log(\varepsilon^{-1})} \left ( d_2 - \frac{d_1}{\beta} \right ) + d_2 \frac{1}{1+1/\log(\varepsilon^{-1})} \right ]
\\
& \leq \E^{d_2} \varepsilon^{d_1/\beta-d_2} 
}
for all sufficiently small $\varepsilon$. Hence
\bes{
K(\varepsilon) \leq \hat{C} \left ( \log(\epsilon_0 / \varepsilon) + \varepsilon^{d_1/\beta-d_2} \right ) .
}
Therefore, the total iteration bound satisfies
\bes{
t \geq \tilde{C} \begin{cases} (\log(\varepsilon^{-1}))^{1+c_2} & d_2 = d_1/\beta \\ (\log(\varepsilon^{-1}))^{c_2} \varepsilon^{d_1/\beta-d_2} & d_2 \geq d_1/\beta \end{cases} .
}
for small $\varepsilon$, where $\tilde{C}$ also depends on $c_1,c_2$.

As discussed above and in \cref{sec:complexity-bounds}, the version of our restart scheme with fixed $b$ may miss the optimal algebraic rate when $\beta \neq \beta_*$. The above approach, in which $\varepsilon$ becomes an input and the base $b$ is scaled accordingly, restores the optimal algebraic rates, up to the logarithmic term $(\log(\varepsilon^{-1}))^{c_2}$, where the constant $c_2>1$ can be chosen arbitrarily close to $1$ (recall Corollary \ref{rad_ord_cor1}).

\section{Examples and the complexity bounds of \cref{rates_table}}
\label{sec:FOM_examples}

We now present examples of first-order methods that can be used in our restart scheme for different problem settings, including the methods that lead to the various complexity bounds in \cref{rates_table}. We do this by explicitly deriving an algorithm $\Gamma : \bbR_{++} \times \bbR_{++} \times D \rightarrow D$ that satisfies \R{Gamma-delta-to-eps} and give an explicit bound for the cost function $\cC_{\Gamma}(\delta,\epsilon,x_0)$ of the form $C \delta^{d_1} / \epsilon^{d_2}+1$ for suitable $d_1$ and $d_2$.

This section is organized into five subsections, each corresponding to a row of \cref{rates_table}. In each subsection, we first define the class of functions or problems considered and describe the first-order method considered. We then provide a standard lemma on the convergence of the method before showing in a proposition how to convert this method into $\Gamma$ of the form needed for our restart scheme and deriving a suitable cost function $\cC_{\Gamma}$. The rates in the corresponding line of \cref{rates_table} follow directly from this proposition and \cref{rad_ord_cor1}.

\begin{remark}[Optimization over $\mathbb{C}$]
In convex analysis and continuous optimization, it is standard to consider function inputs lying in a finite-dimensional vector space over $\bbR$. The results described below are extended to $\bbC$ since some of the experiments shown in \cref{sec:num-exp} naturally consider complex numbers. For instance, Magnetic Resonance images are usually complex-valued. To this end, we briefly describe the main facets of optimization over $\bbC$.

We are interested in the domain of $f$ being a subset of $\bbC^n$ while its range is, of course, a subset of $\bbR$. Hence, we consider the natural isomorphism between $\bbC^n$ and $\bbR^{2n}$ given by: if $z = x + \I y \in \bbC^n$ with $x,y \in \bbR^n$, then $z \mapsto (x,y)$. We refer to $z$ as the complex representation and $(x,y)$ as the real representation. Now, one proceeds to do convex analysis and continuous optimization in the real representation, then express the results in the equivalent complex representation. Fortunately, not much needs to change (at least symbolically) when switching between real and complex representations. 

For example, the Euclidean inner products $\ip{\cdot}{\cdot}$ have to be substituted with their real part, i.e., $\ip{\cdot}{\cdot}_{\bbR} := \Re\ip{\cdot}{\cdot}$. Another example pertains to the differentiability of $f$. Specifically, for $x,y \in \bbR^n$, we say that $f$ is differentiable at $z = x + \I y \in D \subseteq \bbC^n$ if and only if $\Re(f)$ is (real) differentiable at $(x,y)$. To define the gradient, denote $\nabla_x$ and $\nabla_y$ as the vector of partial derivatives corresponding to variables $x$ and $y$, respectively. Then $\nabla f := \nabla_x \Re(f) + \I \nabla_y \Re(f)$, noting that because $f$ is real-valued, we have $\Im(f) \equiv 0$. Other parts of convex analysis, such as convexity, functions, proximal mappings, subgradients, also extend to a complex vector domain by applying the definitions to the real representation of complex vectors.\hfill$\blacklozenge$
\end{remark}

\subsection{Row 1 of \cref{rates_table}: Nesterov's method for \texorpdfstring{$L$-smooth functions}{Lg}}
\label{sec:alg_NESTA_smoothing}

For the first row of \cref{rates_table}, we consider Nesterov's method \cite{nesterov2005smooth}, an accelerated projected gradient descent algorithm for general constrained convex optimization problems. Specifically, the algorithm aims to solve \eqref{eqn:cvx-problem} in the special case when $f$ is convex and $L$-smooth:

\begin{definition}
\label{L_smooth}
A function $f : \bbC^n \rightarrow \bbR$ is \textit{$L$-smooth} over $Q \subseteq \bbC^n$ if it is differentiable in an open set containing $Q$, and for all $x$, $y$ in this set, its gradient $\nabla f$ has the Lipschitz property
\begin{align*}
\nmu{\nabla f(x) - \nabla f(y) }_{\ell^2} \leq L \nmu{x - y}_{\ell^2}.\tag*{$\blacktriangle$}
\end{align*}
\end{definition}

Nesterov's method is given in \cref{alg:nesterov}. The algorithm uses the notion of a \textit{prox-function} $p$. Here $p : Q \rightarrow \bbR$ is a proper, closed, and strongly convex function with strong convexity constant $\sigma_p > 0$, that, in addition, satisfies $\min_{x \in Q} p(x) = 0$. Let $x_0 = \mathrm{argmin}_{x \in Q} p(x)$ be the unique minimizer of $p$. To make this dependence explicit, we write $p(\cdot) = p(\cdot;x_0)$. A common and simple choice of prox-function is $p(x;x_0) = \frac{1}{2} \nmu{x - x_0}_{\ell^2}^2$ with $\sigma_p = 1$. This will be useful when we express Nesterov's method \textit{with smoothing}, in terms of $\Gamma$. We now state Nesterov's main result that gives a bound for $f(x_k) - f(x)$ for any $x \in Q$.

\begin{algorithm}[t]
\DontPrintSemicolon
\SetKwInOut{Input}{Input} \SetKwInOut{Output}{Output}
\Input{An $L$-smooth function $f$ and closed, convex set $Q \subseteq \bbC^n$ as in \eqref{eqn:cvx-problem}, prox-function $p(\cdot;x_0)$ with strong convexity constant $\sigma_p$ and unique minimizer $x_0 \in Q$, sequences $\{\gamma_j\}_{j=0}^{\infty}$ and $\{\tau_j\}_{j=0}^{\infty}$, and number of iterations $N$.
}
\Output{The vector $x_N$, which estimates a minimizer of \eqref{eqn:cvx-problem}.}
$z_0\leftarrow x_0$ \\
\For{$j = 0, 1, \ldots, N-1$}{
$x_{j+1} \leftarrow \argmin{x \in Q} \frac{L}{2} \nm{x - z_j}_{\ell^2}^{2} + \ip{\nabla f(z_j)}{x - z_j}_{\bbR}$ \\
$v_j \leftarrow \argmin{x \in Q} \frac{L}{\sigma_p} p(x;x_0) + \sum_{i=0}^{j} \gamma_i \ip{\nabla f(z_i)}{x - z_i}_{\bbR}$ \\
$z_{j+1} \leftarrow \tau_j v_j + (1 - \tau_j) x_{j+1}$
}
\caption{Nesterov's method} \label{alg:nesterov}
\end{algorithm}

\begin{lemma}[Nesterov's theorem]\label{lem:Nesterov's theorem}
Let $Q \subseteq \bbC^n$ be nonempty, closed, and convex, $f$ a convex $L$-smooth function over $Q$. In addition, let $p : Q \rightarrow \bbR$ be a proper, closed, and strongly convex function over $Q$ with strong convexity constant $\sigma_p > 0$ with $\min_{x \in Q} p(x) = 0$. Then \cref{alg:nesterov} with
\bes{
    \gamma_j = \frac{j+1}{2}, \qquad \tau_j = \frac{2}{j+3}, \qquad x_0 = \argmin{x \in Q}\ p(x),
}
generates a sequence $\{x_k\}_{k=1}^\infty \subset Q$ satisfying
\begin{align}
    f(x_k) - f(x) \leq \frac{4 L p(x;x_0)}{k(k+1) \sigma_p}, \qquad \forall x \in Q.
\end{align}
\end{lemma}

\cref{lem:Nesterov's theorem} consists of two modifications of \cite[Theorem 2]{nesterov2005smooth}. First, we do not assume $Q$ is bounded, as the results in the original work do not use this. Second, we allow $x \in Q$ instead of $x \in \widehat{X}$. The proof in the original work does not use the optimality of $x$ and only requires $x$ to be feasible. We utilize this property when considering Nesterov's method with smoothing. The following is now immediate.

\begin{proposition} \label{prop:nesterov-Gamma}
Let $Q \subseteq \bbC^n$ be nonempty, closed and convex, $f$ a convex $L$-smooth function over $Q$ (\cref{L_smooth}).
Given input $(\delta,\epsilon,x_0)\in \bbR_+  \times \bbR_+ \times Q$, let $\Gamma(\delta,\epsilon,x_0)$ be the output of \cref{alg:nesterov} with
$$
p(x;x_0)=\frac{1}{2}\nmu{x-x_0}_{\ell^2}^2,\quad \gamma_j = \frac{j+1}{2}, \quad \tau_j = \frac{2}{j+3}, \quad N=\left\lceil \frac{\delta\sqrt{2L}}{\sqrt{\epsilon} } \right \rceil.
$$
Then \R{Gamma-delta-to-eps} holds with $g_{Q} \equiv 0$. Specifically, 
\begin{equation}
\label{NESTA_fits1}
f(\Gamma(\delta,\epsilon,x_0)) - \hat{f} \leq \epsilon,\quad \forall x_0\in Q\text{ with }d(x_0, \widehat{X}) \leq \delta,
\end{equation}
where $d$ is the metric induced by the $\ell^2$-norm. It follows that we can take
\begin{equation}
\label{NESTA_cost1}
\cC_\Gamma(\delta, \epsilon)= \left\lceil \frac{\delta\sqrt{2L}}{\sqrt{\epsilon} } \right \rceil.
\end{equation}
\end{proposition}

\cref{prop:nesterov-Gamma} shows that we can take $d_1=1$ and $d_2=1/2$ in the cost bound \eqref{growth_cond_c} for Nesterov's method (without smoothing). If $f$ is $L-$smooth and satisfies \eqref{eqn:sharpness} with $\eta=0$, then $\beta\geq 2$. It follows that we can take $\beta_0=2$. \cref{thm:restart-unknown-consts-err} and \cref{rad_ord_cor1} now imply the rates in the first row of \cref{rates_table}. Note that if $L$ is unknown, it is standard to employ line searches.

Several other remarks are in order. First, in Nesterov's method, the iterates $x_j$ are always feasible since the corresponding update step returns a point in $Q$. Thus, in \cref{prop:nesterov-Gamma}, we do not have to define $g_Q$ since $\Gamma$ trivially satisfies \eqref{Gamma-delta-to-eps} with $g_{Q} \equiv 0$. Finally, the requirement $x_0\in Q$ can be relaxed in Nesterov's method. For instance, we only require $f$ is $L$-smooth over the union of $Q$ and an open neighborhood of $x_0$ for some $L > 0$ to start with $x_0 \notin Q$.

\subsection{Row 2 of \cref{rates_table}: Nesterov's method for \texorpdfstring{$(u,v)$-smoothable functions}{Lg}}
\label{Nesta_smooth_sec}

We can extend Nesterov's method to solve \eqref{eqn:cvx-problem} without assuming that $f$ is differentiable. This is done via \textit{smoothing}. For this, we need the following definition from {\cite[Definition 10.43]{beck2017first}} (extended to functions with complex-vector domains).

\begin{definition} \label{def:uv-smoothable}
Let $u, v > 0$. 
A convex function $f : \bbC^n \rightarrow \bbR$ is called \textit{$(u, v)$-smoothable} if for any $\mu > 0$ there exists a convex differentiable function $f_\mu : \bbC^n \rightarrow \bbR$ such that
\begin{enumerate}
\item $f_\mu(x) \leq f(x) \leq f_\mu(x) + v \mu$ for all $x \in \bbC^n$
\item $f_\mu$ is $\frac{u}{\mu}$-smooth over $\bbC^n$
\end{enumerate}
The function $f_\mu$ is referred to as a \textit{$\frac{1}{\mu}$-smooth approximation} of $f$ with parameters $(u, v)$, and $\mu$ is referred to as the \textit{smoothing parameter}.\hfill$\blacktriangle$
\end{definition}

Smoothing is a framework that approximates $f$ arbitrarily closely by a family of smooth functions, i.e., functions with Lipschitz gradients. This means that we can apply Nesterov's method to a smooth approximation of $f$ and also analyze the objective error in terms of $f$. The following provides a modified version of \cref{lem:Nesterov's theorem} for $(u,v)$-smoothable $f$, and is proven in \cref{sec_append_nestarov}.

\begin{lemma}
[Nesterov's theorem for smoothable functions]
\label{lem:nesterovs-thm-smoothable-f}
Let $f : \bbC^n \rightarrow \bbR$ be a convex $(u,v)$-smoothable function. Given any $\mu > 0$, let $f_\mu$ be a $\frac{1}{\mu}$-smooth approximation of $f$ with parameters $(u,v)$. Then taking $Q$, $p$, $\gamma_j$, $\tau_j$, $x_0$ as in \cref{lem:Nesterov's theorem} and applying \cref{alg:nesterov} to the function $f_\mu$ produces a sequence $\{x_k\}_{k=1}^\infty$ satisfying
\ea{f(x_k) - f(x) \leq \frac{4 u p(x;x_0)}{\mu k(k+1) \sigma_p} + v \mu , \qquad x \in Q.}
\end{lemma}

The following proposition shows that Nesterov's method with smoothing can be formulated as an algorithm $\Gamma$ in our framework and is proven in \cref{sec_append_nestarov}.

\begin{proposition} \label{prop:nesterov-smoothing-Gamma}
Let $Q \subseteq \bbC^n$ be nonempty, closed and convex, and $f : \bbC^n \rightarrow \bbR$ a convex $(u,v)$-smoothable function (\cref{def:uv-smoothable}). Given input $(\delta, \epsilon, x_0) \in \bbR_+ \times \bbR_+ \times Q$, let $\Gamma(\delta, \epsilon, x_0)$ be the output of \cref{alg:nesterov} applied to function $f_\mu$ with 
$$\mu = \frac{\epsilon}{2 v}, \quad p(x;x_0) = \frac{1}{2} \nmu{x - x_0}_{\ell^2}^2, \quad \gamma_j = \frac{j+1}{2}, \quad \tau_j = \frac{2}{j+3}, \quad N = \left\lceil \frac{2 \sqrt{2uv} \cdot \delta}{\epsilon} \right \rceil. $$
Then
$$ f(\Gamma(\delta, \epsilon, x_0)) - \hat{f} \leq \epsilon, \qquad \forall x_0\in Q \ \mbox{satisfying} \ d(x_0, \widehat{X}) \leq \delta, $$
where $d$ is the metric induced by the $\ell^2$-norm. It follows that we can set
$$ \cC_{\Gamma} (\delta, \epsilon, x_0) = \left \lceil \frac{2 \sqrt{2uv} \cdot \delta}{\epsilon} \right \rceil .$$
\end{proposition}

This result shows that we can take $d_1=1$ and $d_2=1$ in \R{growth_cond_c} in the case of Nesterov's method with smoothing. \cref{thm:restart-unknown-consts-err} and \cref{rad_ord_cor1} now imply the rates in the second row of \cref{rates_table}. Note that, for example, Lipschitz functions are smoothable \cite{beck2017first}.

\subsection{Row 3 of \cref{rates_table}: The universal fast gradient method}
\label{sec:FGM}

\begin{algorithm}[t]
\DontPrintSemicolon
\SetKwInOut{Input}{Input} \SetKwInOut{Output}{Output} \SetKwRepeat{Do}{do}{while}
\Input{Parameters $\epsilon > 0$, $L_0 > 0$, $\phi_0(x) = 0$, $y_0 = x_0$, $A_0 = 0$.}
\Output{The vector $x_N$, which estimates a minimizer of \eqref{UFGM_1}.}
\For{$k = 0, 1, \dots, N$}{
$v_k \leftarrow \mathrm{prox}_{\phi_k,Q}(x_0)$ \;
$i_k \leftarrow -1$ \;
\Do{$q(y_{k\! +\! 1,i_k})\! >\!  q(x_{k\! +1\! ,i_k})\!  +\!  \ip{\nabla q(x_{k\! +\! 1,i_k})}{y_{k\! +\! 1,i_k}\! -\! x_{k\! +\! 1,i_k}}_{\bbR}\!  +\!  2^{i_k\! -\! 1} L_k \nmu{y_{k\! +\! 1,i_k}\! -\! x_{k\! +\! 1,i_k}}_{\ell^2}^2\!  +\!  \frac{\epsilon}{2} \tau_{k,i_k}\! \! \! \!\! \!\! \! \! \! \! \! \! \! \! \!   $}
{
$i_k \leftarrow i_k + 1$ \;
Compute $a_{k+1,i_k}$ from the equation $a_{k+1,i_k}^2 = \frac{1}{2^{i_k} L_k}(A_k + a_{k+1,i_k})$. \\
$A_{k+1,i_k} \leftarrow A_k + a_{k+1,i_k}$ \;
$\tau_{k,i_k} \leftarrow a_{k+1,i_k}/A_{k+1,i_k}$ \;
$x_{k+1,i_k} \leftarrow \tau_{k,i_k} v_k + (1-\tau_{k,i_k}) y_k$ \;
Choose a subgradient $\nabla q(x_{k+1,i_k}) \in \partial q(x_{k+1,i_k})$. \\
$\hat{\phi}_{k+1,i_k}(x) \leftarrow a_{k+1,i_k}[\ip{\nabla q(x_{k+1,i_k})}{x}_{\bbR} + g(x)]$ \;
$\hat{x}_{k+1,i_k} \leftarrow \mathrm{prox}_{\hat{\phi}_{k+1,i_k},Q}(v_k)$ \;
$y_{k+1,i_k} \leftarrow \tau_{k,i_k} \hat{x}_{k+1,i_k} + (1 - \tau_{k,i_k})y_k$ \;
}
$x_{k+1} \leftarrow x_{k+1,i_k}$, $y_{k+1} \leftarrow y_{k+1,i_k}$, $a_{k+1} \leftarrow a_{k+1,i_k}$, $\tau_k \leftarrow \tau_{k,i_k}$ \\
$A_{k+1} \leftarrow A_k + a_{k+1}$, $L_{k+1} \leftarrow 2^{i_k-1} L_k$ \\
$\phi_{k+1}(x) \leftarrow \phi_k(x) + a_{k+1}[q(x_{k+1}) + \ip{\nabla q(x_{k+1})}{x-x_{k+1}}_{\bbR} + g(x)]$.
}
\caption{Universal fast gradient method} \label{alg:UFGM}
\end{algorithm}

We next consider H\"older smooth functions, which are a natural way of interpolating between non-smooth and smooth objective functions.

\defn{
\label{def:holder-smooth22}
A convex function $q : \bbC^n \rightarrow \bbR$ is H\"older smooth over $Q \subseteq \bbC^n$ with parameter $\nu \in [0,1]$ if
\begin{equation}
\nmu{\nabla q(x) - \nabla q(y) }_{\ell^2} \leq M_\nu \nmu{x - y}_{\ell^2}^\nu, \qquad \forall \ x, y \in Q, \nabla q(x)\in\partial q(x),\nabla q(y)\in\partial q(y).\tag*{$\blacktriangle$}
\end{equation}
}

We consider the universal fast gradient method \cite{nesterov2015universal} for the problem
\begin{equation}
\label{UFGM_1}
\min_{x \in Q} f(x),\qquad f(x) : = q(x) + g(x),
\end{equation}
where $q$ is a proper convex function that is H\"older smooth for some $\nu\in[0,1]$, and $g$ is a closed convex function whose proximal map,
$$
\mathrm{prox}_{c g,Q}(x)=\argmin{y \in Q} \left\{ c\cdot g(y)+\frac{1}{2}\nmu{x - y}_{\ell^2}^2 \right\},
$$
is straightforward to compute. The iterates of the universal fast gradient method are summarized in \cref{alg:UFGM}.

\begin{lemma}[Theorem 3 of \cite{nesterov2015universal}]
Let $Q \subseteq \bbC^n$ be nonempty, closed and convex, $q$ a proper convex function that is H\"older smooth for some $\nu\in[0,1]$ and $M_\nu<\infty$ (\cref{def:holder-smooth22}), and $g$ a closed convex function. Then \cref{alg:UFGM} generates a sequence $\{x_k\}_{k=1}^\infty \subset Q$ satisfying
\begin{equation}
\label{UFGM_bound}
    f(x_k) - \hat f \leq \left(\frac{2^{2+4\nu}M_\nu^2}{\epsilon^{1-\nu}k^{1+3\nu}}\right)^{\frac{1}{1+\nu}}\frac{d(x_0, \widehat{X})^2}{2}+\frac{\epsilon}{2}, \qquad \forall x \in Q,
\end{equation}
where $d$ is the metric induced by the $\ell^2$-norm.
\end{lemma}

The following proposition is immediate when choosing $k$ to match the two terms on the right-hand side of \eqref{UFGM_bound}.

\begin{proposition} \label{prop:UFGM-Gamma}
Let $Q \subseteq \bbC^n$ be nonempty, closed and convex, $q$ a proper convex function is H\"older smooth for some $\nu \in [0,1]$ and $M_{\nu} \geq 0$ (\cref{def:holder-smooth22}), and $g$ a closed convex function. Given input $(\delta, \epsilon, x_0) \in \bbR_+ \times \bbR_+ \times Q$, let $\Gamma(\delta, \epsilon, x_0)$ be the output of \cref{alg:UFGM} with 
$$
N= \left\lceil \frac{2^{\frac{2+4\nu}{1+3\nu}} M_\nu^{\frac{2}{1+3\nu}}  \delta^{\frac{2+2\nu}{1+3\nu}}}{\epsilon^{\frac{2}{1+3\nu}}} \right \rceil.
$$
Then
$$ f(\Gamma(\delta, \epsilon, x_0)) - \hat{f} \leq \epsilon, \qquad \forall x_0\in Q \ \mbox{satisfying} \ d(x_0, \widehat{X}) \leq \delta, $$
where $d$ is the metric induced by the $\ell^2$-norm. It follows that we can set
$$ \cC_{\Gamma} (\delta, \epsilon, x_0) =\left\lceil \frac{2^{\frac{2+4\nu}{1+3\nu}} M_\nu^{\frac{2}{1+3\nu}}  \delta^{\frac{2+2\nu}{1+3\nu}}}{\epsilon^{\frac{2}{1+3\nu}}} \right \rceil .$$
\end{proposition}

\cref{prop:UFGM-Gamma} shows that we can take $d_1=(2+2\nu)/(1+3\nu)$ and $d_2=2/(1+3\nu)$ for the universal fast gradient method. Note that if $q$ satisfies both \eqref{eqn:sharpness} for $\eta=0$ and \cref{def:holder-smooth22}, then $\beta\geq 1+\nu$ \cite{roulet2020sharpness}. Therefore, we take $\beta_0=1+\nu$. \cref{thm:restart-unknown-consts-err} and \cref{rad_ord_cor1} now imply the rates in the third row of \cref{rates_table}.

\subsection{Row 4 of \cref{rates_table}: The primal-dual iteration for unconstrained problems}
\label{sec:PD_unconstrained_desc}

We now consider Chambolle and Pock's primal-dual algorithm \cite{chambolle2016ergodic,chambolle2011first}. The primal-dual hybrid gradient (PDHG) algorithm is a popular method to solve saddle point problems \cite{esser2010general,pock2009algorithm,chambolle2018stochastic}. Consider the problem
\begin{equation}
\label{PD_example_1_1}
\min_{x \in \bbC^n} f(x),\qquad f(x) : = q(x) + g(x) + h(Bx),
\end{equation}
where: $B\in\mathbb{C}^{m\times n}$ with $\nm{B}\leq L_B$; $q$ is a proper, lower semicontinuous, convex function, and is $L_q$-smooth; and $g, h$ are proper, lower semicontinuous, convex functions whose proximal maps are straightforward to compute. We also use the standard Euclidean metric for $d$ in \eqref{eqn:sharpness} and write the primal-dual iterates in their simplified form accordingly.

\begin{algorithm}[t]
\SetKwInOut{Input}{Input} 
\SetKwInOut{Output}{Output}
\SetKwComment{Comment}{// }{}
\Input{Initial vectors ${x}_0 \in \mathbb{C}^n$ and ${y}_0 \in \mathbb{C}^m$, proximal step sizes $\tau,\sigma>0$, number of iterations $N$, matrix $B\in\mathbb{C}^{m\times n}$, and routines for appropriate proximal maps.}
\Output{Final ergodic average $X_N$ approximating a solution to \eqref{PD_example_1_1}.}
Initiate with $x^{(0)}=x_0$, $y_1^{(0)}=y_0$, $X_0=0$, and $Y_0=0$.\\
\For{$j=0,\ldots,N-1$}{
$x^{(j+1)}\gets \mathrm{prox}_{\tau g}\left(x^{(j)}-\tau B^* y^{(j)}-\tau\nabla q(x^{(j)})\right)$\;
$y^{(j+1)}\gets\mathrm{prox}_{\sigma h^*}\left(y^{(j)}+\sigma B(2x^{(j+1)}-x^{(j)})\right)$\;
$X_{j+1}\gets\frac{1}{j+1}\left(jX_{j}+x^{(j+1)}\right)$\;
$Y_{j+1}\gets\frac{1}{j+1}\left(jY_{j}+y^{(j+1)}\right)$\;
}
\caption{Primal-dual algorithm for the problem \eqref{PD_example_1_1}.
}
\label{alg:PD_basic}
\end{algorithm}

The saddle-point problem associated with \eqref{PD_example_1_1} is
\begin{equation}
\label{PD_example_1_2}
\min_{x \in \bbC^n} \max_{y\in \bbC^m} \mathcal{L}(x,y):=\langle Bx,y\rangle_{\mathbb{R}}+q(x)+g(x)-h^*(y).
\end{equation}
The primal-dual iterates are summarized in \cref{alg:PD_basic}, where the output is the ergodic average of the primal-dual iterates. Note that the primal-dual algorithm allows us to deal with the matrix $B$ easily, which can be difficult with other first-order methods. If $\tau(\sigma L_B^2+L_q)\leq 1$, then \cite[Theorem 1]{chambolle2016ergodic} shows that for any $x\in\bbC^n$ and $y\in\bbC^m$,
\begin{equation}
\label{saddle_bound_WARPd1}
\mathcal{L}\left(X_k,y\right)-\mathcal{L}\left(x,Y_k\right)\leq \frac{1}{k}\left(\frac{\nmu{x-x^{(0)}}^2}{\tau}+\frac{\nmu{y-y^{(0)}}^2}{\sigma}\right).
\end{equation}
The following lemma is a simple consequence of this bound and is proven in \cref{proofs_PD_appendix}.

\begin{lemma}
\label{PD_lemma_1}
Consider the primal-dual iterates in \cref{alg:PD_basic}. If $\tau(\sigma L_B^2+L_q)\leq 1$, then
\begin{equation}
\label{objective_bound_WARPd1}
f(X_k)-f(x)\leq \frac{1}{k}\left(\frac{\nmu{x-x^{(0)}}^2}{\tau}+\frac{\nmu{y-y^{(0)}}^2}{\sigma}\right), \quad \forall x\in\mathbb{C}^n,\ y \in \partial h(B X_k).
\end{equation}
\end{lemma}

We can take the infimum over $y \in \partial h(B X_k)$ on the right-hand side of \eqref{objective_bound_WARPd1} to obtain
\begin{equation}
\label{objective_bound_WARPd2}
f(X_k)-f(x)\leq \frac{1}{k}\left(\frac{\nmu{x-x^{(0)}}^2}{\tau}+\frac{\sup_{z\in \mathrm{dom}(h)}\inf_{y\in\partial h(z)}\nmu{y-y^{(0)}}^2}{\sigma}\right), \quad \forall x\in\mathbb{C}^n.
\end{equation}
To bound the right-hand side, we take $y^{(0)}=0$ and consider the case where $h$ is such that there always exist points $y$ in the subdifferential of $h$ for which $\nm{y}$ is not too large. Note that this always holds if, for example, $h$ is Lipschitz continuous, and its domain is open \cite[Theorem 3.61]{beck2017first}. The following proposition now shows how this falls into the framework of our restart scheme and is proven in \cref{proofs_PD_appendix}.

\begin{proposition}
\label{PDGH_prop_1}
Suppose that
\begin{equation}
\label{PDGH_prop_1_assumption}
\sup_{z\in \mathrm{dom}(h)}\inf_{y\in\partial h(z)}\nm{y}\leq L_h<\infty.
\end{equation}
Given input $(\delta,\epsilon,x_0)\in \bbR_+  \times \bbR_+ \times \mathbb{C}^n$, let $\Gamma(\delta,\epsilon,x_0)$ be the output of \cref{alg:PD_basic} with
$$
y_0=0,\quad \tau=\frac{\delta}{L_BL_h+\delta L_q},\quad \sigma=\frac{L_h}{\delta L_B}, \quad N=\left\lceil \frac{\delta}{\epsilon }\left(2L_BL_h+\delta L_q\right) \right \rceil.
$$
Then 
\begin{equation}
\label{PD_fits1}
f(\Gamma(\delta,\epsilon,x_0)) - \hat{f} \leq \epsilon,\quad \forall x_0\text{ with }d(x_0, \widehat{X}) \leq \delta.
\end{equation}
It follows that we can take
\begin{equation}
\label{PD_cost1}
\cC_\Gamma(\delta, \epsilon, x_0)= \left \lceil \frac{\delta}{\epsilon }\left(2L_BL_h+\delta L_q\right) \right \rceil.
\end{equation}
\end{proposition}

Assuming that $\delta$ is bounded, \cref{PDGH_prop_1} shows we can take $d_1=1$ and $d_2=1$ for the primal-dual algorithm. \cref{thm:restart-unknown-consts-err} and \cref{rad_ord_cor1} now imply the rates in the fourth row of \cref{rates_table}. Note, however, that it is not immediately clear how to employ line searches in the case that $L_q$ is unknown.

\subsection{Row 5 of \cref{rates_table}: The primal-dual iterations for constrained problems}
\label{sec:PD_constrained_alg}

We now consider primal-dual iterations, but for the constrained problem
\begin{equation}
\label{PD_example_1_3}
\min_{x \in \bbC^n} f(x)+\chi_C(Ax),\qquad f(x) : = q(x) + g(x) + h(Bx),
\end{equation}
with the same assumptions on $q,g,h$ and $B$ as in \cref{sec:PD_unconstrained_desc}, but now with the additional term $\chi_C(Ax)$. Here, $C$ is a closed and non-empty convex set, $\chi_C$ is its indicator function, and $A\in\mathbb{C}^{m'\times n}$ with $\nm{A}\leq L_A$. This fits into our framework with the choice
$$
Q=\{x\in\mathbb{C}^n:Ax\in C\},\qquad g_Q(x)=g_Q(\kappa;x)=\kappa\cdot\inf_{z\in C} \nm{Ax-z},
$$
for $\kappa>0$. Note that $\kappa$ is an additional parameter that can be chosen to balance the reduction rate in the feasibility gap versus the objective function error. It is possible to formulate a projected version of the primal-dual iteration. However, like with Nesterov's method, this is only possible when the projection onto $Q$ can be easily computed. This section considers a primal-dual iteration for \R{PD_example_1_3} that only involves computing the projection onto the set $C$ at the price of producing non-feasible iterations.

\begin{algorithm}[t]
\SetKwInOut{Input}{Input} 
\SetKwInOut{Output}{Output}
\SetKwComment{Comment}{// }{}
\Input{Initial vectors ${x}_0 \in \mathbb{C}^n$, $[{y}_0]_1 \in \mathbb{C}^{m}$ and $[{y}_0]_2 \in \mathbb{C}^{m'}$, proximal step sizes $\tau,\sigma_1,\sigma_2>0$, number of iterations $N$, matrices $B\in\mathbb{C}^{m\times n}$ and $A\in\mathbb{C}^{m'\times n}$, and routines for appropriate proximal maps.}
\Output{Final ergodic average $X_N$ approximating a solution to \eqref{PD_example_1_3}.}
Initiate with $x^{(0)}=x_0$, $y_1^{(0)}=[y_0]_1$, $y_2^{(0)}=[y_0]_2$, $X_0=0$, $[Y_0]_1=0$, and $[Y_0]_2=0$.\\
\For{$j=0,\ldots,N-1$}{
$x^{(j+1)}\gets \mathrm{prox}_{\tau g}\left(x^{(j)}-\tau B^* y^{(j)}_1-\tau A^* y^{(j)}_2-\tau\nabla q(x^{(j)})\right)$\;
$y^{(j+1)}_1\gets\mathrm{prox}_{\sigma_1 h^*}\left(y^{(j)}_1+\sigma_1 B(2x^{(j+1)}-x^{(j)})\right)$\;
$y^{(j+1)}_2\gets y^{(j)}_2+\sigma_2A(2x^{(j+1)}-x^{(j)})-\sigma_2\mathcal{P}_C\left(y^{(j)}_2/\sigma_2+A(2x^{(j+1)}-x^{(j)})\right)$\;
$X_{j+1}\gets\frac{1}{j+1}\left(jX_{j}+x^{(j+1)}\right)$\;
$[Y_{j+1}]_1\gets\frac{1}{j+1}\left(j[Y_{j}]_1+y^{(j+1)}_1\right)$\;
$[Y_{j+1}]_2\gets\frac{1}{j+1}\left(j[Y_{j}]_2+y^{(j+1)}_2\right)$\;
}
\caption{Primal-dual algorithm for the constrained problem \eqref{PD_example_1_3}.}
\label{alg:PD_basic2}
\end{algorithm}

The saddle-point problem associated with \eqref{PD_example_1_3} is
\begin{equation}
\label{PD_example_1_4}
\min_{x \in \bbC^n} \max_{y_1\in \bbC^m} \max_{y_2\in \bbC^{m'}} \mathcal{L}_C(x,y_1,y_2):=\langle Bx,y_1\rangle_{\mathbb{R}}+q(x)+g(x)-h^*(y_1)+\langle Ax,y_2\rangle_{\mathbb{R}}-\sup_{z\in C}\langle z,y_2\rangle_{\mathbb{R}}.
\end{equation}
The primal-dual iterates are summarized in \cref{alg:PD_basic2}, where the output is the ergodic average of the primal-dual iterates. We have included three proximal step sizes: $\tau$, $\sigma_1$, and $\sigma_2$, corresponding to the primal and two dual variables. To compute the proximal map associated with the second dual variable, we use Moreau's identity to write
$$
\mathrm{prox}_{\sigma_2 \chi_C^*}(y)=y-\sigma_2\mathcal{P}_C(y/\sigma_2),
$$
where $\mathcal{P}_C$ denotes the projection onto $C$ (with respect to the standard Euclidean norm). 

If $\tau(\sigma_1 L_B^2+\sigma_2 L_A^2+L_q)\leq 1$, then a straightforward adaption of \cite[Theorem 1]{chambolle2016ergodic} shows that for any $x\in\bbC^n$, $y_1\in\bbC^m$ and $y_2\in\bbC^{m'}$,
\begin{equation}
\label{saddle_bound_WARPd1b}
\mathcal{L}_C\left(X_k,y_1,y_2\right)-\mathcal{L}_C\left(x,[Y_k]_1,[Y_k]_2\right)\leq \frac{1}{k}\left(\frac{\nmu{x-x^{(0)}}^2}{\tau}+\frac{\nmu{y_1-y_1^{(0)}}^2}{\sigma_1}+\frac{\nmu{y_2-y_2^{(0)}}^2}{\sigma_2}\right).
\end{equation}
We now have the following lemma and resulting proposition, both proven in \cref{proofs_PD_appendix2}.

\begin{lemma}
\label{PD_lemma_2}
Consider the primal-dual algorithm in \cref{alg:PD_basic2} with $y_2^{(0)}=0$. If $\tau(\sigma_1 L_B^2+\sigma_2 L_A^2+L_q)\leq 1$, then for any $\kappa>0$
\begin{equation}
\label{objective_bound_WARPd1b}
f(X_k)-f(x)+ g_Q(\kappa;X_k) \leq \frac{1}{k}\left(\frac{\nmu{x-x^{(0)}}^2}{\tau}+\frac{\nmu{y_1-y_1^{(0)}}^2}{\sigma_1}+\frac{\kappa^2}{\sigma_2}\right), \quad \forall x\in Q,\ y_1 \in \partial h(B X_k).
\end{equation}
\end{lemma}

\begin{proposition}
\label{PDGH_prop_2}
Suppose that
\begin{equation}
\label{PDGH_prop_1_assumptionb}
\sup_{z\in \mathrm{dom}(h)}\inf_{y\in\partial h(z)}\nm{y}\leq L_h<\infty.
\end{equation}
Given input $(\delta,\epsilon,x_0)\in \bbR_+ \times \bbR_+ \times \mathbb{C}^n$, let $\Gamma(\delta,\epsilon,x_0)$ be the output of \cref{alg:PD_basic2} with
$$
[y_0]_1=0, [y_0]_2=0, \tau=\frac{\delta}{\kappa L_A+L_hL_B+\delta L_q}, \sigma_1=\frac{L_h}{\delta L_B}, \sigma_2=\frac{\kappa}{\delta L_A}, N=\left \lceil \frac{\delta\left(2\kappa L_A+2L_hL_B+\delta L_q\right)}{\epsilon } \right \rceil.
$$
Then 
\begin{equation}
\label{PD_fits1b}
f(\Gamma(\delta,\epsilon,x_0)) - \hat{f} + g_Q(\kappa;\hat{x}) \leq \epsilon,\quad \forall x_0\text{ with }d(x_0, \widehat{X}) \leq \delta.
\end{equation}
It follows that we can take
\begin{equation}
\label{PD_cost1b}
\cC_\Gamma(\delta, \epsilon,  x_0)= \left \lceil \frac{\delta\left(2\kappa L_A+2L_hL_B+\delta L_q\right)}{\epsilon } \right \rceil.
\end{equation}
\end{proposition}

\cref{PDGH_prop_2} shows we can take $d_1=1$ and $d_2=1$ for the primal-dual algorithm. \cref{thm:restart-unknown-consts-err} and \cref{rad_ord_cor1} now imply the rates in the final row of \cref{rates_table}.

\section{Numerical experiments}\label{sec:num-exp}

We implement several numerical experiments for the general restart scheme (\cref{alg:restart-unknown-constsC}) applied to three different problems. The first is a simple sparse recovery problem modeled as QCBP, which is solved using the primal-dual iteration for constrained problems (\cref{alg:PD_basic2}). Second, we consider image reconstruction from Fourier measurements via TV minimization. The reconstruction is computed using NESTA \cite{becker2011nesta}, where NESTA is an accelerated projected gradient descent algorithm derived from Nesterov's method (\cref{alg:nesterov}) with smoothing. Third, we perform feature selection on three real-world datasets. This selection is made by solving an SR-LASSO problem on the data with unconstrained primal-dual iterations (\cref{alg:PD_basic}). The experiments are implemented in MATLAB, and code is available at \textcolor[rgb]{0,0,1}{\url{https://github.com/mneyrane/restart-schemes}}.

Before discussing the examples, we will make general remarks about the implementation. First, we use the schedule criteria from \cref{sec:expansion_functions}, and for parameters we always set $c_1 = c_2 = 2$, $b = \E$, $r = \E^{-1}$, and $a = e^{c_1 \beta / d_1}$ for unknown $\alpha$ but known $\beta$ (\cref{known_beta}), otherwise $a = e^{c_1 / d_1}$ if both are unknown (\cref{rad_ord_cor1}). The choice of $r$ is motivated by \cref{sec:bestr} and the choice of $a$ by \cref{rem:choice-of-ab}. The choice of $c_1$ and $c_2$ were arbitrary, intending to be sensible defaults, and otherwise can be tuned to improve performance. Second, when using the restart scheme for primal-dual iterations, we store and perform restarts on the dual variables for each instance indexed by $(i,j)$. Third, we use a simple workaround to handle finite precision arithmetic. In the grid search for the restart scheme, the sharpness parameter $\alpha_i$ can be arbitrarily large or small, and $\beta_j$ can be arbitrarily large. Also, the adaptive restart parameters $\delta = \delta_{i,j,U}$ and $\epsilon = \epsilon_{i,j,U}$ can become arbitrarily small. Regarding the grid indices, we limit $i$ and $j$ so that 
$$
|i| \leq \lfloor \log_a(1/\epsilon_{\text{mach}}) \rfloor, \qquad j \leq \lfloor \log_b(1/\epsilon_{\text{mach}}) \rfloor,
$$
where $\epsilon_{\text{mach}}$ is machine epsilon. Regarding the adaptive parameters, after the assignments of $\delta_{i,j,U+1}$ and $\epsilon_{i,j,U+1}$ in \cref{alg:restart-unknown-constsC}, we insert the updates $\delta_{i,j,U+1} := \max(\delta_{i,j,U+1}, 10 \epsilon_{\text{mach}})$ and $\epsilon_{i,j,U+1} := \max(\epsilon_{i,j,U+1}, 10 \epsilon_{\text{mach}})$ to avoid setting them to zero. Fourth, we slightly modify the primal-dual algorithm to improve overall performance. For each $j \geq 1$, we track a separate iterate $\widetilde{X}_j$ defined by
$$
\widetilde{X}_j = \mathrm{argmin}_{i=1,\dots,j} f(X_i) + \kappa g_Q(X_i), \qquad j \geq 1.
$$
The iterates $\{\widetilde{X}_j\}_{j \geq 1}$ are not used in the primal-dual algorithm but are instead used to evaluate the reconstruction or objective error in our experiments. In addition, the algorithm returns $\widetilde{X}_N$ as its final iterate. We similarly track a separate iterate for the dual variables, selecting them based on evaluating the Lagrangian \eqref{PD_example_1_4} with $\widetilde{X}_j$. Note that choosing to output $\widetilde{X}_N$ instead of $X_N$ is theoretically justified, since if \eqref{Gamma-delta-to-eps} holds, then our modification would still satisfy \eqref{Gamma-delta-to-eps} for the same parameters $(\delta, \epsilon, x_0)$. Fifth, in each example below, we can take $f(x_0)+g_Q(x_0)$ as a suitable value of $\epsilon_0$ since the objective considered in these applications are always non-negative. We found that the method was not sensitive to this starting value, as predicted by its logarithmic appearance in our convergence bounds.

\subsection{Sparse recovery via QCBP}\label{sec:ne-sparse-recovery-via-QCBP}

We now consider the sparse recovery problem previously introduced in \cref{sec:sparse-recovery}.
We consider reconstructing a vector $x \in \bbR^n$ from noisy measurements $y = Ax + e \in \bbR^m$, where $A \in \bbR^{m \times n}$ is a matrix whose entries are i.i.d.\ Gaussian random variables with mean zero and variance $1/m$, and $e \in \bbR^m$ is a noise vector satisfying $\nmu{e}_{\ell^2} \leq \varsigma$ for some noise level $\varsigma > 0$. For a positive integer $n$, we write $[n] = \{1, 2, \dots, M\}$. Given a vector $z = (z_i)_{i=1}^n \in \bbC^n$ and $S \subseteq [n]$, the vector $z_S$ has $i$th entry $z_i$ if $i \in S$, and is zero otherwise. The best $s$-term approximation error of $z$ is once more defined as
$$
\sigma_s(z)_{\ell^1} = \min \{ \nmu{u_S - z}_{\ell^1} : u \in \bbR^n, \ S \subseteq [n], \ |S| \leq s \}.
$$
We assume that $x$ is \textit{approximately} $s$-sparse, in the sense that its best $s$-term approximation error $\sigma_{s}(x)_{\ell^1}$ is small. The recovery of $x$ is formulated as solving the QCBP problem
\ea{
\min_{z \in \bbR^n} \nmu{z}_{\ell^1} \ \text{subject to} \ \nmu{Az-y}_{\ell^2} \leq \varsigma.
\label{eqn:QCBP-Gaussian}
}
We use the following condition on the matrix $A$ to ensure that approximate sharpness holds.

\begin{definition}[Robust null space property, e.g., Definition 5.14 of \cite{adcock2021compressive}]
\label{def:rNSP}
The matrix $A \in \bbC^{m \times n}$ satisfies the \textit{robust Null Space Property (rNSP)} of order $s$ with constants $0 < \rho < 1$ and $\gamma > 0$ if
$$\nmu{v_S}_{\ell^2} \leq \frac{\rho}{\sqrt{s}} \nmu{v_{S^\complement}}_{\ell^1} + \gamma \nmu{Av}_{\ell^2},$$
for all $v \in \bbC^n$ and $S \subseteq [n]$ with $|S| \leq s$.\hfill$\blacktriangle$
\end{definition}

In \cite[Theorem 3.3]{colbrook2021warpd}, it was shown that the robust null space property (rNSP) implies approximate sharpness. We restate the result in the notation of this paper for completeness.

\begin{proposition}[Approximate sharpness of $\ell^1$-norm for QCBP sparse recovery]
\label{prop:sharpness_of_QCBP}
Let $\varsigma > 0$. Suppose $A \in \bbC^{m \times n}$ has the rNSP of order $s$ with constants $0 < \rho < 1$, $\gamma > 0$. Let $y \in \bbC^m$, $D = \bbC^n$, $Q = \{x \in \bbC^n : \nm{A x - y}_{\ell^2} \leq \varsigma \}$ and  $f(x) = \nmu{x}_{\ell^1}$. Then the approximate sharpness condition \eqref{eqn:sharpness} holds with
$$
g_{Q}(z ; \sqrt{s}) = \sqrt{s} \max \{ \nm{A z - y}_{\ell^2} - \varsigma , 0 \}, \quad \alpha = \hat{c}_1\sqrt{s}, \quad \beta = 1, \quad \eta = \hat{c}_2 \sigma_s(x)_{\ell^1} + \hat{c}_3 \varsigma\sqrt{s} ,
$$
for constants $\hat{c}_1,\hat{c}_2,\hat{c}_3 > 0$ depending only on $\rho$ and $\gamma$.
\end{proposition}

The theory of compressed sensing \cite{adcock2021compressive,foucart2013invitation} aims to construct (random) matrices satisfying the rNSP, which is itself implied by the better-known \textit{Restricted Isometry Property} (RIP). For example, if $A$ is a Gaussian random matrix, then it satisfies the rNSP with probability at least $1-\varepsilon$, provided $m \geq C \cdot \left (s \cdot \log(\E N / s) + \log(2/\varepsilon) \right )$ (see, e.g., \cite[Theorem 5.22]{adcock2021compressive}). However, a sharp value of the constant $C$ and the rNSP constants $\rho$ and $\gamma$ are unknown. This implies that the approximate sharpness constants $\alpha$ and $\eta$ are also unknown. This motivates using the restart scheme (\cref{alg:restart-unknown-constsC}), which does not require knowledge of $\alpha$ or $\eta$, to solve \eqref{eqn:QCBP-Gaussian}.

\subsubsection{Experimental setup}

We use the primal-dual iteration for constrained problems (\cref{alg:PD_basic2}) to solve the sparse recovery problem. This can be done by expressing QCBP in \eqref{eqn:QCBP-Gaussian} as \eqref{PD_example_1_3} with 
$$
q \equiv 0, \quad h \equiv 0, \quad B = 0, \quad g(x) = \nmu{x}_{\ell^1}, \quad C = \{z \in \bbC^N: \nmu{z - y}_{\ell^2} \leq \varsigma \}.
$$
Given these choices, the proximal map of $\tau g$ is the shrinkage-thresholding operator, and the projection map is straightforward to compute since $C$ is a shifted $\ell^2$-ball. Moreover, we have $h^*(z) = +\infty$ whenever $z \neq 0$, and is zero otherwise. Therefore the proximal map $\mathrm{prox}_{\sigma_1 h^*}(x) = \nmu{x}_{\ell^2}^2/2$, and thus $y_1^{(j)} = 0$ for all $j > 0$ if the initial data $y_1^{(0)} = 0$. In essence, we can ignore the parameter $\sigma_1$ and update the iterates $y_1^{(j)}$ in the primal-dual iterations (\cref{alg:PD_basic2}). The error bound derived in \cref{PD_lemma_2} holds with the $\sigma_1$ term omitted.

Unless stated otherwise, the parameters used are ambient dimension $n = 128$, sparsity level $s = 10$, measurements $m = 60$, noise level $\varsigma = 10^{-6}$. The ground truth vector $x$ is sparse with $s$ of its entries (randomly selected) corresponding to i.i.d. standard normal entries. The noise vector $e$ is selected uniformly random on the $\ell^2$-ball of radius $\varsigma$ and thus $\nmu{e}_{\ell^2} = \varsigma$. The objective function is $f(x) = \nmu{x}_{\ell^1}$ and the feasibility gap is given by $g_Q(x;\kappa) = \kappa \cdot \max \{\nmu{Ax-y}_{\ell^2} - \varsigma, 0\}$, which is derived from \cref{sec:PD_constrained_alg}. The feasibility gap weight is set to $\kappa = \sqrt{m}$ from \cref{prop:sharpness_of_QCBP}, noting that $s \leq m$ in general. In addition, $\alpha_0 = \sqrt{m}$, $\beta_0 = 1$. The choice of $\alpha_0$ is also motivated by \cref{prop:sharpness_of_QCBP}.

\subsubsection{Results}

\cref{fig:QCBP-Gaussian} shows the performance of the restart scheme in \cref{alg:restart-known-consts} for various fixed values of $\alpha$ and $\beta = 1$. For smaller $\alpha$, the error decreases linearly down to the noise level $\varsigma = 10^{-6}$. This agrees with \cref{thm:restart-known-consts}. Increasing $\alpha$ leads to fast linear convergence up to a threshold (between $10^{1}$ and $10^{1.2}$). After this point, the performance of the restart scheme abruptly breaks down since large $\alpha$ violates the approximate sharpness condition \eqref{eqn:sharpness}.

\begin{figure}[t]
\centering
\begin{minipage}[b]{1\textwidth}
\centering
\begin{overpic}[width=0.49\textwidth,trim={0mm 0mm 0mm 0mm},clip]{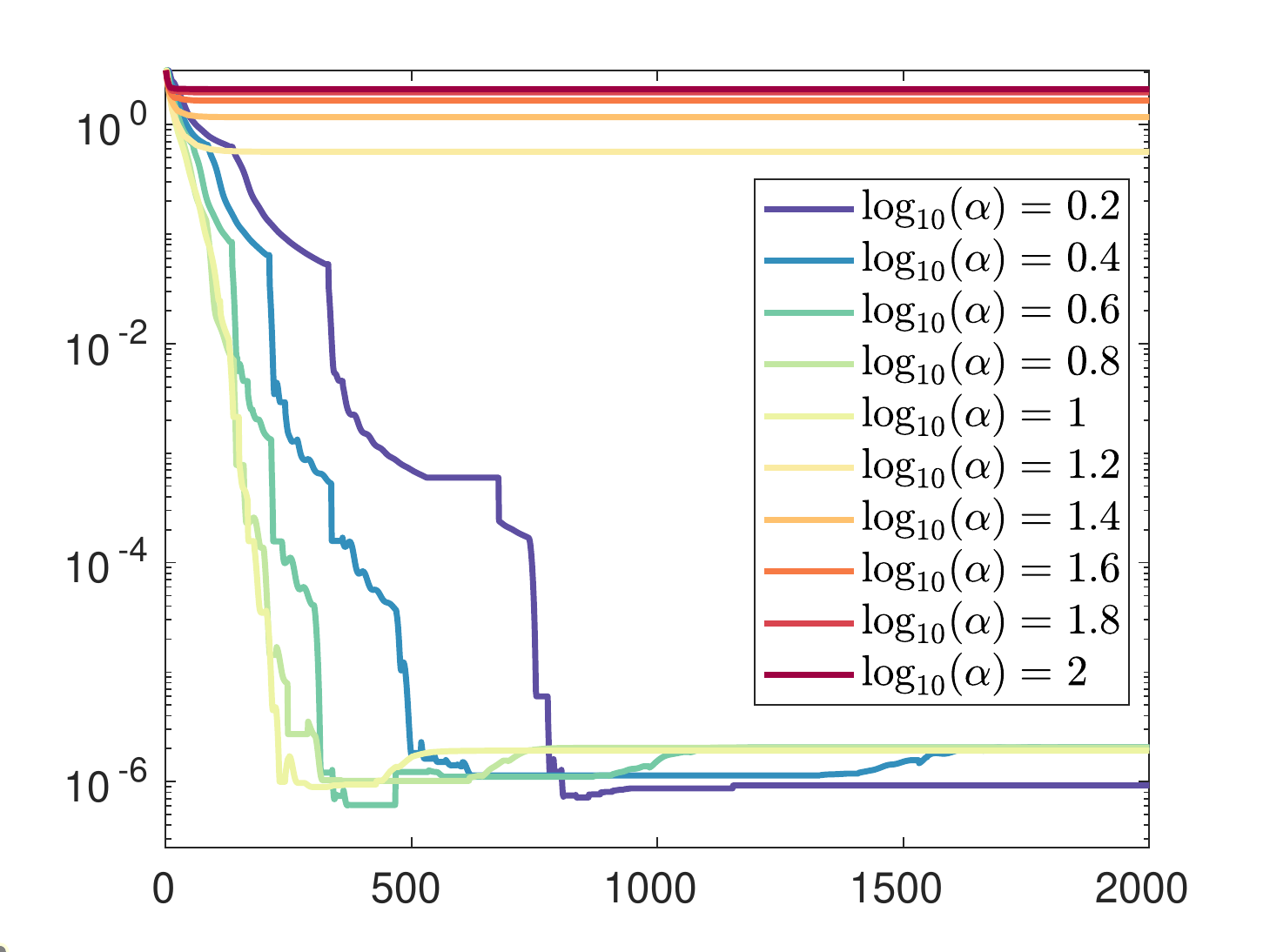}
\put (29,-2) {\small Total inner iterations $t$}
		\put (-1,27) {\small\rotatebox{90}{$\nmu{x_t - x}_{\ell^2}$}}
   \end{overpic}\vspace{2mm}
\begin{overpic}[width=0.49\textwidth,trim={0mm 0mm 0mm 0mm},clip]{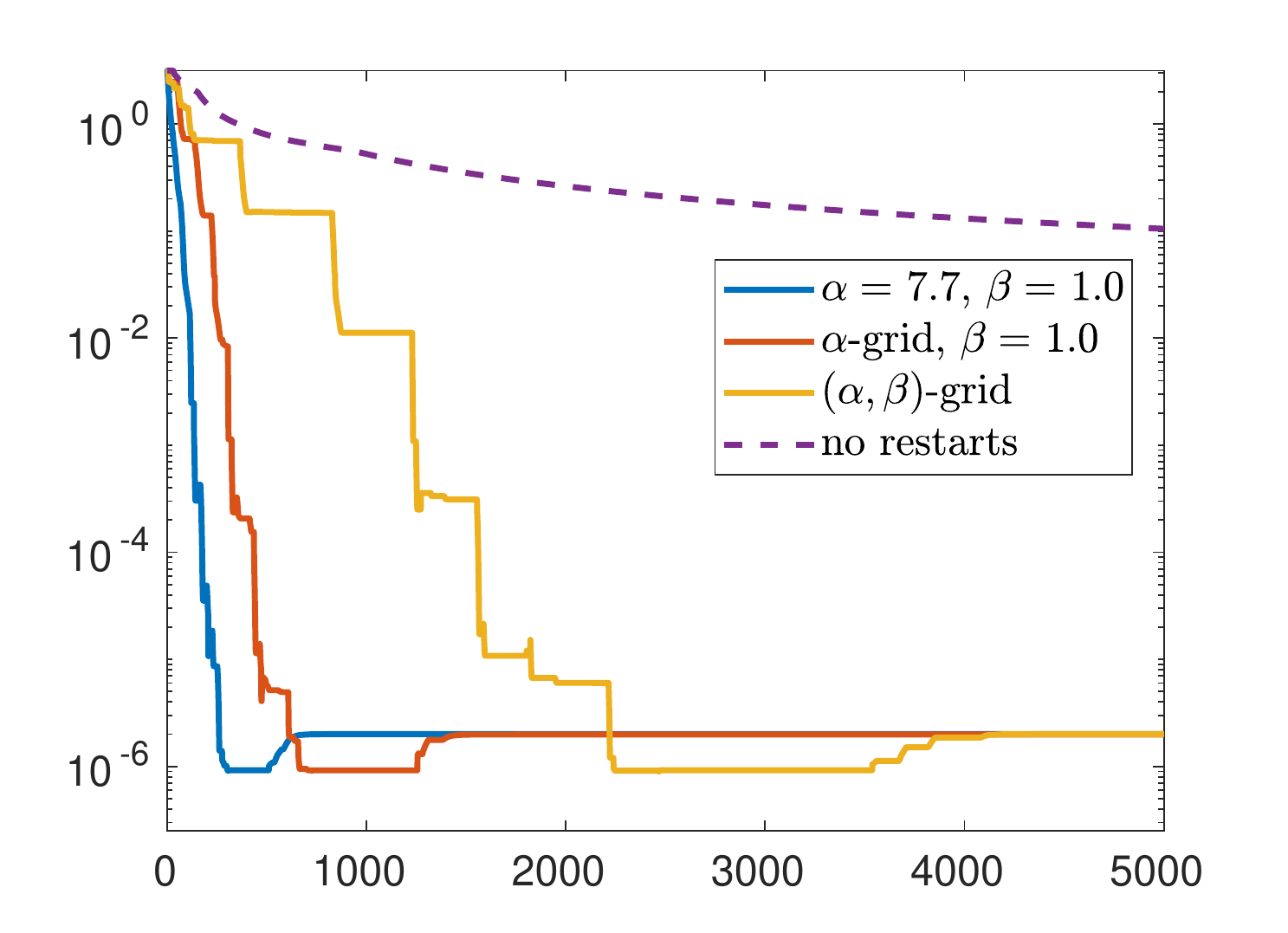}
\put (29,-2) {\small Total inner iterations $t$}
		\put (-1,27) {\small\rotatebox{90}{$\nmu{x_t - x}_{\ell^2}$}}
   \end{overpic}
\end{minipage}
\caption{Reconstruction error of restarted primal-dual iteration for QCBP with $\varsigma = 10^{-6}$. Left: The restart scheme with fixed sharpness constants $\beta = 1$ and various $\alpha$. Right: Various different schemes (including restarted and non-restarted schemes).}
\label{fig:QCBP-Gaussian}
\end{figure}

We use \cref{alg:restart-unknown-constsC} to overcome such parameter sensitivity. \cref{fig:QCBP-Gaussian} also compares the performance of the restart scheme with fixed $(\alpha,\beta) = (\sqrt{m},1)$ with restart schemes that (i) perform a grid search over $\alpha$, for fixed $\beta = 1$, and (ii) perform a grid search over both $\alpha$ and $\beta$. Both grid search schemes exhibit linear convergence, in agreement with \cref{thm:MAIN}. They converge less rapidly than the scheme with fixed $(\alpha,\beta)$ but require no empirical parameter tuning. Note that all restart schemes significantly outperform the non-restarted primal-dual iteration (``no restarts'').

\begin{figure}
\centering
\begin{minipage}[b]{1\textwidth}
\centering
\begin{overpic}[width=0.49\textwidth,trim={0mm 0mm 0mm 0mm},clip]{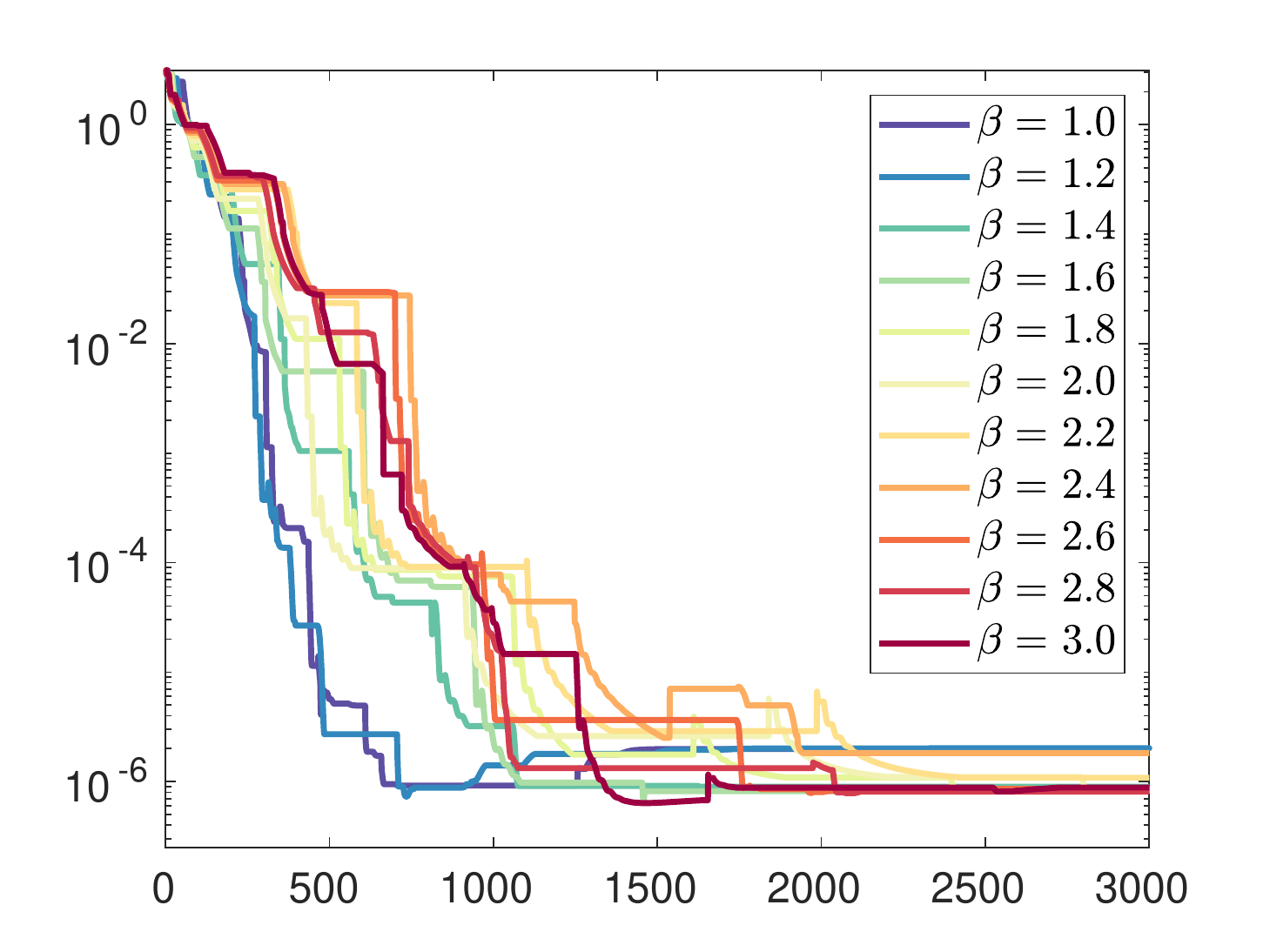}
\put (31,-2) {\small Total inner iterations $t$}
		\put (-1,27) {\small\rotatebox{90}{$\nmu{x_t - x}_{\ell^2}$}}
		\put (35,72) {\small Grid search over $\alpha$}
   \end{overpic}\hfill
\begin{overpic}[width=0.49\textwidth,trim={0mm 0mm 0mm 0mm},clip]{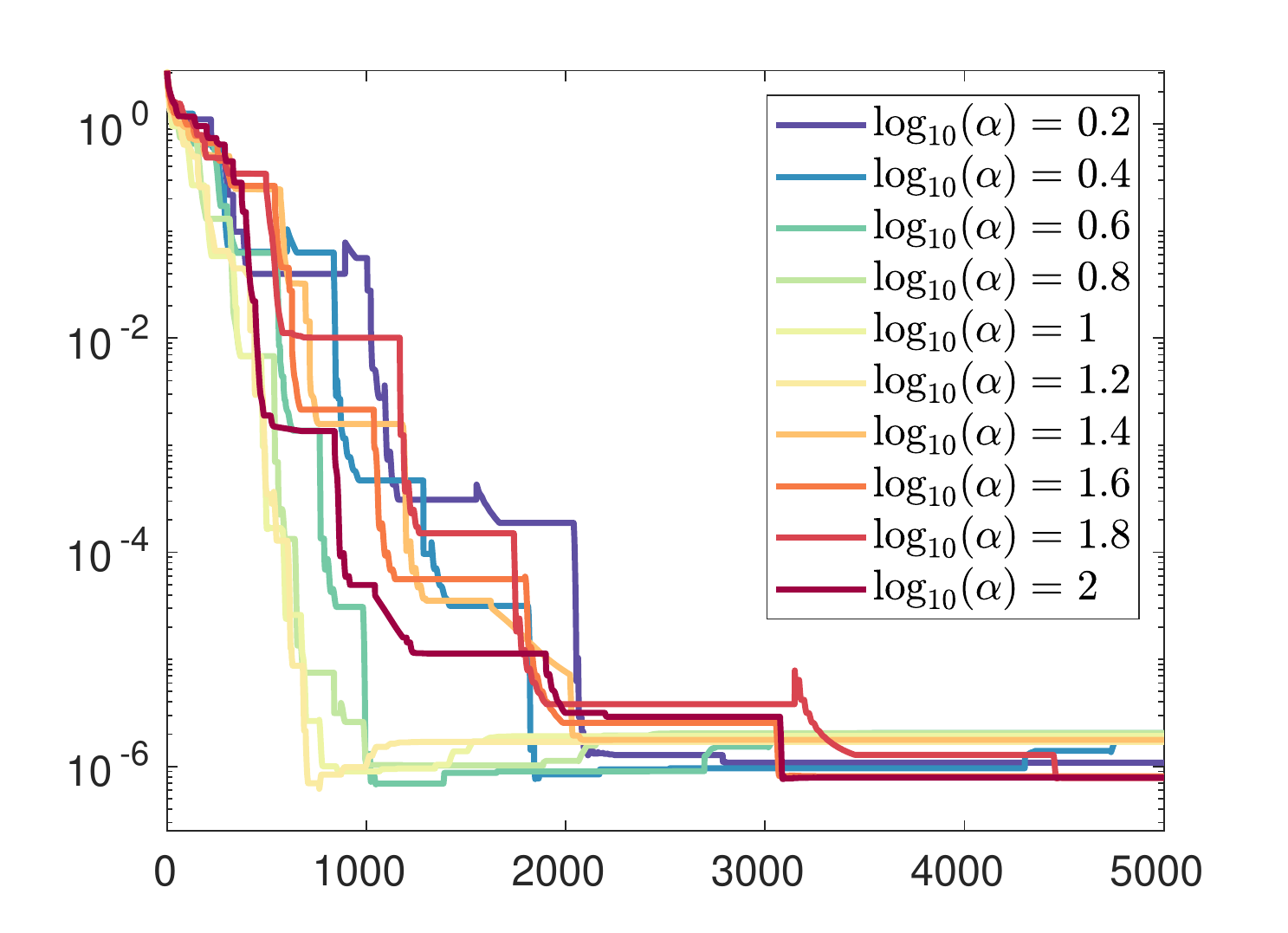}
\put (31,-2) {\small Total inner iterations $t$}
		\put (-1,27) {\small\rotatebox{90}{$\nmu{x_t - x}_{\ell^2}$}}
		\put (35,72) {\small Grid search over $\beta$}
   \end{overpic}
\end{minipage}
\caption{Reconstruction error of restarted primal-dual iteration for QCBP with $\varsigma = 10^{-6}$. Left: The restart scheme with grid search over $\alpha$ and various fixed $\beta$. Right: The restart scheme with grid search over $\beta$ and various fixed $\alpha$.}
\label{fig:QCBP-Gaussian-single-grid-search}
\end{figure}

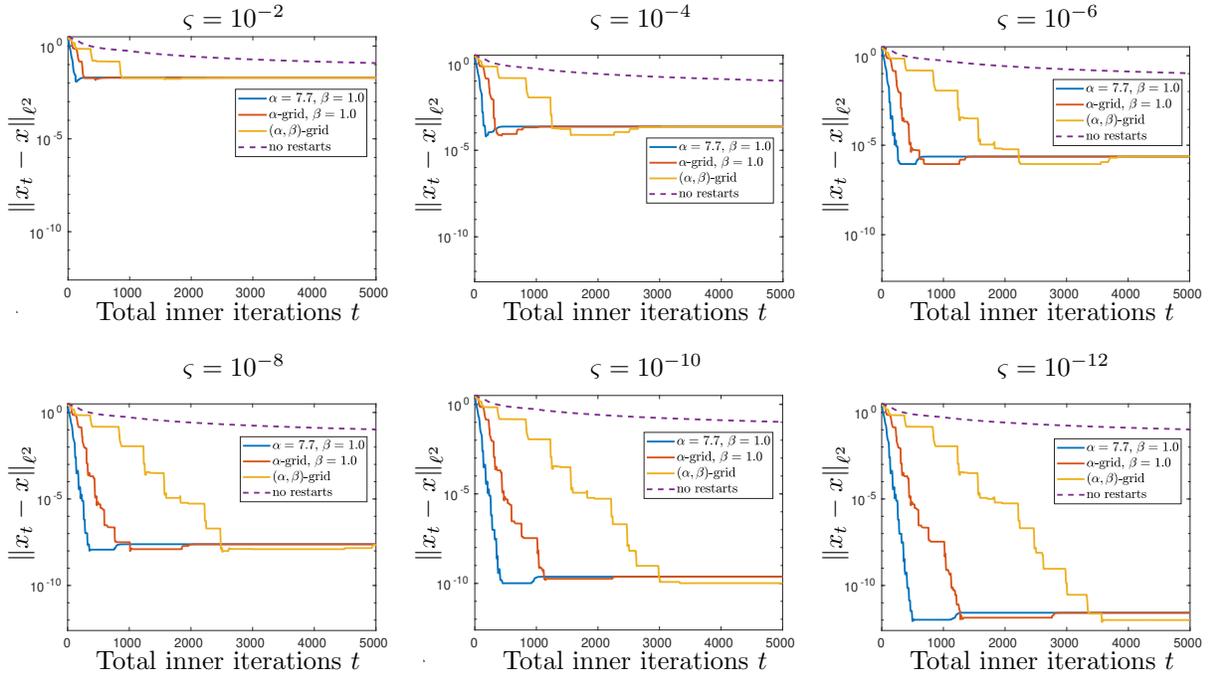
\begin{figure}
\centering
\begin{minipage}[b]{1\textwidth}
\centering
\begin{overpic}[width=0.32\textwidth,trim={0mm 0mm 0mm 0mm},clip]{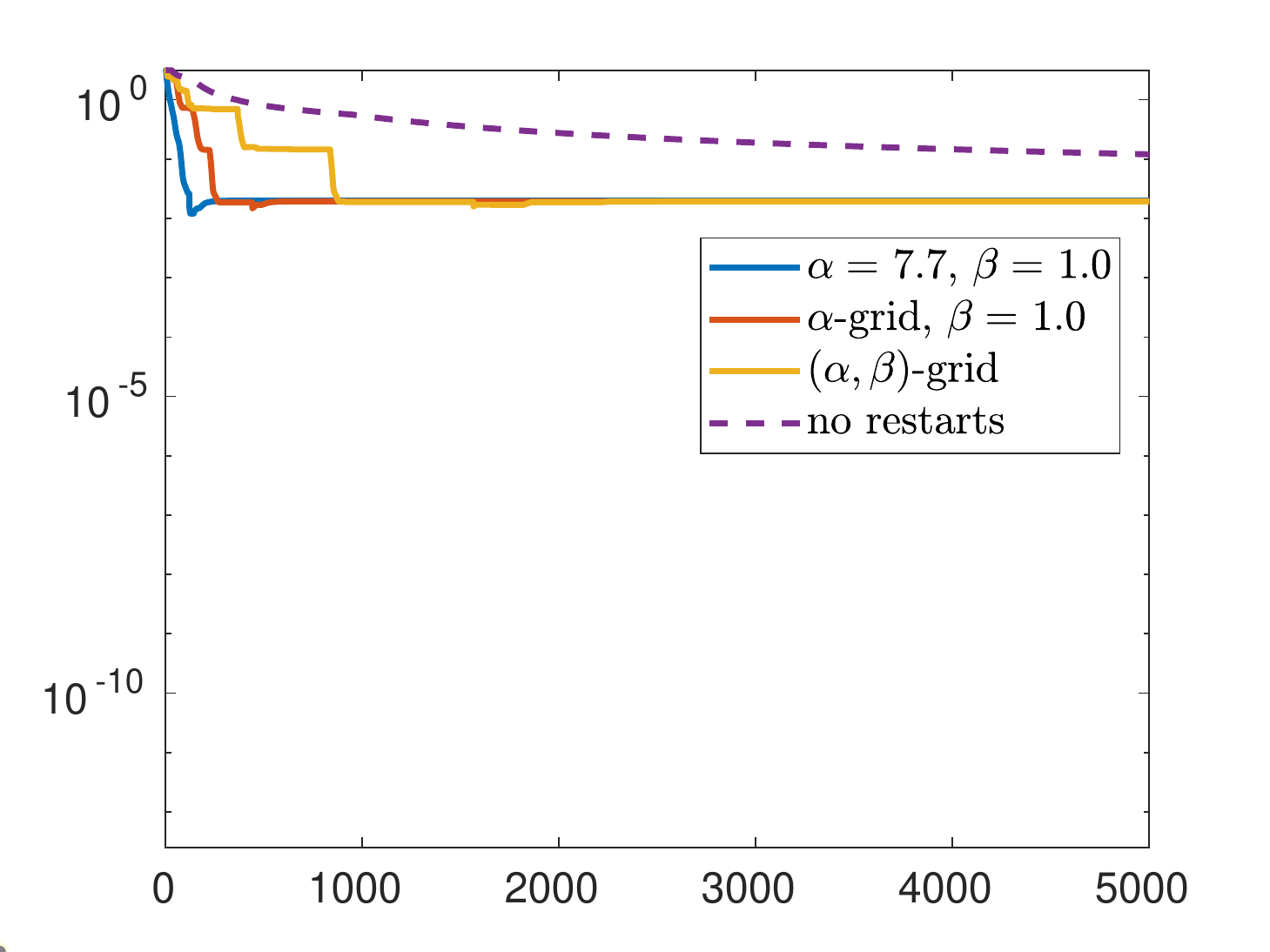}
\put (20,-2) {\small Total inner iterations $t$}
\put (42,72) {\small$\varsigma = 10^{-2}$}
\put (-2,25) {\small\rotatebox{90}{$\nmu{x_t - x}_{\ell^2}$}}
\end{overpic}
\begin{overpic}[width=0.32\textwidth,trim={0mm 0mm 0mm 0mm},clip]{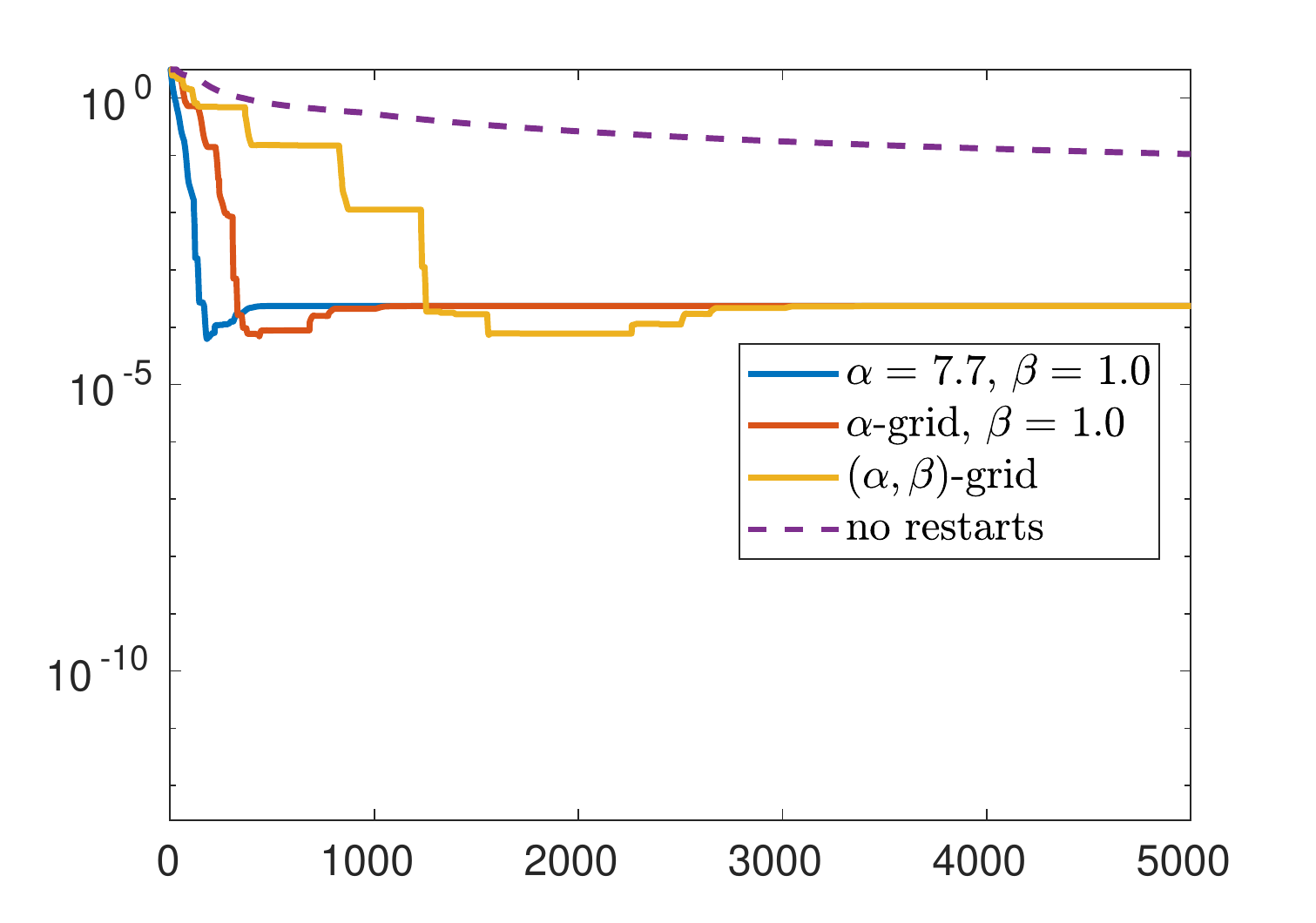}
\put (20,-2) {\small Total inner iterations $t$}
\put (42,72) {\small$\varsigma = 10^{-4}$}
\put (-2,25) {\small\rotatebox{90}{$\nmu{x_t - x}_{\ell^2}$}}
\end{overpic}
\begin{overpic}[width=0.32\textwidth,trim={0mm 0mm 0mm 0mm},clip]{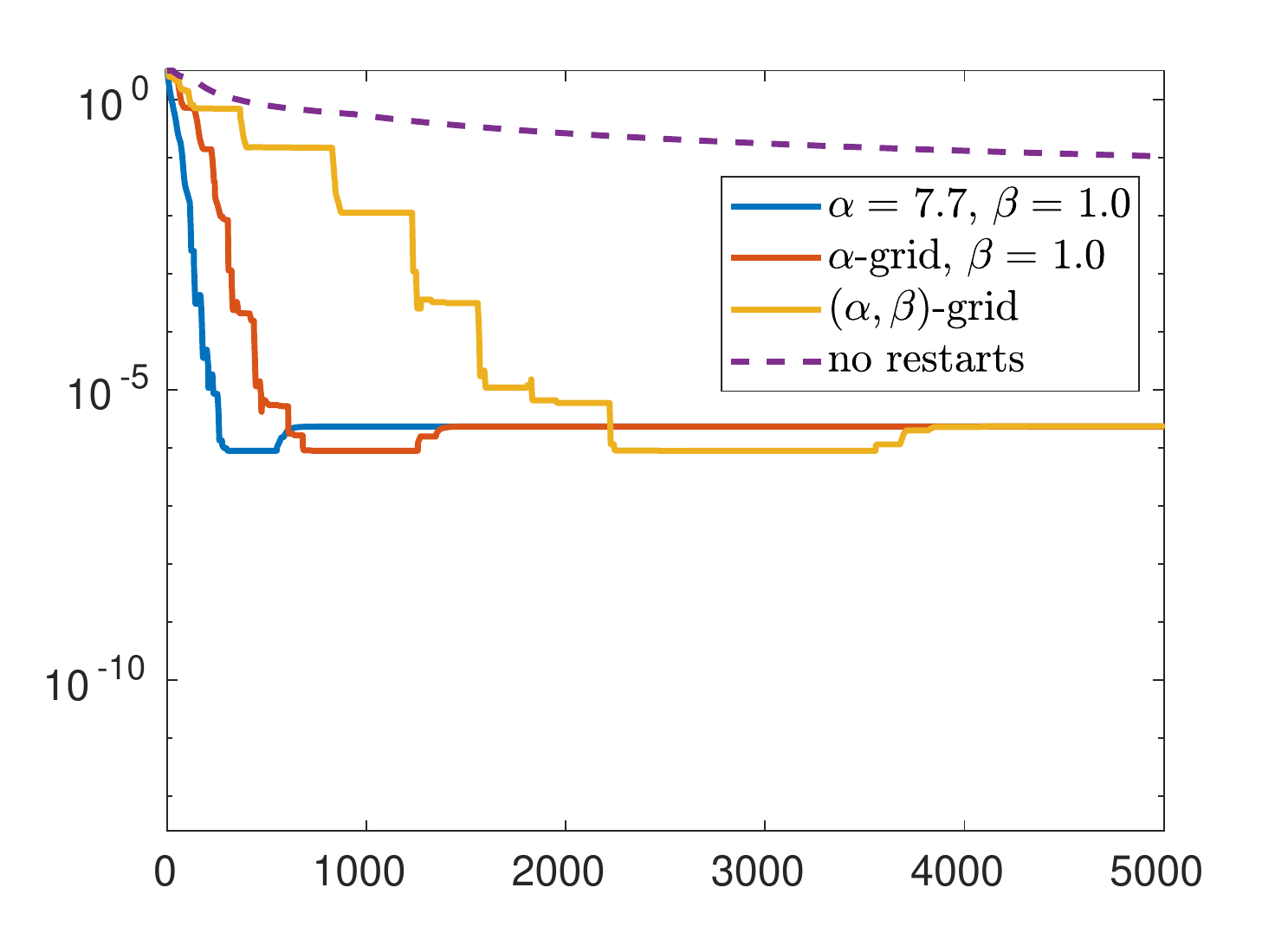}
\put (20,-2) {\small Total inner iterations $t$}
\put (42,72) {\small$\varsigma = 10^{-6}$}
\put (-2,25) {\small\rotatebox{90}{$\nmu{x_t - x}_{\ell^2}$}}
\end{overpic}\\\vspace{8mm}
\begin{overpic}[width=0.32\textwidth,trim={0mm 0mm 0mm 0mm},clip]{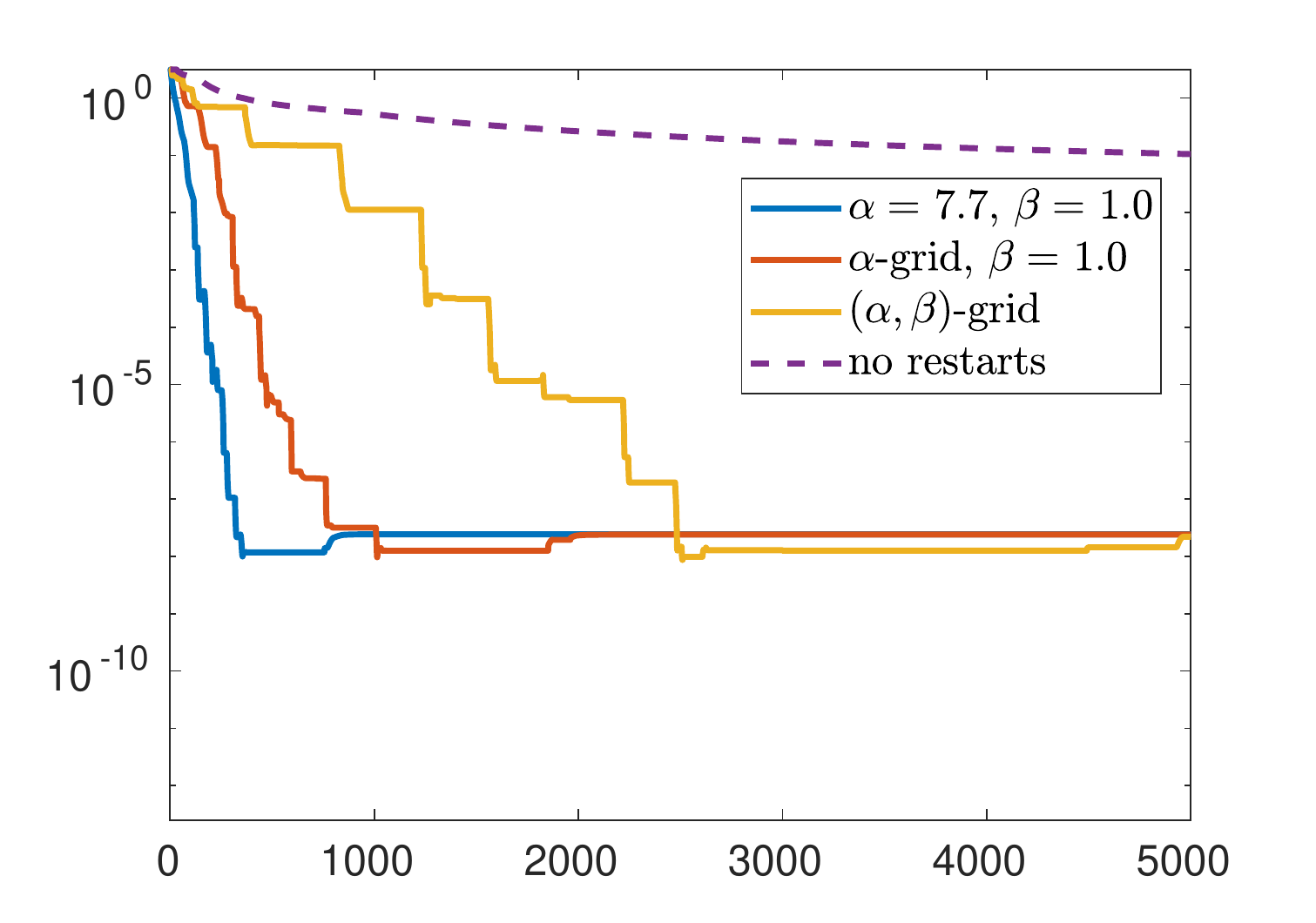}
\put (20,-2) {\small Total inner iterations $t$}
\put (42,72) {\small$\varsigma = 10^{-8}$}
\put (-2,25) {\small\rotatebox{90}{$\nmu{x_t - x}_{\ell^2}$}}
\end{overpic}
\begin{overpic}[width=0.32\textwidth,trim={0mm 0mm 0mm 0mm},clip]{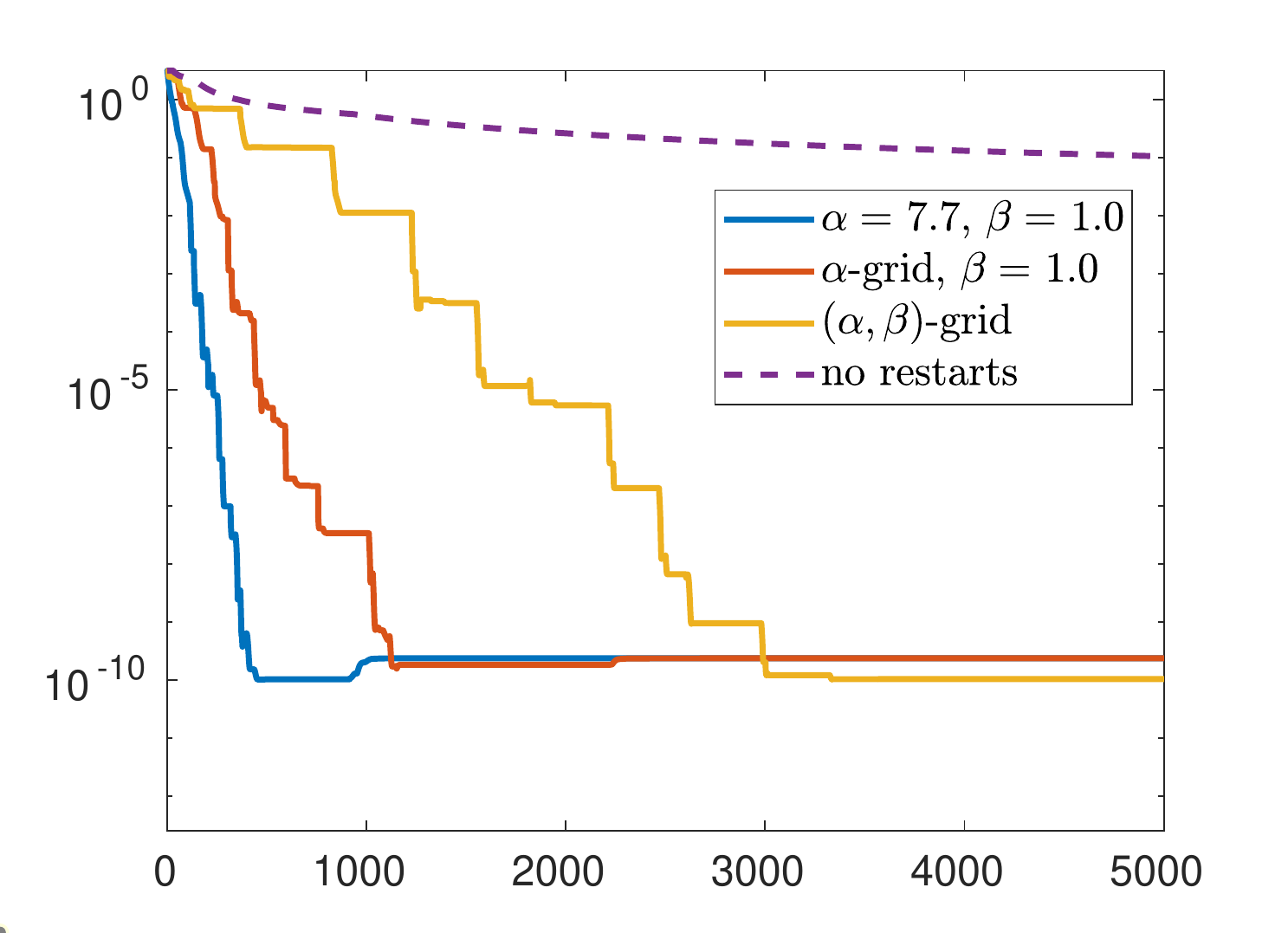}
\put (20,-2) {\small Total inner iterations $t$}
\put (42,72) {\small$\varsigma = 10^{-10}$}
\put (-2,25) {\small\rotatebox{90}{$\nmu{x_t - x}_{\ell^2}$}}
\end{overpic}
\begin{overpic}[width=0.32\textwidth,trim={0mm 0mm 0mm 0mm},clip]{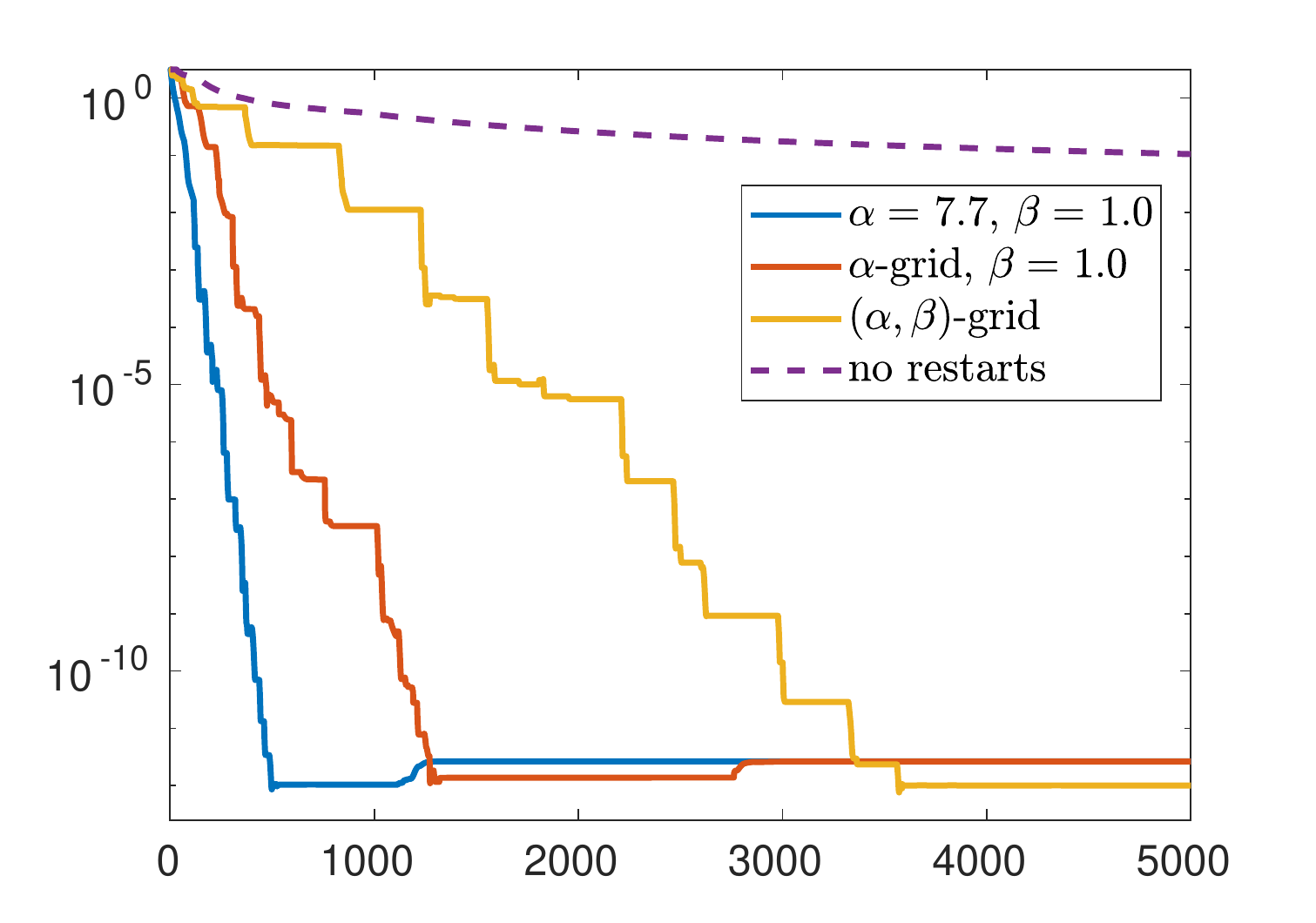}
\put (20,-2) {\small Total inner iterations $t$}
\put (42,72) {\small$\varsigma = 10^{-12}$}
\put (-2,25) {\small\rotatebox{90}{$\nmu{x_t - x}_{\ell^2}$}}
\end{overpic}
\end{minipage}
\caption{Reconstruction error of restarted primal-dual iteration for QCBP with $\varsigma = 10^{-2k}$ for $k = 1,2,\dots,6$. Each plot includes the various (restarted and non-restarted) schemes.}
\label{fig:QCBP-Gaussian-noise-levels}
\end{figure}

Next, we consider two cases of grid searching over exactly one sharpness constant and leaving the other fixed. \cref{fig:QCBP-Gaussian-single-grid-search} shows the results for fixed $\alpha$ with $\beta$ grid search and fixed $\beta$ with $\alpha$ grid search. Both yield linear decay, although at a slightly worse rate. A key point to note is the potential benefit of grid searching. Compare the reconstruction error with those for the fixed restart schemes in \cref{fig:QCBP-Gaussian} with $\log_{10}(\alpha) \geq 1.2$ and $\beta = 1$. In the fixed constant scheme, these parameter choices stall the error. However, $\beta$ grid search overcomes this and reconstructs $x$ within a tolerance proportional to $\varsigma$ after sufficiently many restarts.

Finally, \cref{fig:QCBP-Gaussian-noise-levels} considers the effect on the restart schemes when changing the noise level $\varsigma$. In all cases, the restart schemes linearly decay to a tolerance proportional to $\varsigma$, outperforming the non-restarted primal-dual iterations.

\subsection{Image reconstruction via TV minimization}
\label{sec:ne-tv-minimization}

In this experiment, we consider image reconstruction with Fourier measurements -- a sensing modality with applications notably in Magnetic Resonance Imaging (MRI) \cite{adcock2021compressive}. Specifically, we consider the recovery of a vector $x \in \bbR^n$ from noisy Fourier measurements $y = Ax + e \in \bbC^m$, where $A \in \bbC^{m \times n}$ corresponds to a subsampled Fourier matrix and $e \in \bbC^m$ models noise or perturbations. The vector $x$ is a vectorized complex 2-D image $X \in \bbC^{R \times R}$, where $n = R^2$ for some positive power-of-two integer $R$. The matrix $A$ has the form $A = m^{-1/2} P_\Omega F$, where $F \in \bbC^{n \times n}$ is the 2-D discrete Fourier transform and $\Omega \subseteq n$ is a sampling mask with $|\Omega| = m$. Here, $\Omega$ defines the matrix $P_\Omega \in \bbC^{m \times n}$, which selects the rows of $F$ by index according to the indices in $\Omega$. Lastly, $\nmu{e}_{\ell^2} \leq \varsigma$ for some noise level $\varsigma>0$. A widely used tool for reconstructing $x$ from $y$ is the \textit{total variation (TV) minimization} problem
$$
\min_{z \in \bbC^n} \nmu{V z}_{\ell^1} \ \text{subject to} \ \nmu{Az-y}_{\ell^2} \leq \varsigma,
$$
where $V$ is the 2-D (anisotropic) discrete gradient transform with periodic boundary conditions \cite{adcock2021improved}.

Similar to the sparse recovery problem described in the previous section, the TV-Fourier image reconstruction problem can be shown to have the approximate sharpness condition \eqref{eqn:sharpness} with high probability under a suitable random sampling pattern $\Omega$. Stating and proving this is more involved but can be done by carefully adapting the analysis within \cite[Sec. 7.4]{adcock2021improved}.

\subsubsection{Experimental setup}

The first-order solver we use is NESTA (NESTerov's Algorithm), an accelerated projected gradient descent algorithm used to solve problems of the form
$$
\min_{z \in \bbC^n} \nmu{W^* z}_{\ell^1} \ \text{subject to} \ \nmu{Az-y}_{\ell^2} \leq \varsigma, \qquad W \in \bbC^{n \times m'},
$$
where TV minimization is a special case with $W = V^\top$. NESTA is derived from Nesterov's method with smoothing, where the objective function $f(z) = \nmu{W^*z}_{\ell^1}$ is smoothed by replacing the $\ell^1$-norm with its Moreau envelope. This yields a $1/\mu$-smooth approximation $f_\mu(z) = \nmu{W^*z}_{\ell^1,\mu}$ of $f$ with parameters $(\nmu{W^*}_{\ell^2}^2,m'/2)$.
Here $\nmu{w}_{\ell^1,\mu} = \sum_{i=1}^{m'} |w_i|_\mu$ for $w = (w_i)_{i=1}^{m'}$ and $|\cdot|_\mu$ is the complex Huber function (see, e.g., \cite{neyra-nesterenko2022nestanets}). In particular, we have $\nmu{V}_{\ell^2} = 2 \sqrt{2}$ for TV minimization in 2-D.

The second part of the derivation of NESTA is finding closed-form expressions for the update steps. In general, this is only possible to do in special cases. However, NESTA considers $A$ with orthonormal rows up to a constant factor, i.e., $AA^* = \nu I$ for some $\nu > 0$. Such an assumption yields a closed form for the update formulas and is not unreasonable since many forward operators in compressive imaging have orthonormal rows. For example, with the subsampled Fourier matrix, we have $AA^* = (N/m)I$; hence, the desired property holds with $\nu = N/m$.

We reconstruct an $R \times R$ GPLU phantom image \cite{guerquin--kern2012realistic} with $R = 512$ so that the ambient dimension is $n = 512^2$. The noise $e$ is uniformly sampled from an $\ell^2$-ball of radius $\varsigma = 10^{-5}$, and so $\nmu{e}_{\ell^2} = \varsigma$. Two sampling masks are considered for the subsampled Fourier matrix $A$ and are shown in \cref{fig:TV-Fourier-masks}. The first is a near-optimal sampling scheme \cite[Sec.~4.2]{adcock2021improved}, and the second is a radial sampling scheme, where the latter is common in practice. Each mask yields approximately a 12.5\% sampling rate. For the restart scheme, the objective function is $f(z) = \nmu{Vx}_{\ell^1}$ and the feasibility gap $g_Q \equiv 0$ since NESTA always produces feasible iterates. The smoothing parameters $\mu$ are handled directly by the restarting procedure and explicitly depend on $\epsilon_{i,j,U}$ (see \cref{prop:nesterov-smoothing-Gamma}). The two main experiments were done for each of the two sampling masks. Lastly, we choose $\alpha_0 = \sqrt{m}$, $\beta_0 = 1$. The choice of $\alpha_0$ is motivated by \cite[Theorem 6.3]{colbrook2021warpd} which generalizes \cref{prop:sharpness_of_QCBP}.

\begin{figure}
\centering
\begin{minipage}[b]{1\textwidth}
\centering
\begin{overpic}[width=0.49\textwidth,trim={0mm 0mm 0mm 0mm},clip]{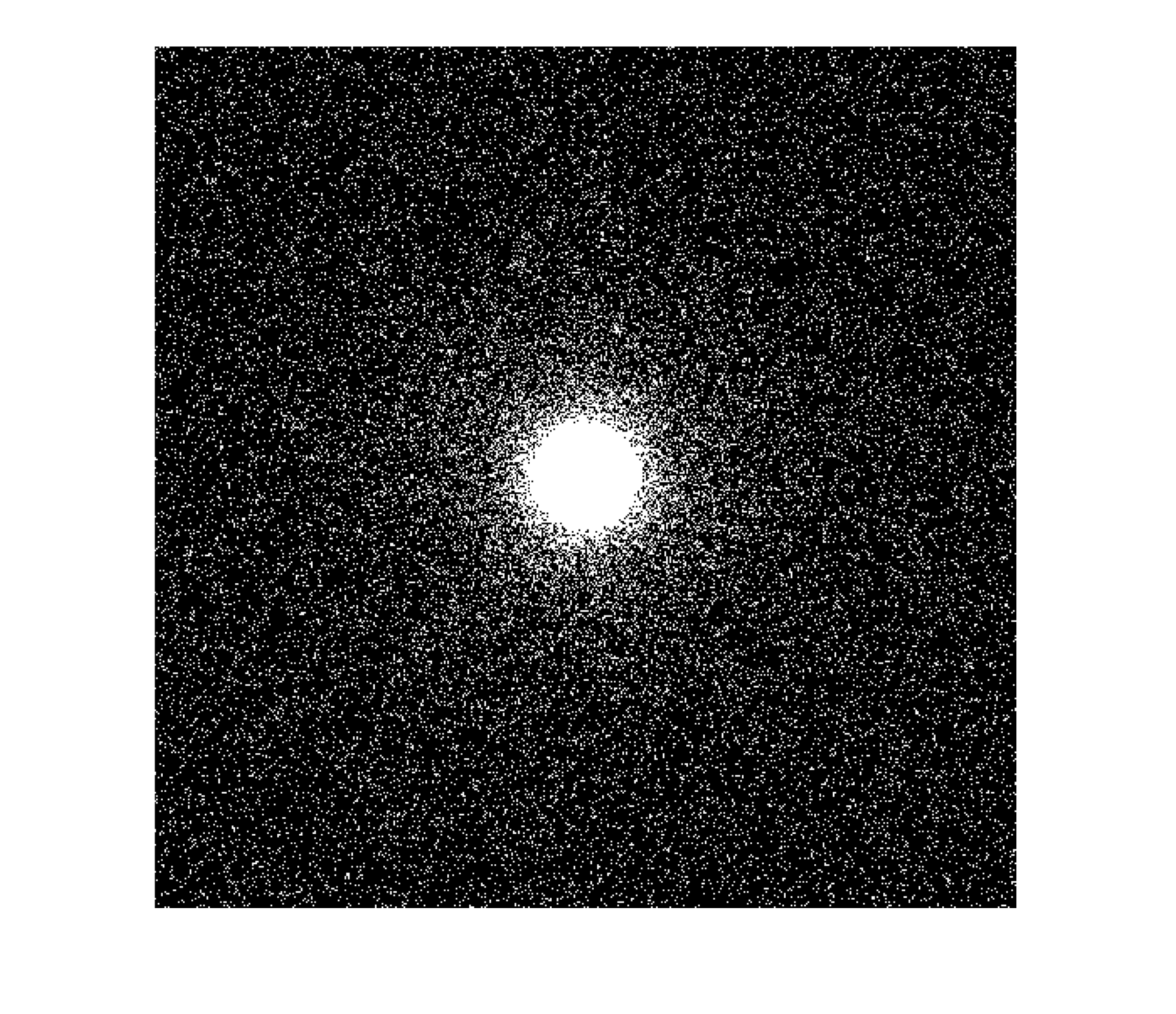}
\put (23,85) {\small Near-optimal sampling mask}
   \end{overpic}\vspace{2mm}
\begin{overpic}[width=0.49\textwidth,trim={0mm 0mm 0mm 0mm},clip]{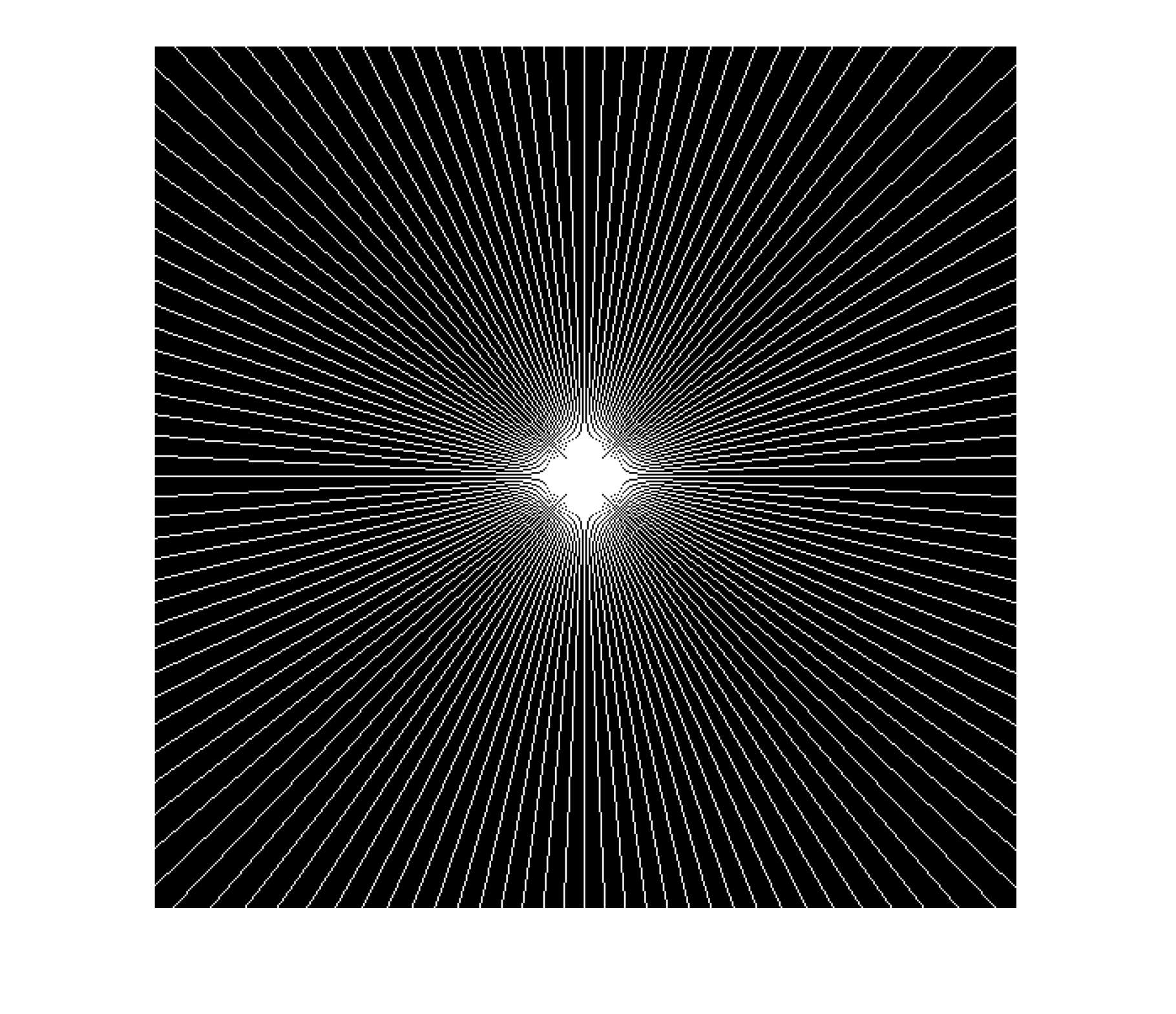}
\put (28,85) {\small Radial sampling mask}
   \end{overpic}\vspace{-5mm}
\end{minipage}\vspace{-3mm}
\caption{Sampling patterns for the Fourier measurements used in the image reconstruction experiments.}
\label{fig:TV-Fourier-masks}
\end{figure}

\subsubsection{Results}

\begin{figure}
\centering
\begin{minipage}[b]{1\textwidth}
\centering
\begin{overpic}[width=0.49\textwidth,trim={0mm 0mm 0mm 0mm},clip]{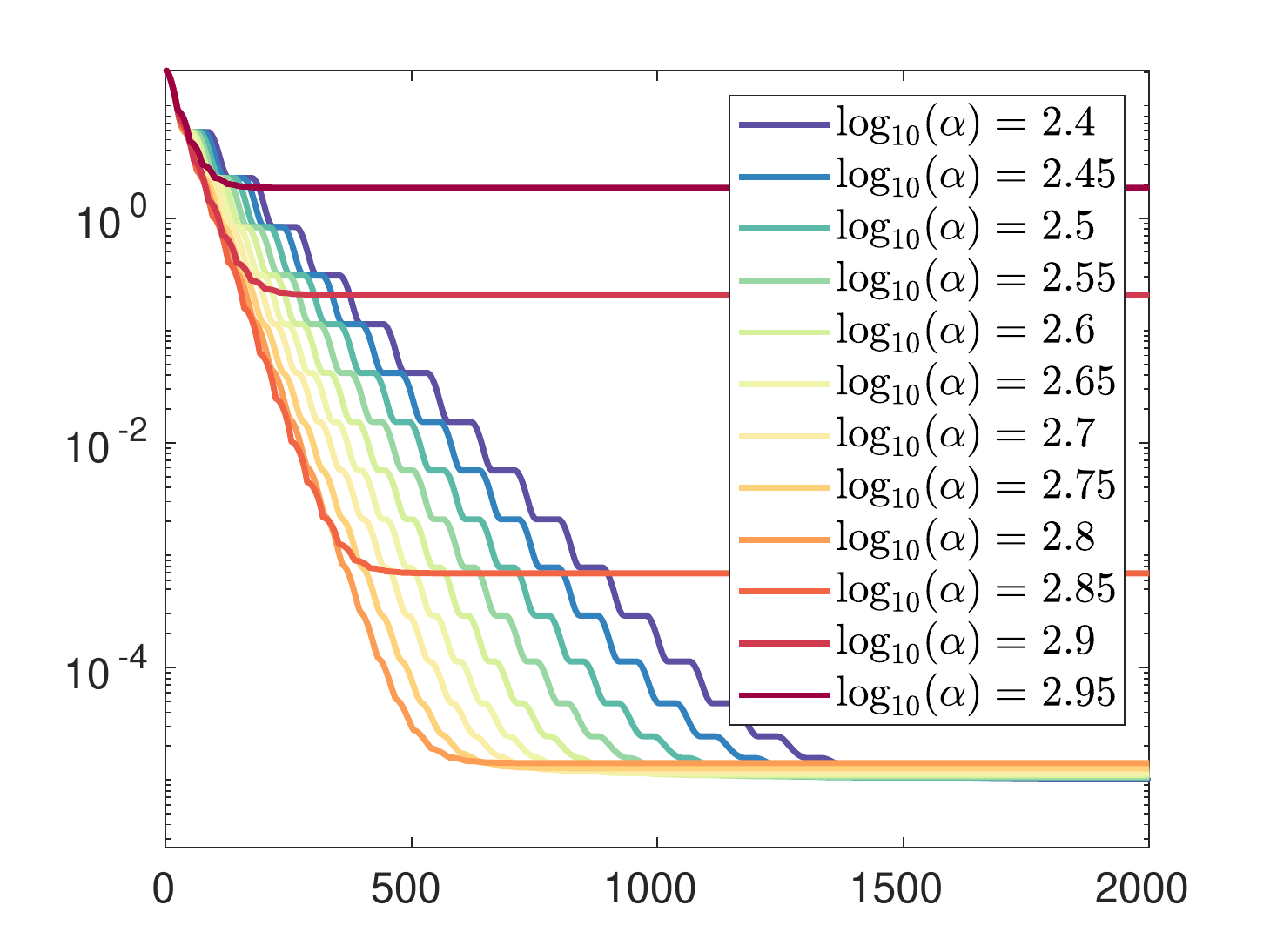}
\put (29,-2) {\small Total inner iterations $t$}
		\put (-1,27) {\small\rotatebox{90}{$\nmu{x_t - x}_{\ell^2}$}}
		\put (25,72) {\small Near-optimal sampling mask}
   \end{overpic}\vspace{2mm}
\begin{overpic}[width=0.49\textwidth,trim={0mm 0mm 0mm 0mm},clip]{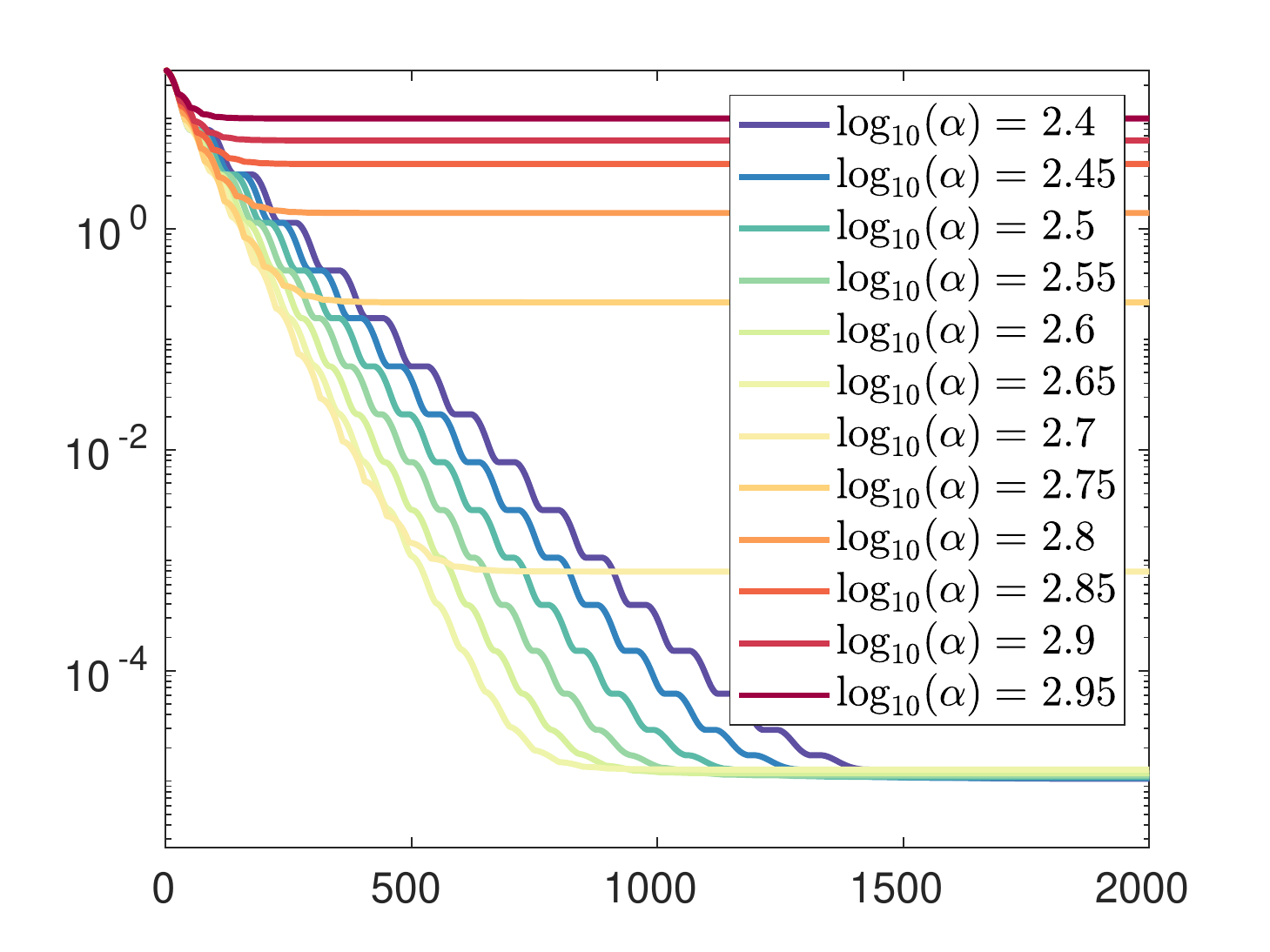}
\put (29,-2) {\small Total inner iterations $t$}
		\put (-1,27) {\small\rotatebox{90}{$\nmu{x_t - x}_{\ell^2}$}}
		\put (31,72) {\small Radial sampling mask}
   \end{overpic}\vspace{2mm}
\end{minipage}
\caption{Reconstruction error of restarted NESTA for TV minimization with $\varsigma = 10^{-5}$, and with the near-optimal and radial sampling masks, respectively. The restart scheme uses fixed sharpness constants $\beta=1$ and various $\alpha$.}
\label{fig:TVmin-Fourier-1}
\end{figure}

\begin{figure}
\centering
\begin{minipage}[b]{1\textwidth}
\centering
\begin{overpic}[width=0.49\textwidth,trim={0mm 0mm 0mm 0mm},clip]{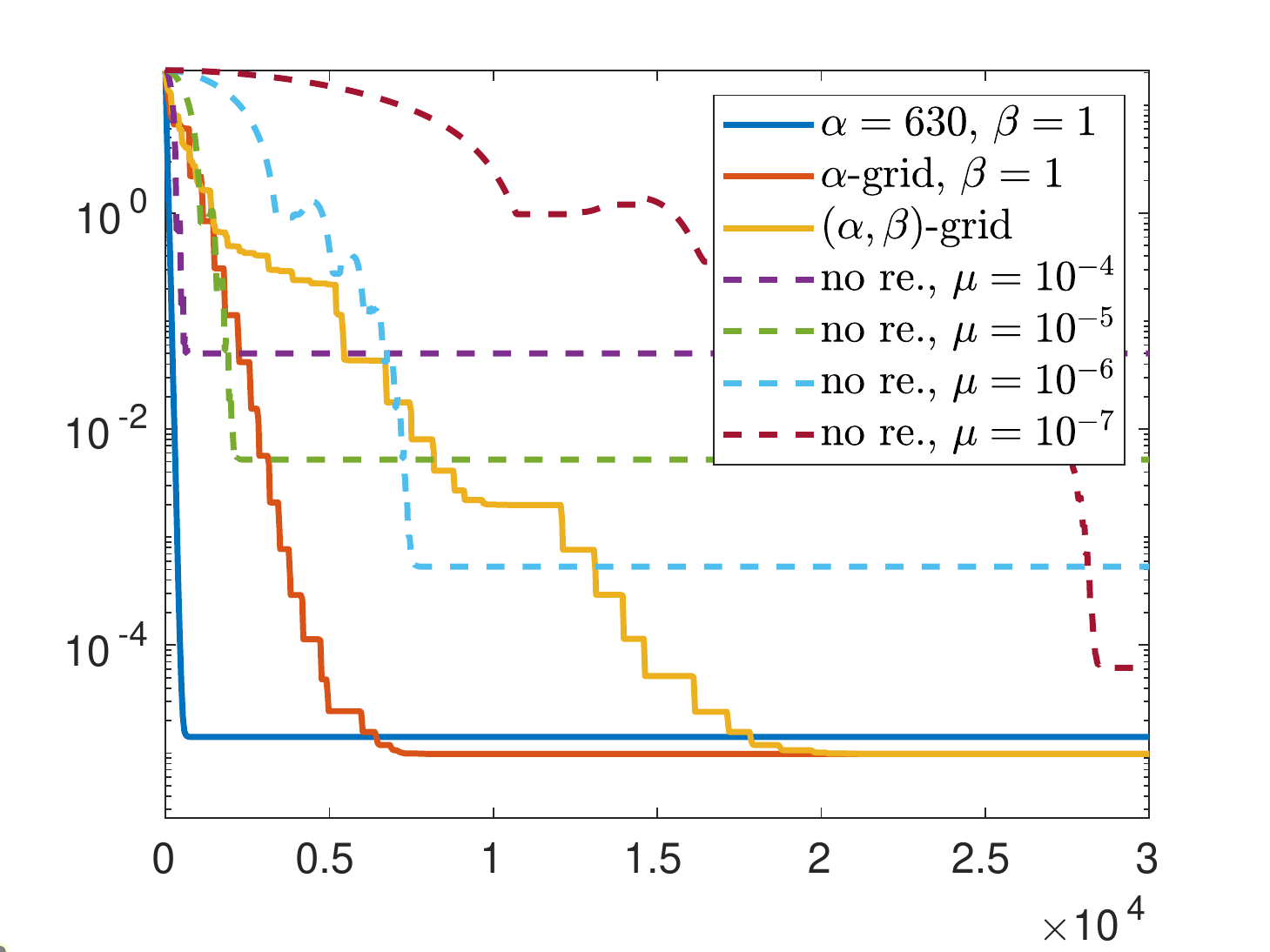}
\put (29,-2) {\small Total inner iterations $t$}
		\put (-1,27) {\small\rotatebox{90}{$\nmu{x_t - x}_{\ell^2}$}}
		\put (25,72) {\small Near-optimal sampling mask}
   \end{overpic}
\begin{overpic}[width=0.49\textwidth,trim={0mm 0mm 0mm 0mm},clip]{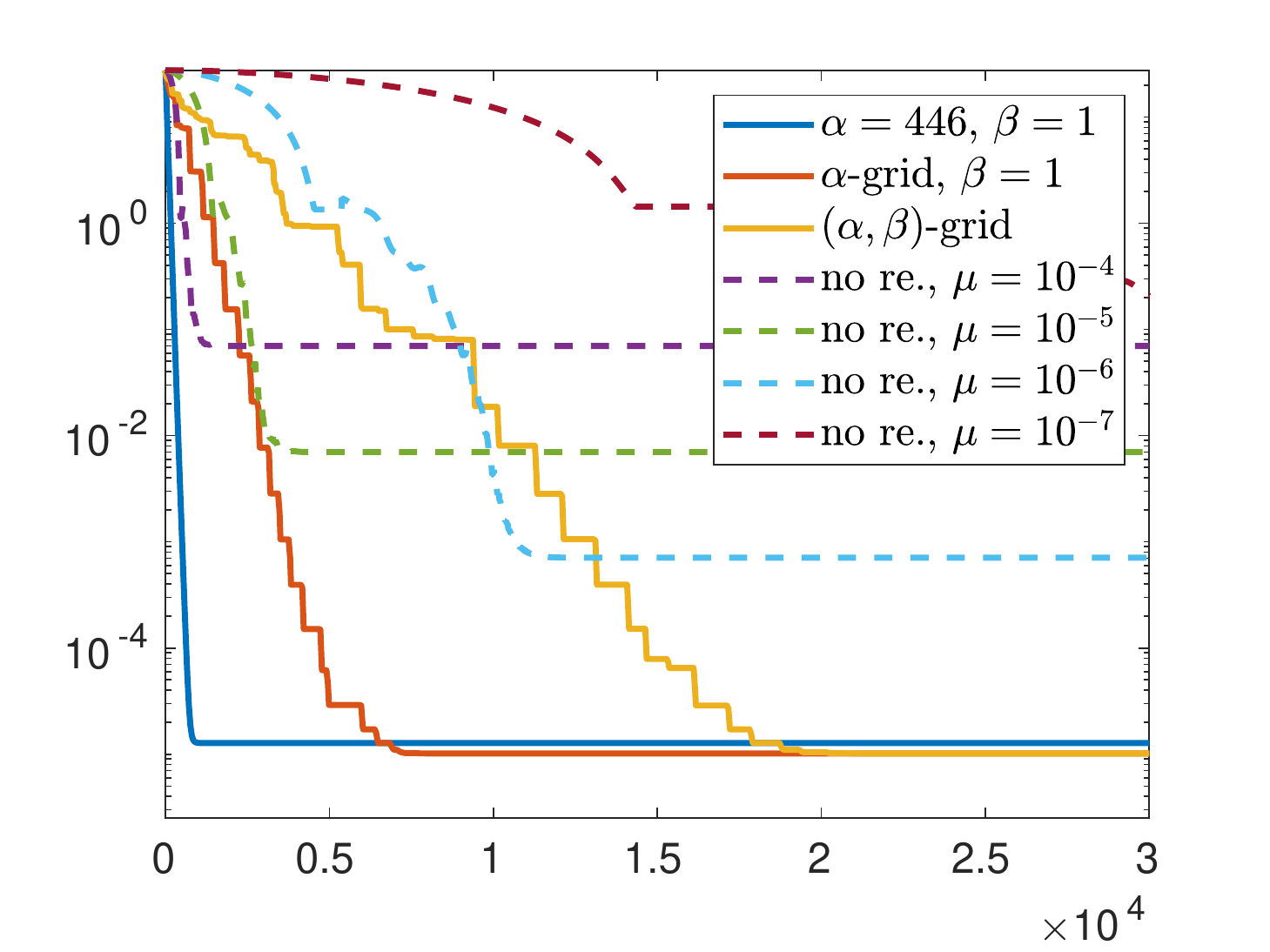}
\put (29,-2) {\small Total inner iterations $t$}
		\put (-1,27) {\small\rotatebox{90}{$\nmu{x_t - x}_{\ell^2}$}}
		\put (31,72) {\small Radial sampling mask}
   \end{overpic}
\end{minipage}
\caption{Reconstruction error of restarted NESTA for TV minimization with $\varsigma = 10^{-5}$, and with the near-optimal and radial sampling masks, respectively. Various restarted and non-restarted schemes are used.}
\label{fig:TVmin-Fourier-2}
\end{figure}

First, we run the restart scheme with fixed sharpness constants (no grid search) corresponding to pairs $(\alpha,\beta)$ with $\beta = 1$ and various $\alpha$ values. The reconstruction error versus total inner iterations is plotted in \cref{fig:TVmin-Fourier-1} with near-optimal sampling (left) and radial sampling (right). The results are very similar to the first sparse recovery via QCBP experiment. Again, the decay rate corresponds to linear decay as anticipated from \cref{thm:restart-unknown-consts-err}. The convergence rate improves as $\alpha$ increases, up to a threshold (about $\alpha = 630$ for near-optimal sampling, and about $\alpha = 446$ for radial sampling), where afterward the limiting tolerance increases steadily, yielding poor reconstruction results. This phenomenon is discussed in the first experiment of sparse recovery via QCBP. A key observation is how changing the sampling mask changes the threshold $\alpha$ value. This motivates using a grid search to avoid tuning $\alpha$ as a parameter for different sampling masks.

In the second experiment, we compare the reconstruction errors of several restart schemes and standalone NESTA (i.e., no restarts) with various smoothing parameters. This is shown in \cref{fig:TVmin-Fourier-2} with near-optimal sampling (left) and radial sampling (right). The smoothing parameters used are $\mu = 10^{i} \varsigma$, $i \in \{-2,1,0,1\}$. The results are analogous to the fourth experiment with sparse recovery via QCBP. The radial sampling mask produces marginally slower convergence rates than the near-optimal scheme. Moreover, we observe that converging to the limiting tolerance of NESTA is sensitive to the choice of smoothing parameter $\mu$. By making $\mu$ smaller, we better approximate the original problem and thus the reconstruction, but require more iterations to achieve a better approximation. In contrast, restarting NESTA via \cref{alg:restart-unknown-constsC} does not require any tuning of the smoothing parameter and outperforms the non-restarted algorithm.

\subsection{Feature selection via SR-LASSO}

Our third experiment considers feature selection via the \textit{Square Root LASSO (SR-LASSO)} problem \cite{adcock2022sparse,belloni2011square,belloni2014pivotal,geer2016estimation}. Let $X \in \bbR^{m \times n}$ be a data matrix, where each row corresponds to a data point, and each column corresponds to a feature, and $y \in \bbR^{m}$ the label vector for the data points. Since we wish to learn an \textit{affine} mapping from data points to labels, we augment $X$ by appending a new column consisting of ones, with the augmentation denoted by $A \in \bbR^{m \times (n+1)}$. Now fix $\lambda > 0$. Then we seek a vector $x \in \bbR^{n+1}$ that solves the SR-LASSO problem
$$
\min_{z \in \bbR^{n+1}} \nmu{Az-y}_{\ell^2} + \lambda \nmu{z}_{\ell^1}.
$$
An advantage of this problem over the classical LASSO is that it requires less tuning of the parameter $\lambda$ as the problem instance or noise level changes. See \cite{geer2016estimation} for discussion and recovery conditions for this problem. Feature selection is performed by identifying the indices of close-to-zero entries of $x$, which are the features to discard. This reduces the number of columns of $X$ for future processing or analysis.

The SR-LASSO is a well-known tool in high-dimensional statistics. It can also be used for sparse recovery problems, in which case approximate sharpness follows (like it did with QCBP) from the rNSP (\cref{def:rNSP}) \cite{colbrook2021warpd}. However, in the feature selection problem, properties such as the rNSP are unlikely to hold. In this case, more general recovery conditions for SR-LASSO (and LASSO), such as the \textit{compatibility condition} \cite{geer2016estimation}, are more useful. Under these conditions, one also has approximate sharpness with unknown constants.

\subsubsection{Setup}

We use the unconstrained primal-dual iterations (\cref{alg:PD_basic}) to solve SR-LASSO. We can express SR-LASSO as \eqref{PD_example_1_3} by
$$q \equiv 0, \quad g(x) = \lambda \nmu{x}_{\ell^1}, \quad h(Bx) = \nmu{Bx-y}_{\ell^2}, \quad B = A.$$
From this, the primal-dual updates can be computed explicitly. The proximal map $\tau g$ is the shrinkage-thresholding operator, and the proximal map of $\sigma h^*$ is a projection map onto the $\ell^2$-ball. In either case, the proximal maps are straightforward to compute. We compare the SR-LASSO objective error of various non-restarted and restarted schemes for three different datasets. The minimum of SR-LASSO for each dataset is computed using CVX \cite{cvx,gb08} with high precision and the SDPT3 solver and is used to compute the objective errors in \cref{fig:SRLASSO-data,fig:SRLASSO-beta}.

We use three datasets: wine quality (\texttt{wine}) \cite{misc_wine_quality_186} with $m = 6497$ points and $n = 11$ features, colon cancer (\texttt{cc}) \cite{misc_data_sets} with $m = 62$ points and $n = 2000$ features, and leukemia (\texttt{leu}) \cite{misc_data_sets} with $m = 38$ points and $n = 7129$ features. The \texttt{wine} data corresponds to a regression task of predicting wine quality, \texttt{cc} and \texttt{leu} are two-class classification tasks of diagnosing illness based on data features. We use $\lambda=3$, $2$, and $4$ for the \texttt{wine}, \texttt{cc}, and \texttt{leu} datasets, respectively. We measure sparsity $s$ of $\hat{x}$ by interpreting an entry as non-zero if its absolute value is greater than $10^{-5}$. The values $\alpha_0$ and $\beta_0$ are chosen empirically as estimates of the true sharpness constants $\alpha$ and $\beta$, respectively.

\subsubsection{Results}

\begin{figure}
\centering
\begin{minipage}[b]{1\textwidth}
\centering
\begin{overpic}[width=0.32\textwidth,trim={0mm 0mm 0mm 0mm},clip]{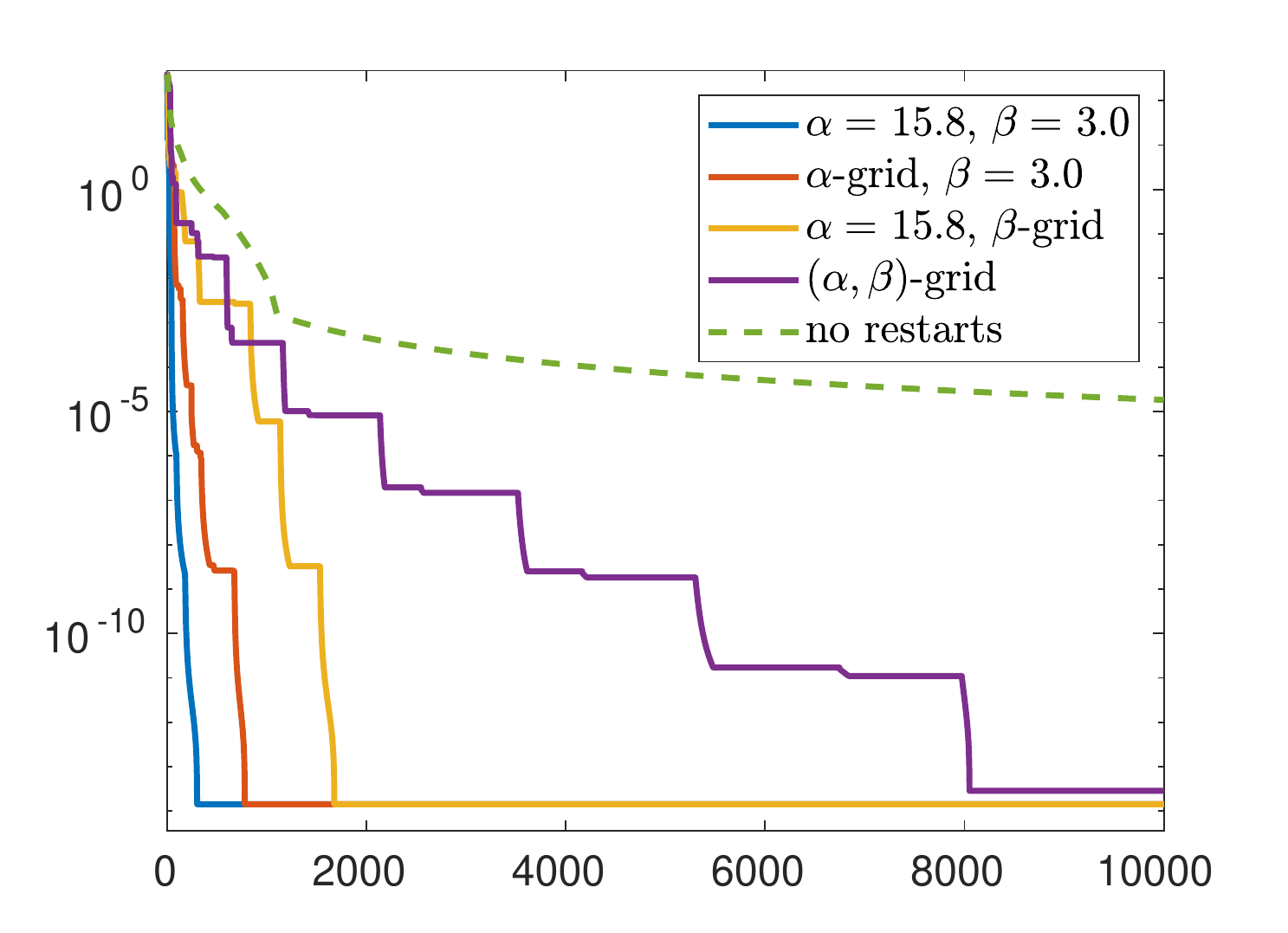}
\put (20,-4) {\small Total inner iterations $t$}
\put (-5,25) {\small\rotatebox{90}{$f(x_t) - \hat{f}$}}
\put (45,72) {\small\texttt{wine}}
\end{overpic}
\begin{overpic}[width=0.32\textwidth,trim={0mm 0mm 0mm 0mm},clip]{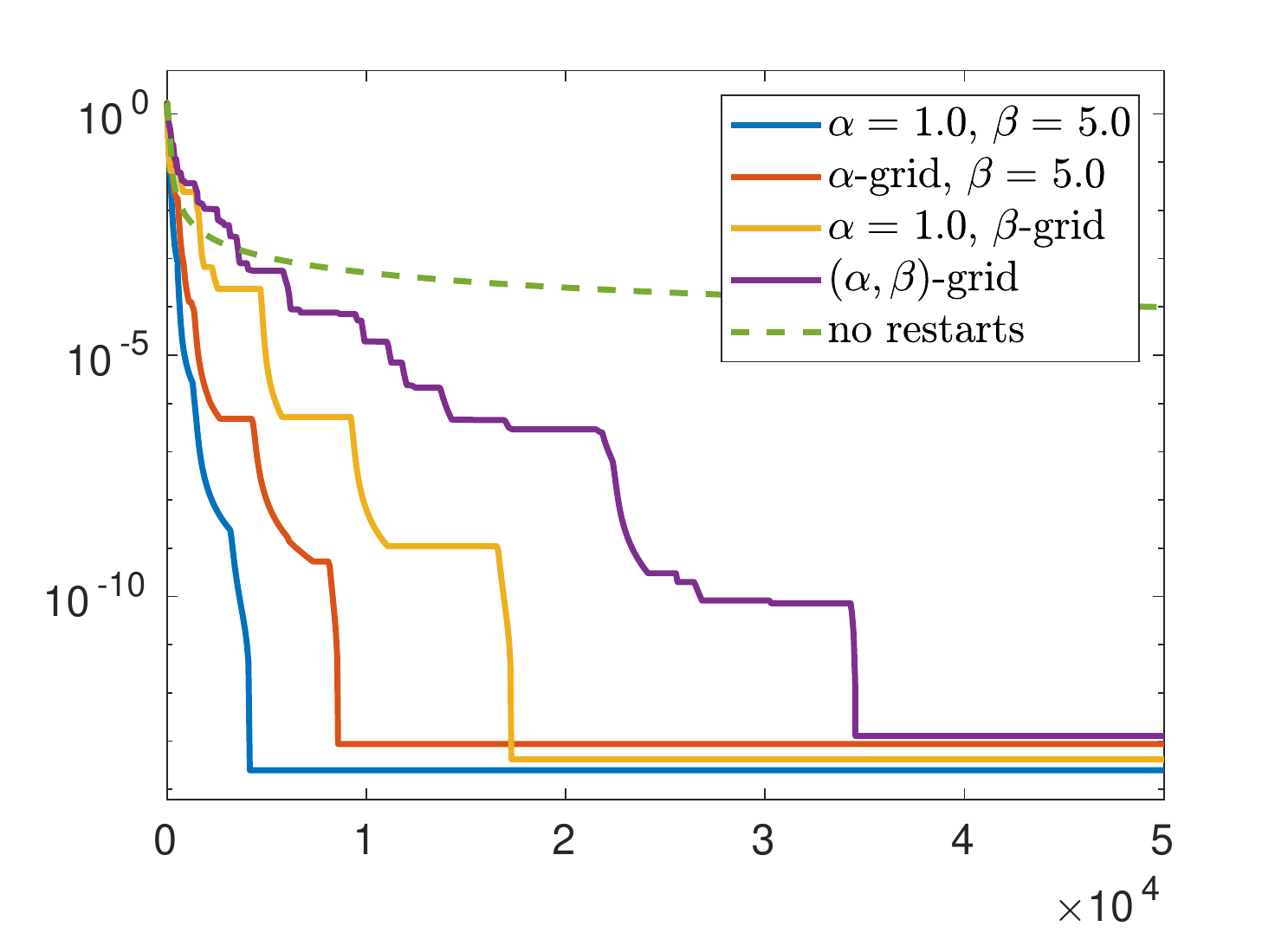}
\put (20,-4) {\small Total inner iterations $t$}
\put (-5,25) {\small\rotatebox{90}{$f(x_t) - \hat{f}$}}
\put (47,72) {\small\texttt{cc}}
\end{overpic}
\begin{overpic}[width=0.32\textwidth,trim={0mm 0mm 0mm 0mm},clip]{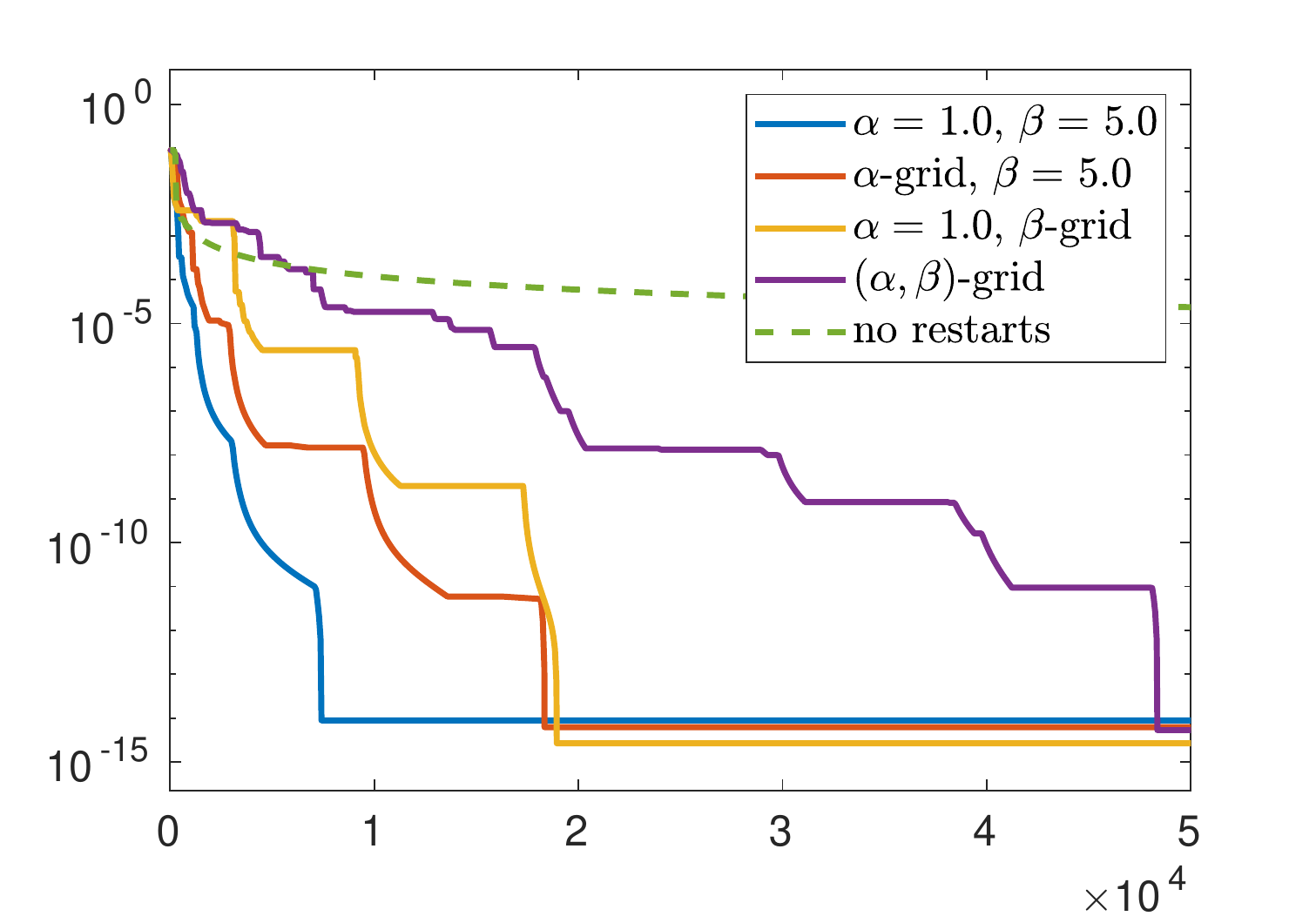}
\put (20,-4) {\small Total inner iterations $t$}
\put (-5,25) {\small\rotatebox{90}{$f(x_t) - \hat{f}$}}
\put (46,72) {\small\texttt{leu}}
\end{overpic}
\end{minipage}\vspace{2mm}
\caption{Objective error versus the total inner iteration of various (restarted and non-restarted) schemes of primal-dual iteration for SR-LASSO. The plots correspond to three different datasets.}
\label{fig:SRLASSO-data}
\end{figure}

\begin{figure}
\centering
\begin{minipage}[b]{1\textwidth}
\centering
\begin{overpic}[width=0.32\textwidth,trim={0mm 0mm 0mm 0mm},clip]{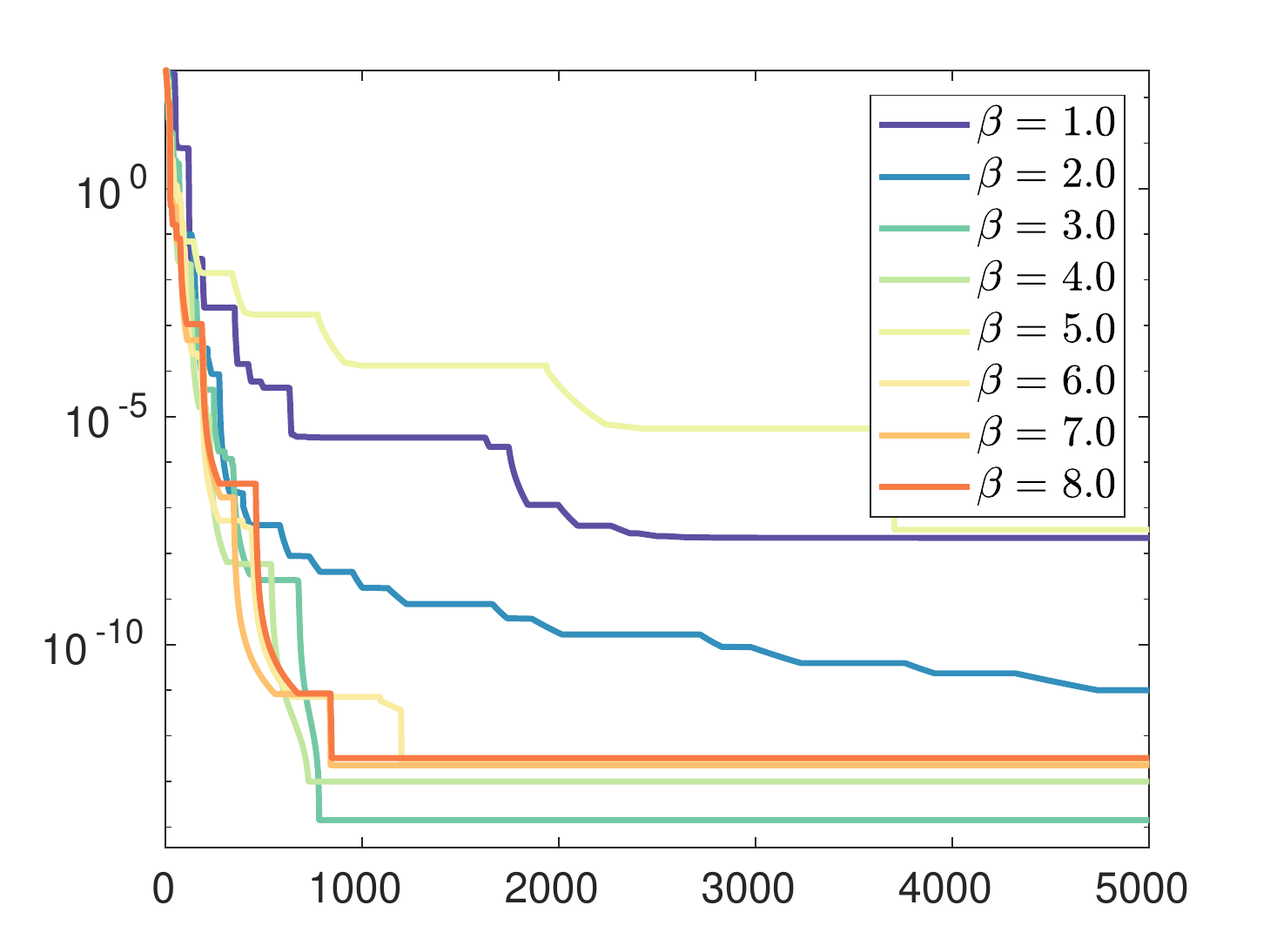}
\put (-5,25) {\small\rotatebox{90}{$f(x_t) - \hat{f}$}}
\put (45,72) {\small\texttt{wine}}
\end{overpic}
\begin{overpic}[width=0.32\textwidth,trim={0mm 0mm 0mm 0mm},clip]{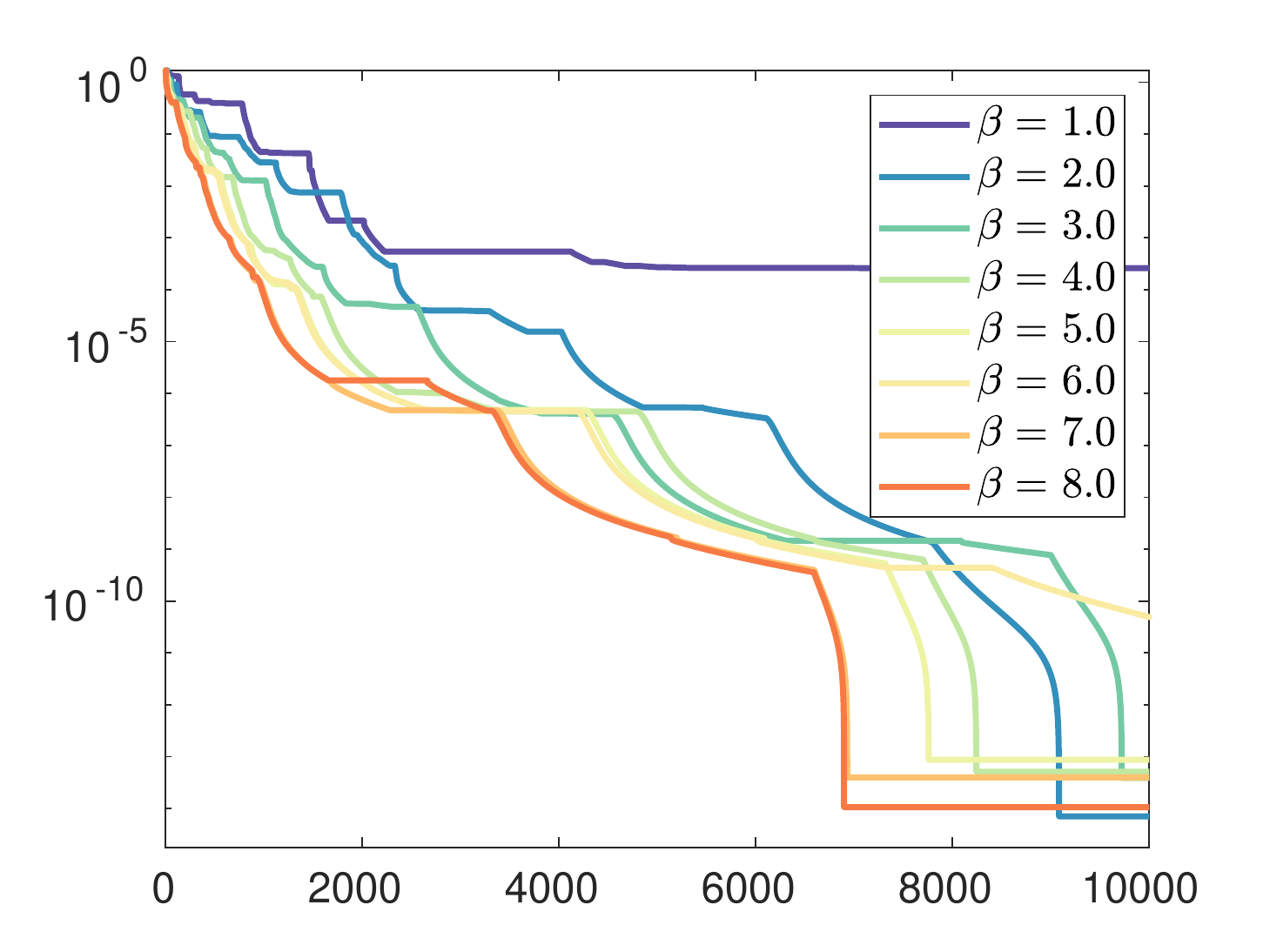}
\put (20,-4) {\small Total inner iterations $t$}
\put (-5,25) {\small\rotatebox{90}{$f(x_t) - \hat{f}$}}
\put (47,72) {\small\texttt{cc}}
\end{overpic}
\begin{overpic}[width=0.32\textwidth,trim={0mm 0mm 0mm 0mm},clip]{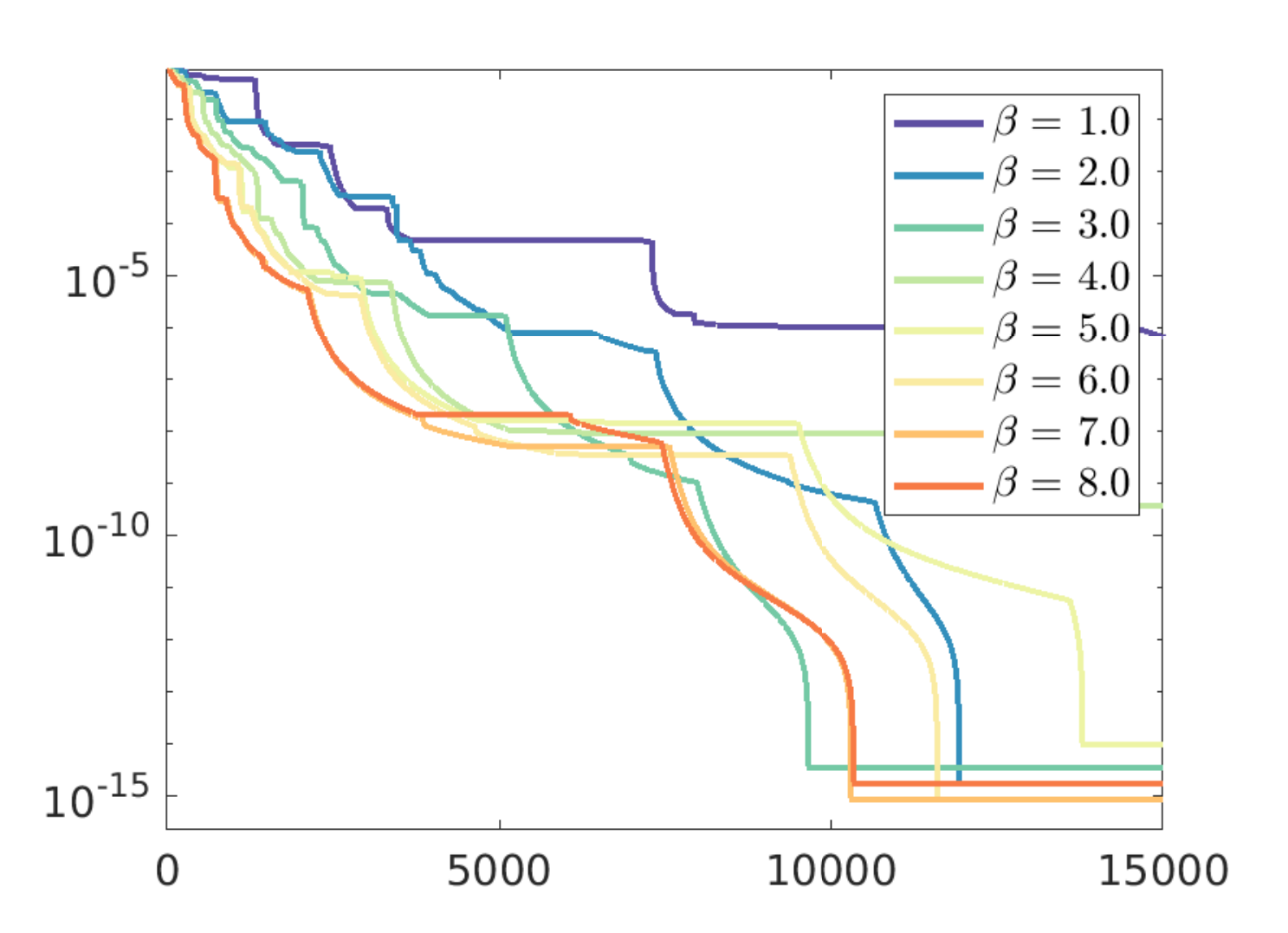}
\put (20,-4) {\small Total inner iterations $t$}
\put (-5,25) {\small\rotatebox{90}{$f(x_t) - \hat{f}$}}
\put (46,72) {\small\texttt{leu}}
\end{overpic}
\end{minipage}\vspace{2mm}
\caption{Objective error versus the total inner iteration of restarted primal-dual iteration for SR-LASSO. The plots correspond to a grid search over $\alpha$ with various fixed $\beta$ for three different datasets.}
\label{fig:SRLASSO-beta}
\end{figure}
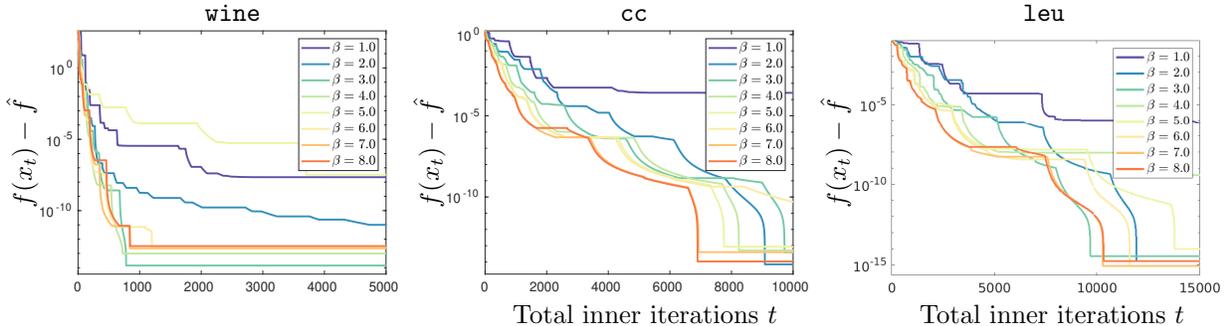

\cref{fig:SRLASSO-data} shows the performance of various restart schemes for this problem on the three datasets. In all cases, the restarted schemes outperform the non-restarted scheme. The suitable values of $\alpha$ and $\beta$ differ significantly across the datasets, indicating that the optimal sharpness parameters are problem-dependent. This is further demonstrated in \cref{fig:SRLASSO-beta}, where we show the restart scheme for various fixed $\beta$ and grid search over $\alpha$ - the restart schemes with choices of $\beta > 1$ outperform the schemes that use $\beta = 1$. This contrasts the sparse recovery example, where theory and experiment suggest $\beta = 1$ as a good choice. This phenomenon is unsurprising since the approximate sharpness condition (see \eqref{eqn:sharpness}) for this problem is expected to depend highly on the data. Nonetheless, using our grid search scheme, we obviate the need to estimate or tune these parameters.

\subsection{Comparison with the restart scheme of \cite{renegar2021simple}}
\label{sec:comparison_sec}

Finally, we compare our restart schemes with the scheme introduced in \cite{renegar2021simple}. Specifically, we consider their scheme \texttt{Sync||FOM} where every first-order method instance broadcasts its current iterate to other instances. The variant produces the best numerical results in \cite{renegar2021simple}.
The comparison is drawn using the sparse recovery problem of \cref{sec:ne-sparse-recovery-via-QCBP}.
However, \texttt{Sync||FOM} and other restart schemes in \cite{renegar2021simple} are limited to first-order methods that can only produce feasible iterates, so they cannot be used with the primal-dual iteration, as was done in \cref{sec:ne-sparse-recovery-via-QCBP}.
To proceed, we slightly modify the sparse recovery problem so that NESTA, as in \cref{sec:ne-tv-minimization}, can be used instead.

Consider the sparse recovery problem from \cref{sec:ne-sparse-recovery-via-QCBP}, but where the measurement matrix $A \in \bbC^{m \times n}$ is now a subsampled Fourier matrix.
Specifically, $A$ has the form $A = m^{-1/2} P_\Omega F$ where $F \in \bbC^{n \times n}$ is the 1-D discrete Fourier transform and $\Omega \subseteq [n]$ is a sampling mask with $\abs{\Omega} = m$.
We construct the mask $\Omega$ by including each row as an i.i.d.\ Bernoulli random variable with probability of success equal to $m/n$. The expected value of $\abs{\Omega}$ is $m$.
The matrix $A$ can then be shown to have the rNSP with high probability (see, for instance, \cite[Lem.\ 3.2.1]{neyra-nesterenko2023unrolled}).
In turn, the approximate sharpness condition \eqref{eqn:sharpness} holds for the modified sparse recovery problem with high probability.

\subsubsection{Experimental setup}

Using NESTA from \cref{sec:ne-tv-minimization} to solve QCBP (\cref{sec:ne-sparse-recovery-via-QCBP}), we set $W = I$.
The subsampled Fourier matrix $A \in \bbC^{m \times n}$ satisfies the orthonormal row condition with $AA^* = (\sqrt{n}/m) I$.
This is enough for NESTA to be applicable.
The parameters used are ambient dimension $n = 128$, sparsity level $s = 15$, number of measurements $m = 60$, and noise level $\varsigma = 10^{-6}$.
The ground truth vector $x$ is sparse with $s$ of its entries randomly selected as i.i.d.\ standard normal entries. The noise vector $e$ is selected uniformly random on the $\ell^2$-ball of radius $\varsigma$ and so $\nmu{e}_{\ell^2} = \varsigma$.

In terms of restart schemes developed in this paper, the objective function is $f(x) = \nmu{x}_{\ell^1}$ and the feasibility gap can be set as $g_Q \equiv 0$ since NESTA always produces feasible iterates.
Again, the smoothing parameters $\mu$ are changed directly by the restarting procedure and explicitly depend on $\epsilon_{i,j,U}$. 
Lastly, $\alpha_0 = \sqrt{m}$, $\beta_0 = 1$. The choice of $\alpha_0$ is motivated by \cref{prop:sharpness_of_QCBP} as before.

Regarding \texttt{Sync||FOM}, the code was transcribed from the Julia implementation in \cite{renegar2021simple} into MATLAB. 
The objective error tolerance $\epsilon$ is a parameter in this scheme, specifically for the number of parallel instances created, equal to $N = \max(2, \log_2(1/\epsilon))$.
The smoothing parameters $\{\mu_k\}$ corresponding to $k = 1, \dots, N$ instances depend on each instance's tolerance $\epsilon_k$, where $\mu_k = \epsilon_k / n$ so that instance $k$ can achieve an objective error within $\epsilon_k$. The specific choice of $\mu_k$ is informed by \cref{prop:nesterov-smoothing-Gamma}.
Moreover, we track reconstruction and objective errors using the first-order method's iterations.
Specifically, each time \texttt{Sync||FOM} calls the first-order method, the iterate returned is kept (and the errors are computed) if it produces a lower objective function value than the previously kept iterate. Otherwise, the previous iterate is used to compute the errors.

\subsubsection{Results}

\begin{figure}
\centering
\begin{minipage}[b]{1\textwidth}
\centering
\begin{overpic}[width=0.49\textwidth,trim={0mm 0mm 0mm 0mm},clip]{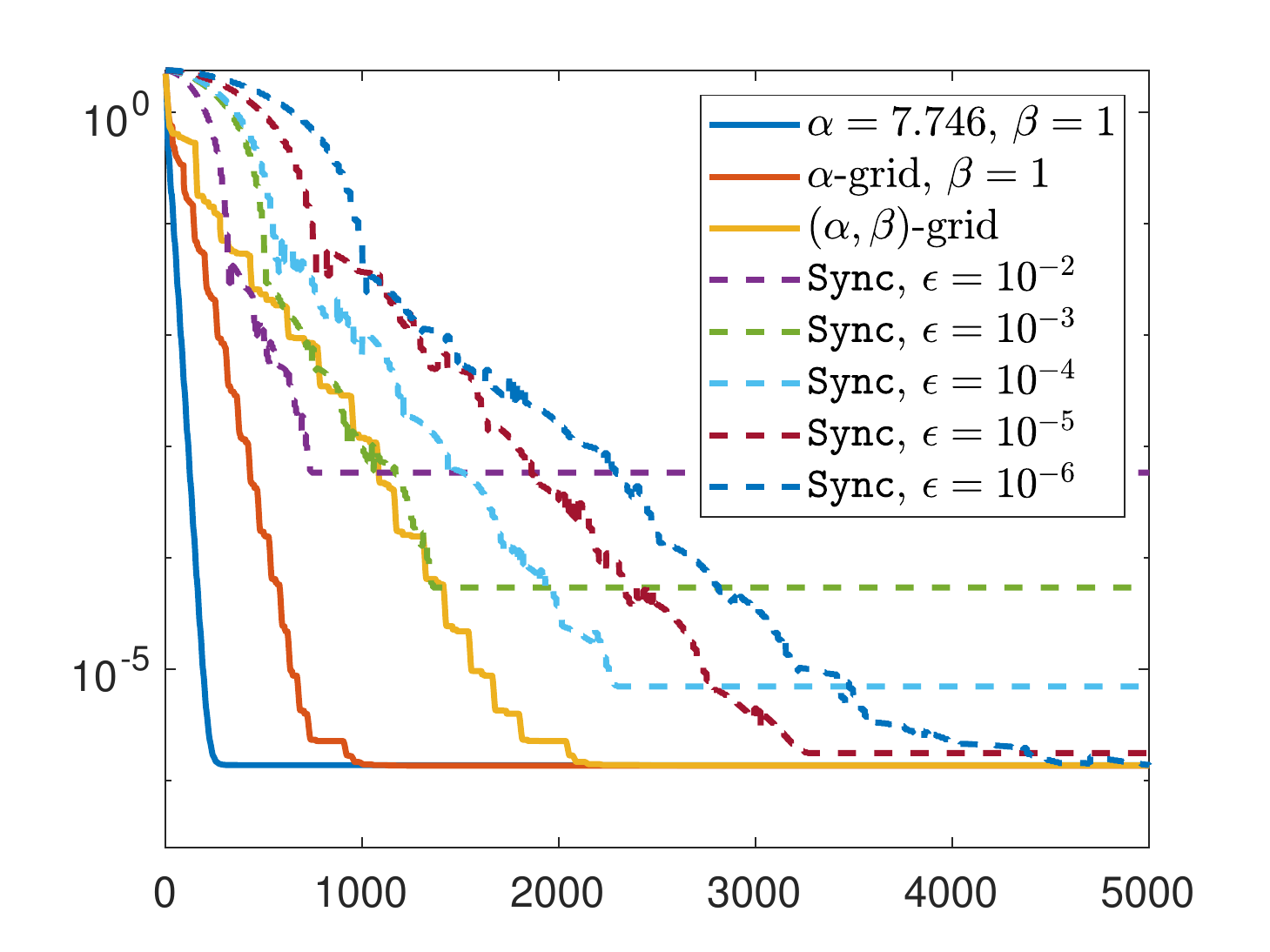}
\put (29,-2) {\small Total inner iterations $t$}
		\put (-1,27) {\small\rotatebox{90}{$\nmu{x_t - x}_{\ell^2}$}}
   \end{overpic}
\begin{overpic}[width=0.49\textwidth,trim={0mm 0mm 0mm 0mm},clip]{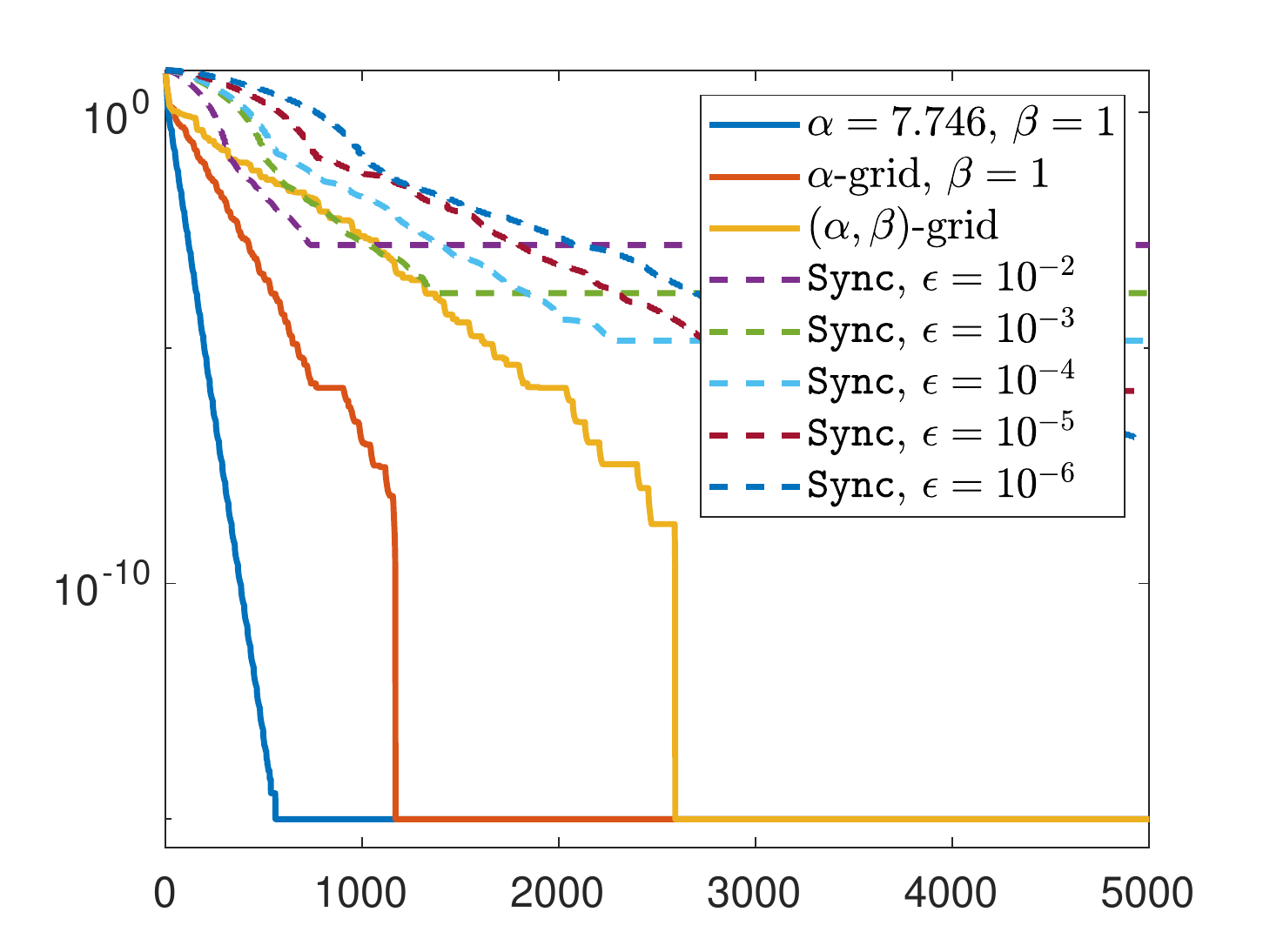}
\put (29,-2) {\small Total inner iterations $t$}
		\put (-1,27) {\small\rotatebox{90}{$f(x_t) - \hat{f}$}}
   \end{overpic}
\end{minipage}
\caption{Reconstruction error (left) and objective error (right) of restarted NESTA for QCBP with $\varsigma = 10^{-6}$. Various restart schemes are used to compare with Renegar and Grimmer's \texttt{Sync||FOM} restart scheme.}
\label{fig:RG-comparison}
\end{figure}

\begin{figure}
\centering
\begin{minipage}[b]{1\textwidth}
\centering
\begin{overpic}[width=0.49\textwidth,trim={0mm 0mm 0mm 0mm},clip]{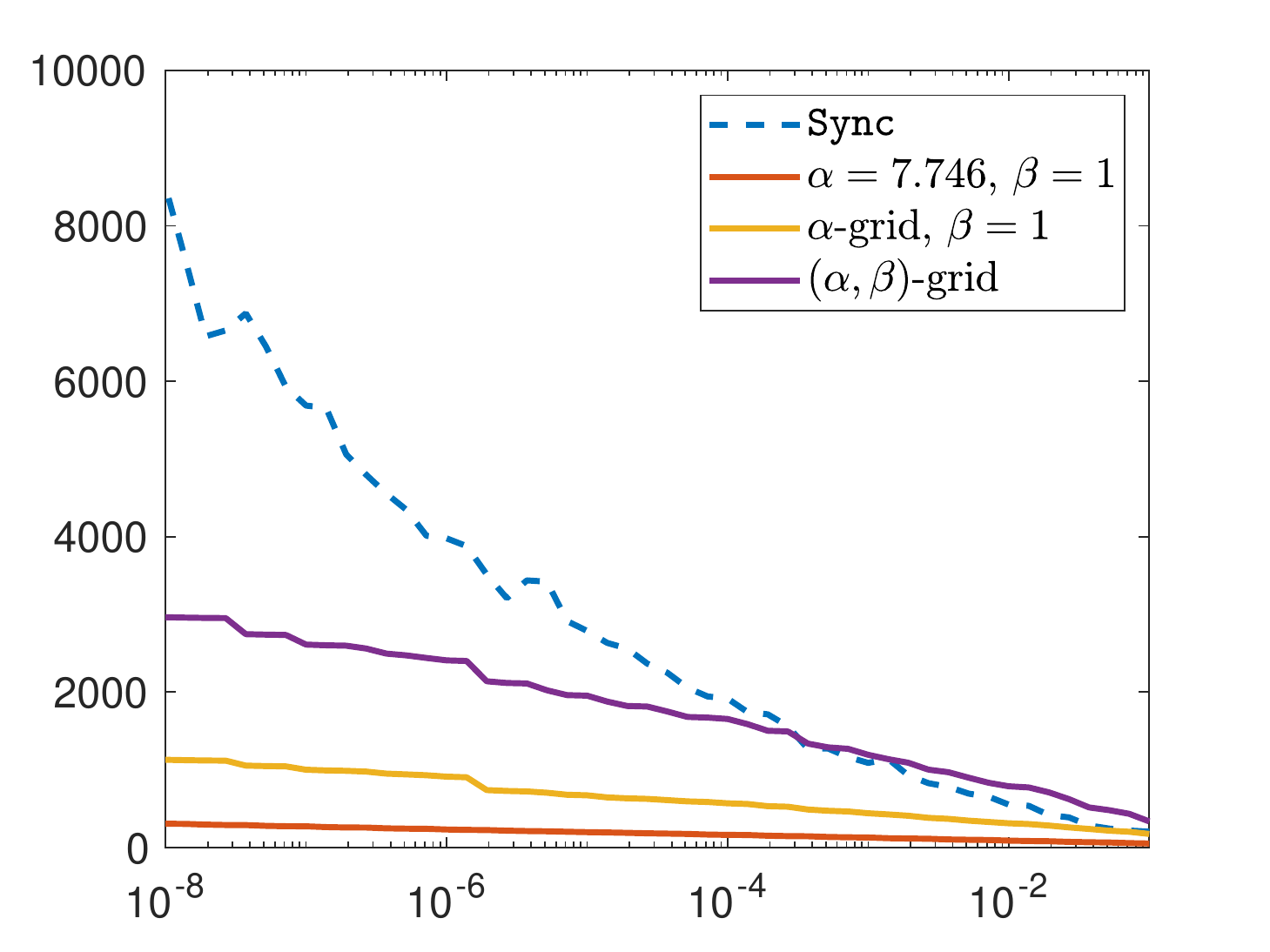}
\put (50,-2) {$\epsilon$}
        \put (-1,17) {\small\rotatebox{90}{\small Total inner iterations $t$}}
    \end{overpic}
\end{minipage}
\caption{Number of inner iterations needed to reach $\epsilon$ objective error using restarted NESTA for QCBP with $\varsigma = 10^{-6}$. Various restart schemes and Renegar and Grimmer's \texttt{Sync||FOM} are compared.}
\label{fig:RG-iters-vs-epsilon}
\end{figure}

\cref{fig:RG-comparison} shows the results of the experiment, where we compare the performance of several of our restart schemes and \texttt{Sync||FOM} with tolerance levels $\epsilon = 10^{-i}$, $i \in \{2,3,4,5,6\}$. Performance is measured in terms of reconstruction error and objective error versus total iteration.

As is evident from this figure, our restart schemes with fixed (optimized) $(\alpha,\beta)$ and with a grid search over $\alpha$ (with $\beta = 1$, following the theory of Section \ref{sec:ne-sparse-recovery-via-QCBP}) both outperform \texttt{Sync||FOM} in reconstruction speed and achieving reconstruction up to the tolerance $\eta$, a quantity proportional to the noise level $\varsigma$. The grid search over $\alpha$ and $\beta$ also achieves this tolerance and manages to do so more quickly than \texttt{Sync||FOM} instances that achieve the same tolerance.
\texttt{Sync||FOM} decreases in performance but achieves a smaller limiting tolerance as $\epsilon$ is made smaller.
This indicates an optimal choice of $\epsilon$ depending on the tolerance $\eta$, which is generally unknown.
A key advantage of our restart schemes is that they do not require knowledge of $\eta$.
Note that similar remarks can be made when examining objective error.
In addition, we expect that after enough iterations \texttt{Sync||FOM} with parameter $\epsilon$ achieves an objective error within $\epsilon$. This is precisely what is observed in \cref{fig:RG-comparison} (right).

Lastly, \cref{fig:RG-iters-vs-epsilon} compares how many inner iterations are needed for each restart scheme to achieve an objective error of $\epsilon$. The QCBP problem falls within row 2 of Table \ref{rates_table} with $\beta = 1$. Hence, our scheme should achieve an error of $\epsilon$ using $\ord{\log(1/\epsilon)}$ iterations. We see exactly this behavior for all three versions of our scheme. Conversely, for \texttt{Sync||FOM}, the number of iterations scales like $\log^2(1/\epsilon)$. This is exactly as shown in \cite{renegar2021simple} (see Corollary 5 and the following discussion), which is worse than the performance of our schemes.

\section{Conclusion}\label{sec:conclusion}

We have developed a framework that accelerates first-order methods under approximate sharpness conditions. These conditions generalize traditional sharpness definitions by incorporating an unknown constant perturbation into the objective error, offering greater robustness (e.g., to noise or model classes). Our scheme achieves optimal convergence rates for a wide variety of problems without requiring prior knowledge of the constants appearing in \eqref{eqn:sharpness}. Additionally, our method does not necessitate that the first-order methods produce feasible iterates, adding a layer of flexibility beneficial for techniques such as primal-dual iterations. Our numerical experiments demonstrate that our schemes are practical and often result in significant performance enhancements compared to non-restarted schemes or restart schemes with suboptimal parameter selections.

There are numerous possible avenues for future research and extensions of our framework. One potential area involves adapting the metric in \eqref{eqn:sharpness} to a Bregman distance and acceleration for convex optimization problems within Banach spaces. Another exciting direction is the application of our methods to non-convex bilevel optimization schemes. Additionally, developing a generic framework that utilizes sharpness-type bounds for saddle point problems presents an interesting challenge. For instance, in saddle-point problems such as \eqref{PD_example_1_2} and \eqref{PD_example_1_4}, it might be feasible to design restart schemes that use primal-dual gaps instead of $f(x) - \hat{f}$ in \eqref{eqn:sharpness} - see \cite{applegate2022faster} and \cite{fercoq2022quadratic} for primal-dual gap sharpness and restart schemes in the cases of $\beta=1$ and $\beta=2$, respectively. See also \cite{kerdreux2022restarting,kerdreux2019restarting} for recent work on restarts based on gap functions for Frank-Wolfe algorithms. Another promising research area involves linking $d_1$ and $d_2$ to rates and condition numbers in scenarios of approximate sharpness, a topic previously explored for sharpness in \cite{roulet2020sharpness}. Lastly, we envision extending our restart schemes to accommodate stochastic first-order methods, which could significantly impact larger-scale machine-learning problems.

\appendix

\section{Further optimal choices of parameters}

In this appendix, we derive optimal choices of parameters for our algorithms.

\subsection{The optimal choice of $r$ in \cref{alg:restart-known-consts}}
\label{sec:bestr}

Suppose that $d_2= d_1/\beta$ and
$$
\left\lceil\frac{\log(\epsilon_0 /\varepsilon)}{\log(1/r)} \right\rceil\leq 2\frac{\log(\epsilon_0 /\varepsilon)}{\log(1/r)}.
$$
Using this new bound instead, the total number of iterations $T$ performed by $\Gamma$ is bounded by
$$
T\leq \left\lceil\frac{\log(\epsilon_0 /\varepsilon)}{\log(1/r)} \right\rceil+\frac{C2^{d_1/\beta+1}}{\alpha^{d_1/\beta} }{\log(\epsilon_0 /\varepsilon)} \frac{r^{-d_2}}{\log(1/r)}.
$$
Hence $T$ is bounded by an $\varepsilon$-dependent constant times $r^{-d_2}/\log(1/r)$, which can be minimized analytically by choosing $r = \E^{-1/d_2}$. The optimal $r$ here does not depend on the approximate sharpness constants. Therefore, one has
\eas{
T \leq \left\lceil{d_2\log(\epsilon_0 /\varepsilon)} \right\rceil+\frac{C \E d_2 2^{d_1/\beta+1}}{\alpha^{d_1/\beta} }{\log(\epsilon_0 /\varepsilon)}
}
This is meaningful in choosing one less parameter, namely $r$ for \cref{alg:restart-known-consts}.

An optimal value of $r$ can also be found for the case $d_2>d_1/\beta$. However, this optimal value depends on $\varepsilon$ in a somewhat complicated manner. In the limit $\varepsilon\downarrow 0$, the optimal choice is
$$
r=\left(\frac{d_2}{2d_2-d_1/\beta}\right)^{\frac{1}{d_2-d_1/\beta}},
$$
which does depend on the sharpness constant $\beta$. As $d_2-d_1/\beta\downarrow 0$, this choice converges to the choice $r = \E^{-1/d_2}$, that is obtained when $d_2=d_1/\beta$. Similarly, if $d_2<d_1/\beta$, then the optimal choice depends on $\varepsilon$ in a complicated manner but converges to the choice $r = \E^{-1/d_2}$ as $d_2-d_1/\beta\uparrow 0$.

In any of these cases, the same argument for optimal $r$ applies to the algorithms in \cref{sec:restart-unknown-constants}. In the case that $\beta$ is unknown, we recommend the choice $r = \E^{-1/d_2}$.

\subsection{How to choose $a$, $b$}\label{rem:choice-of-ab}
In the case of \cref{known_beta} and assuming \eqref{growth_cond_c}, we can select an optimal value of $a$. From \cref{known_beta} and $\alpha_*\geq \alpha/a$, the part of $\tau$ that depends on $a$ is bounded by $\mathcal{O}({(| \lfloor \log_a (\alpha/\alpha_0) \rfloor | + 1)^{c_1} a^{d_1/\beta}})$. We can upper bound this further by dropping the floor function and, then, dropping the $+1$ in brackets. We are then led to minimizing
$$
|\log_a (\alpha/\alpha_0)|^{c_1}a^{d_1/\beta}=|\log(\alpha/\alpha_0)|^{c_1}a^{d_1/\beta}/\log(a)^{c_1}.
$$
Under these assumptions, the optimal value of $a$ is $\E^{c_1\beta/d_1}$. In the case of \cref{known_alpha}, there is no clear optimal choice for $b$ since the optimal choice is $\varepsilon$-dependent.

\subsection{How to choose $c_1$, $c_2$}
For \cref{known_beta,rad_ord_cor1}, an optimal choice of $c_1 > 1$ exists, but it depends on the unknown parameter $\alpha$. To see this, minimize the lower bound of $t$ in the aforementioned corollaries with respect to $c_1$, noting that the only term in $\tau$ that depends on $c_1$ is $(| \lfloor \log_a (\alpha/\alpha_0) \rfloor | + 1)^{c_1}$. Hence, we must minimize
$
(| \lfloor \log_a (\alpha/\alpha_0) \rfloor | + 1)^{c_1}c_1/(c_1-1).
$
Assuming $\alpha_0 \neq \alpha$, differentiating and finding a minima gives 
$$c_1 = \frac{1 + \sqrt{1 + \frac{4}{\log \left(| \lfloor \log_a (\alpha/\alpha_0) \rfloor | + 1 \right) }}}{2}.$$
By the same reasoning, for \cref{known_alpha,rad_ord_cor1} and $\beta_0 \neq \beta$, the optimal choice of $c_2 > 1$ depends on the unknown parameter $\beta$ and is given by
$$c_2 = \frac{1 + \sqrt{1 + \frac{4}{\log \left(| \lceil \log_b (\beta/\beta_0) \rceil | + 1 \right) }}}{2}.$$
Intuitively, if $\alpha_0$ is far from $\alpha$ then $c_1$ should be closer to $1$, and similarly for $\beta_0$ and $\beta$ regarding $c_2$. Without prior knowledge, we recommend a sensible default such as $c_1 = c_2 = 2$.

\section{Miscellaneous proofs}

In this appendix, we prove several results that were stated in \cref{sec:FOM_examples}.

\subsection{Nesterov's method with smoothing}
\label{sec_append_nestarov}

\begin{proof}[Proof of \cref{lem:nesterovs-thm-smoothable-f}]
Applying \cref{lem:Nesterov's theorem} with the function $f_\mu$ and using the second part of \cref{def:uv-smoothable} gives
$$f_\mu (x_k) - f_\mu (x) \leq \frac{4 u p(x;x_0)}{\mu k(k+1) \sigma_p}.$$
Now, using both inequalities in the first part of \cref{def:uv-smoothable} gives the result.
\end{proof}

\begin{proof}[Proof of \cref{prop:nesterov-smoothing-Gamma}]
Suppose that $x_0 \in Q$ with $d(x_0, \widehat{X}) \leq \delta$. Then by \cref{lem:nesterovs-thm-smoothable-f} with $\hat{x} \in \widehat{X} \subseteq Q$, we have
$$f(x_N) - \hat{f} \leq \frac{4 u p(\hat{x};x_0)}{ \mu N(N+1) \sigma_p} + v \mu . $$
Using $\frac{1}{N(N+1)} \leq \frac{1}{N^2}$, $\sigma_p = 1$ and $p(\hat{x}) \leq \frac{1}{2} \delta^2$ by choice of $p$, we get
$$
f(x_N) - \hat{f} \leq \frac{2 u \delta^2 } {\mu N^2} + v \mu.
$$
Substituting $\mu = \frac{\epsilon}{2 v}$ and using that $N \geq 2 \sqrt{2uv} \cdot \frac{\delta}{\epsilon}$ gives the result.
\end{proof}

\subsection{Primal-dual iterations for unconstrained problems}
\label{proofs_PD_appendix}

\begin{proof}[Proof of \cref{PD_lemma_1}]
We use \eqref{saddle_bound_WARPd1} and prove bounds on each of the terms on the left-hand side.
First, we have
$$
\mathcal{L}\left(X_k,y\right)=\langle BX_k,y\rangle_{\mathbb{R}}+q(X_k)+g(X_k)-h^*(y).
$$
Since $h$ is convex and lower semicontinuous, $h^{**}=h$. It follows that
$$
h(BX_k)=\max_{y\in\mathbb{C}^{m}} \langle BX_k,y\rangle_{\mathbb{R}}-h^*(y)=-\min_{y\in\mathbb{C}^{m}}(h^*(y)-\langle BX_k,y\rangle_{\mathbb{R}}).
$$
The objective function is convex and lower semicontinuous, and the set of minimizers is $y$ such that
$$
0 \in \partial \left ( h^*(\cdot) - \ip{\cdot}{B X_k} \right )(y) = \partial h^*(y) - B X_k.
$$
Rearranging and using the Legendre--Fenchel identity, we deduce that this set of minimizers is precisely $\partial h(B X_k)$. It follows that
\begin{equation}
\label{bdbd1}
\mathcal{L}\left(X_k,y\right)=f(X_k), \qquad \forall y \in \partial h(B X_k).
\end{equation}
Second, we have
$$
\mathcal{L}\left(x,Y_k\right)=\langle Bx,Y_k\rangle_{\mathbb{R}}+q(x)+g(x)-h^*(Y_k).
$$
The above argument shows that
$$
h(Bx)=\max_{y\in\mathbb{C}^{m}} \langle Bx,y\rangle_{\mathbb{R}}-h^*(y)\geq \langle Bx,Y_k\rangle_{\mathbb{R}}-h^*(Y_k).
$$
It follows that
\begin{equation}
\label{bdbd2}
\mathcal{L}\left(x,Y_k\right)\leq f(x).
\end{equation}
The bound \eqref{objective_bound_WARPd1} now follows by combining \eqref{bdbd1} and \eqref{bdbd2}.
\end{proof}

\begin{proof}[Proof of \cref{PDGH_prop_1}]
First, consider general $\tau,\sigma>0$ with $\tau(\sigma L_B^2+L_q)= 1$. For input $x_0$ with $d(x_0,\widehat{X})\leq\delta$, \eqref{PDGH_prop_1_assumption} and \eqref{objective_bound_WARPd2} imply that for $x\in \widehat{X}$,
$$
f(X_N)-\hat{f}\leq \frac{1}{N}\left(\frac{\delta^2}{\tau}+\frac{L_h^2}{\sigma}\right)=\frac{1}{N}\left(\sigma\delta^2L_B^2+\frac{L_h^2}{\sigma}+\delta^2L_q\right).
$$
Choosing the step size $\sigma>0$ to minimize the right-hand side leads to
$$
\sigma=\frac{L_h}{\delta L_B},\quad \tau=\frac{\delta}{L_BL_h+\delta L_q},\quad f(X_N)-\hat{f}\leq \frac{\delta}{N}\left(2L_BL_h+\delta L_q\right).
$$
Equations \eqref{PD_fits1} and \eqref{PD_cost1} now follow by taking $N=\left \lceil \frac{\delta}{\epsilon}\left(2L_BL_h+\delta L_q\right) \right \rceil.$ 
\end{proof}

\subsection{Primal-dual iterations for constrained problems}
\label{proofs_PD_appendix2}

\begin{proof}[Proof of \cref{PD_lemma_2}]
Our proof is similar to the technique in \cite{gao2019randomized}. Using the same arguments as the proof of \cref{PD_lemma_1}, \eqref{saddle_bound_WARPd1b} implies that for $y_2^{(0)}=0$,
\begin{align*}
&f(X_k)-f(x)+\langle AX_k,y_2\rangle_{\mathbb{R}}-\sup_{z\in C}\langle z,y_2\rangle_{\mathbb{R}}-\langle Ax,[Y_k]_2\rangle_{\mathbb{R}}+\sup_{z\in C}\langle z,[Y_k]_2\rangle_{\mathbb{R}}\\
&\quad\quad\leq \frac{1}{k}\left(\frac{\nmu{x-x^{(0)}}^2}{\tau}+\frac{\nmu{y_1-y_1^{(0)}}^2}{\sigma_1}+\frac{\nmu{y_2}^2}{\sigma_2}\right), \quad \forall x\in\mathbb{C}^n,\ y_1 \in \partial h(B X_k),\ y_2\in\mathbb{C}^{m'}.
\end{align*}
If $x\in Q$, then
$$
-\langle Ax,[Y_k]_2\rangle_{\mathbb{R}}+\sup_{z\in C}\langle z,[Y_k]_2\rangle_{\mathbb{R}}\geq 0.
$$
Let $\hat{z}\in C$ be of minimal distance to $AX_k$ and let $y_2$ be a multiple of $AX_k-\hat{z}$ such that $y_2$ has norm $\kappa$. Since $C$ is convex, the following holds \cite[Theorem 6.41]{beck2017first}
$$
\langle z,y_2\rangle_{\mathbb{R}}\leq \langle \hat{z},y_2\rangle_{\mathbb{R}},\quad \forall z\in C.
$$
It follows that
$$
\langle AX_k,y_2\rangle_{\mathbb{R}}-\sup_{z\in C}\langle z,y_2\rangle_{\mathbb{R}}\geq \langle AX_k-\hat{z},y_2\rangle_{\mathbb{R}}=\kappa\cdot\inf_{z\in C} \nm{AX_k-z}=g_Q(\kappa;X_k).
$$
Combining the inequalities yields \eqref{objective_bound_WARPd1b}.
\end{proof}

\begin{proof}[Proof of \cref{PDGH_prop_2}]
First, consider general $\tau,\sigma_1,\sigma_2>0$ with $\tau(\sigma_1 L_B^2+\sigma_2 L_A^2+L_q)=1$. For input $x_0$ with $d(x_0,\widehat{X})\leq\delta$, we argue as in the proof of \cref{PDGH_prop_1} (but now using \cref{PD_lemma_2}) to obtain
\begin{equation}
\label{objective_bound_WARPd3b}
f(X_N)-\hat{f}+g_Q(\kappa;X_N)\leq \frac{1}{N}\left(\frac{\delta^2}{\tau}+\frac{L_h^2}{\sigma_1}+\frac{\kappa^2}{\sigma_2}\right)=\frac{1}{N}\left(\sigma_1\delta^2L_B^2+\frac{L_h^2}{\sigma_1}+\sigma_2\delta^2L_A^2+\frac{\kappa^2}{\sigma_2}+\delta^2L_q\right).
\end{equation}
Optimizing the proximal step sizes leads to
$$
\tau=\frac{\delta}{\kappa L_A+L_hL_B+\delta L_q},\quad \sigma_1=\frac{L_h}{\delta L_B},\quad \sigma_2=\frac{\kappa}{\delta L_A}.
$$
Substituting these values into \eqref{objective_bound_WARPd3b} leads to
$$
f(X_N)-\hat{f}+g_Q(X_N)\leq \frac{\delta}{N}\left(2\kappa L_A+2L_hL_B+\delta L_q\right).
$$
The rest of the proof follows the same argument as the proof of \cref{PDGH_prop_1}.
\end{proof}

\linespread{0.89}\selectfont{}
\bibliographystyle{abbrv}
\bibliography{restart_bib}

\begin{thebibliography}{10}

\bibitem{adcock2022efficient}
B.~Adcock, S.~Brugiapaglia, N.~Dexter, and S.~Moraga.
\newblock On efficient algorithms for computing near-best polynomial
  approximations to high-dimensional, {H}ilbert-valued functions from limited
  samples.
\newblock {\em arXiv preprint arXiv:2203.13908}, 2022.

\bibitem{adcock2022sparse}
B.~Adcock, S.~Brugiapaglia, and C.~G. {Webster}.
\newblock {\em Sparse Polynomial Approximation of High-Dimensional Functions}.
\newblock Comput. Sci. Eng. Society for Industrial and Applied Mathematics,
  Philadelphia, PA, 2022.

\bibitem{adcock2021improved}
B.~Adcock, N.~Dexter, and Q.~Xu.
\newblock Improved recovery guarantees and sampling strategies for {TV}
  minimization in compressive imaging.
\newblock {\em SIAM J. Imaging Sci.}, 14(3):1149--1183, 2021.

\bibitem{adcock2021compressive}
B.~Adcock and A.~Hansen.
\newblock {\em Compressive Imaging: Structure, Sampling, Learning}.
\newblock CUP, 2021.

\bibitem{applegate2022faster}
D.~Applegate, O.~Hinder, H.~Lu, and M.~Lubin.
\newblock Faster first-order primal-dual methods for linear programming using
  restarts and sharpness.
\newblock {\em Mathematical Programming}, pages 1--52, 2022.

\bibitem{aster2018parameter}
R.~C. Aster, B.~Borchers, and C.~H. Thurber.
\newblock {\em Parameter estimation and inverse problems}.
\newblock Elsevier, 2018.

\bibitem{attouch2010proximal}
H.~Attouch, J.~Bolte, P.~Redont, and A.~Soubeyran.
\newblock Proximal alternating minimization and projection methods for
  nonconvex problems: {A}n approach based on the {K}urdyka-{\l{}}ojasiewicz
  inequality.
\newblock {\em Math. Oper. Res.}, 35(2):438--457, 2010.

\bibitem{auslender1988global}
A.~Auslender and J.-P. Crouzeix.
\newblock Global regularity theorems.
\newblock {\em Math. Oper. Res.}, 13(2):243--253, 1988.

\bibitem{opt_big}
A.~Bastounis, A.~C. Hansen, and V.~Vla{\v{c}}i{\'c}.
\newblock The extended {S}male's 9th problem.
\newblock {\em arXiv preprint arXiv:2110.15734}, 2021.

\bibitem{beck2017first}
A.~Beck.
\newblock {\em First-order methods in optimization}.
\newblock SIAM, 2017.

\bibitem{beck2009fast}
A.~Beck and M.~Teboulle.
\newblock A fast iterative shrinkage-thresholding algorithm for linear inverse
  problems.
\newblock {\em SIAM J. Imaging Sci.}, 2(1):183--202, 2009.

\bibitem{becker2011nesta}
S.~Becker, J.~Bobin, and E.~J. Cand{\`e}s.
\newblock {NESTA}: {A} fast and accurate first-order method for sparse
  recovery.
\newblock {\em SIAM J. Imaging Sci.}, 4(1):1--39, 2011.

\bibitem{becker2011templates}
S.~Becker, E.~J. Cand{\`e}s, and M.~C. Grant.
\newblock Templates for convex cone problems with applications to sparse signal
  recovery.
\newblock {\em Math. Program. Comput.}, 3(3):165, 2011.

\bibitem{belloni2011square}
A.~Belloni, V.~Chernozhukov, and L.~Wang.
\newblock Square-root {LASSO}: pivotal recovery of sparse signals via conic
  programming.
\newblock {\em Biometrika}, 98(4):791--806, 2011.

\bibitem{belloni2014pivotal}
A.~Belloni, V.~Chernozhukov, and L.~Wang.
\newblock Pivotal estimation via square-root {LASSO} in nonparametric
  regression.
\newblock {\em Ann. Statist.}, 42(2):757--788, 2014.

\bibitem{benlectures}
A.~Ben-Tal and A.~Nemirovski.
\newblock Lectures on modern convex optimization.
\newblock 2020/2021.

\bibitem{bolte2007lojasiewicz}
J.~Bolte, A.~Daniilidis, and A.~Lewis.
\newblock The {\l{}}ojasiewicz inequality for nonsmooth subanalytic functions
  with applications to subgradient dynamical systems.
\newblock {\em SIAM J. Optim.}, 17(4):1205--1223, 2007.

\bibitem{bolte2017error}
J.~Bolte, T.~P. Nguyen, J.~Peypouquet, and B.~W. Suter.
\newblock From error bounds to the complexity of first-order descent methods
  for convex functions.
\newblock {\em Math. Program.}, 165(2):471--507, 2017.

\bibitem{bolte2014proximal}
J.~Bolte, S.~Sabach, and M.~Teboulle.
\newblock Proximal alternating linearized minimization for nonconvex and
  nonsmooth problems.
\newblock {\em Math. Program.}, 146(1):459--494, 2014.

\bibitem{burke2002weak}
J.~Burke and S.~Deng.
\newblock Weak sharp minima revisited {Part I}: basic theory.
\newblock {\em Control Cybernet.}, 31:439--469, 2002.

\bibitem{burke1993weak}
J.~V. Burke and M.~C. Ferris.
\newblock Weak sharp minima in mathematical programming.
\newblock {\em SIAM J. Control Optim.}, 31(5):1340--1359, 1993.

\bibitem{chambolle2018stochastic}
A.~Chambolle, M.~J. Ehrhardt, P.~Richt{\'a}rik, and C.~Schonlieb.
\newblock Stochastic primal-dual hybrid gradient algorithm with arbitrary
  sampling and imaging applications.
\newblock {\em SIAM J. Optim.}, 28(4), 2018.

\bibitem{chambolle2011first}
A.~Chambolle and T.~Pock.
\newblock A first-order primal-dual algorithm for convex problems with
  applications to imaging.
\newblock {\em J. Math. Imaging Vision}, 40(1):120--145, 2011.

\bibitem{chambolle2016introduction}
A.~Chambolle and T.~Pock.
\newblock An introduction to continuous optimization for imaging.
\newblock {\em Acta Numerica}, 25:161--319, 2016.

\bibitem{chambolle2016ergodic}
A.~Chambolle and T.~Pock.
\newblock On the ergodic convergence rates of a first-order primal--dual
  algorithm.
\newblock {\em Math. Program.}, 159(1-2):253--287, 2016.

\bibitem{misc_data_sets}
C.-C. Chang and C.-J. Lin.
\newblock Libsvm: a library for support vector machines.
\newblock {\em ACM Transactions on Intelligent Systems and Technology}, 2,
  2011.

\bibitem{colbrook2021warpd}
M.~J. Colbrook.
\newblock {WARPd: A} linearly convergent first-order primal-dual algorithm for
  inverse problems with approximate sharpness conditions.
\newblock {\em SIAM Journal on Imaging Sciences}, 15(3):1539--1575, 2022.

\bibitem{colbrook2022difficulty}
M.~J. Colbrook, V.~Antun, and A.~C. Hansen.
\newblock The difficulty of computing stable and accurate neural networks: On
  the barriers of deep learning and smale's 18th problem.
\newblock {\em Proceedings of the National Academy of Sciences},
  119(12):e2107151119, 2022.

\bibitem{misc_wine_quality_186}
P.~Cortez, A.~Cerdeira, F.~Almeida, T.~Matos, and J.~Reis.
\newblock {Wine Quality}.
\newblock UCI Machine Learning Repository, 2009.

\bibitem{d2021acceleration}
A.~d'Aspremont, D.~Scieur, A.~Taylor, et~al.
\newblock Acceleration methods.
\newblock {\em Foundations and Trends in Optimization}, 5(1-2):1--245, 2021.

\bibitem{esser2010general}
E.~Esser, X.~Zhang, and T.~F. Chan.
\newblock A general framework for a class of first order primal-dual algorithms
  for convex optimization in imaging science.
\newblock {\em SIAM J. Imaging Sci.}, 3(4), 2010.

\bibitem{fercoq2022quadratic}
O.~Fercoq.
\newblock Quadratic error bound of the smoothed gap and the restarted averaged
  primal-dual hybrid gradient.
\newblock {\em arXiv preprint arXiv:2206.03041}, 2022.

\bibitem{fercoq2016restarting}
O.~Fercoq and Z.~Qu.
\newblock Restarting accelerated gradient methods with a rough strong convexity
  estimate.
\newblock {\em arXiv:1609.07358}, 2016.

\bibitem{fercoq2019adaptive}
O.~Fercoq and Z.~Qu.
\newblock Adaptive restart of accelerated gradient methods under local
  quadratic growth condition.
\newblock {\em IMA Journal of Numerical Analysis}, 39(4):2069--2095, 2019.

\bibitem{foucart2013invitation}
S.~Foucart and H.~Rauhut.
\newblock {\em A mathematical introduction to compressive sensing}.
\newblock Springer, 2013.

\bibitem{frankel2015splitting}
P.~Frankel, G.~Garrigos, and J.~Peypouquet.
\newblock Splitting methods with variable metric for
  {K}urdyka--{\l{}}ojasiewicz functions and general convergence rates.
\newblock {\em J. Optim. Theory Appl.}, 165(3):874--900, 2015.

\bibitem{freund2018new}
R.~M. Freund and H.~Lu.
\newblock New computational guarantees for solving convex optimization problems
  with first order methods, via a function growth condition measure.
\newblock {\em Math. Program.}, 170(2):445--477, 2018.

\bibitem{gao2019randomized}
X.~Gao, Y.-Y. Xu, and S.-Z. Zhang.
\newblock Randomized primal--dual proximal block coordinate updates.
\newblock {\em Journal of the Operations Research Society of China},
  7(2):205--250, 2019.

\bibitem{giselsson2014monotonicity}
P.~Giselsson and S.~Boyd.
\newblock Monotonicity and restart in fast gradient methods.
\newblock In {\em IEEE Conf Decis Control}, pages 5058--5063. IEEE, 2014.

\bibitem{gb08}
M.~Grant and S.~Boyd.
\newblock Graph implementations for nonsmooth convex programs.
\newblock In V.~Blondel, S.~Boyd, and H.~Kimura, editors, {\em Recent Advances
  in Learning and Control}, Lecture Notes in Control and Information Sciences,
  pages 95--110. Springer-Verlag Limited, 2008.
\newblock \url{http://stanford.edu/~boyd/graph_dcp.html}.

\bibitem{cvx}
M.~Grant and S.~Boyd.
\newblock {CVX}: Matlab software for disciplined convex programming, version
  2.1.
\newblock \url{http://cvxr.com/cvx}, Mar. 2014.

\bibitem{guerquin--kern2012realistic}
M.~Guerquin-Kern, L.~Lejeune, K.~P. Pruessmann, and M.~Unser.
\newblock Realistic analytical phantoms for parallel {M}agnetic {R}esonance
  {I}maging.
\newblock {\em IEEE Trans. Med. Imag.}, 31(3):626--636, 2012.

\bibitem{hoffman1952approximate}
A.~J. Hoffman.
\newblock On approximate solutions of systems of linear inequalities.
\newblock {\em J. Research Nat. Bur. Standards}, 49(4), 1952.

\bibitem{iouditski2014primal}
A.~Iouditski and Y.~Nesterov.
\newblock Primal-dual subgradient methods for minimizing uniformly convex
  functions.
\newblock {\em arXiv preprint arXiv:1401.1792}, 2014.

\bibitem{karimi2016linear}
H.~Karimi, J.~Nutini, and M.~Schmidt.
\newblock Linear convergence of gradient and proximal-gradient methods under
  the {P}olyak-{\l{}}ojasiewicz condition.
\newblock In {\em Mach Learn Knowl Discov Databases}, pages 795--811. Springer,
  2016.

\bibitem{kerdreux2019restarting}
T.~Kerdreux, A.~d'Aspremont, and S.~Pokutta.
\newblock Restarting {F}rank-{W}olfe.
\newblock In {\em The 22nd International Conference on Artificial Intelligence
  and Statistics}, pages 1275--1283. PMLR, 2019.

\bibitem{kerdreux2022restarting}
T.~Kerdreux, A.~d'Aspremont, and S.~Pokutta.
\newblock Restarting frank--wolfe: Faster rates under h{\"o}lderian error
  bounds.
\newblock {\em Journal of Optimization Theory and Applications},
  192(3):799--829, 2022.

\bibitem{kim2016optimized}
D.~Kim and J.~A. Fessler.
\newblock Optimized first-order methods for smooth convex minimization.
\newblock {\em Mathematical programming}, 159(1):81--107, 2016.

\bibitem{lin2014adaptive}
Q.~Lin and L.~Xiao.
\newblock An adaptive accelerated proximal gradient method and its homotopy
  continuation for sparse optimization.
\newblock In {\em International Conference on Machine Learning}, pages 73--81.
  PMLR, 2014.

\bibitem{lojasiewicz1963propriete}
S.~Lojasiewicz.
\newblock Une propri{\'e}t{\'e} topologique des sous-ensembles analytiques
  r{\'e}els.
\newblock {\em Les {\'e}quations aux d{\'e}riv{\'e}es partielles}, 117:87--89,
  1963.

\bibitem{mangasarian1985condition}
O.~L. Mangasarian.
\newblock A condition number for differentiable convex inequalities.
\newblock {\em Math. Oper. Res.}, 10(2):175--179, 1985.

\bibitem{necoara2019linear}
I.~Necoara, Y.~Nesterov, and F.~Glineur.
\newblock Linear convergence of first order methods for non-strongly convex
  optimization.
\newblock {\em Math. Program.}, 175(1):69--107, 2019.

\bibitem{nemirovskii1985optimal}
A.~S. Nemirovskii and Y.~E. Nesterov.
\newblock Optimal methods of smooth convex minimization.
\newblock {\em USSR Comput. Math. Math. Phys.}, 25(2):21--30, 1985.

\bibitem{nemirovskij1983problem}
A.~S. Nemirovskij and D.~B. Yudin.
\newblock Problem complexity and method efficiency in optimization.
\newblock 1983.

\bibitem{nesterov2003introductory}
Y.~Nesterov.
\newblock {\em Introductory lectures on convex optimization: {A} basic course},
  volume~87.
\newblock Springer Science \& Business Media, 2003.

\bibitem{nesterov2005smooth}
Y.~Nesterov.
\newblock Smooth minimization of non-smooth functions.
\newblock {\em Math. Program.}, 103(1):127--152, May 2005.

\bibitem{nesterov2013gradient}
Y.~Nesterov.
\newblock Gradient methods for minimizing composite functions.
\newblock {\em Math. Program.}, 140(1):125--161, 2013.

\bibitem{nesterov2015universal}
Y.~Nesterov.
\newblock Universal gradient methods for convex optimization problems.
\newblock {\em Mathematical Programming}, 152(1):381--404, 2015.

\bibitem{nesterov2018lectures}
Y.~Nesterov et~al.
\newblock {\em Lectures on convex optimization}, volume 137.
\newblock Springer, 2018.

\bibitem{nesterov1983method}
Y.~E. Nesterov.
\newblock A method for solving the convex programming problem with convergence
  rate {$O(1/k^2)$}.
\newblock In {\em Dokl. Akad. Nauk SSSR,}, volume 269, pages 543--547, 1983.

\bibitem{neyra-nesterenko2023unrolled}
M.~Neyra-Nesterenko.
\newblock {Unrolled NESTA: constructing stable, accurate and efficient neural
  networks for gradient-sparse imaging problems }.
\newblock Master's thesis, Simon Fraser University, 2023.

\bibitem{neyra-nesterenko2022nestanets}
M.~Neyra-Nesterenko and B.~Adcock.
\newblock {NESTANets}: {Stable}, accurate and efficient neural networks for
  analysis-sparse inverse problems.
\newblock {\em arXiv:2203.00804}, 2022.

\bibitem{o2015adaptive}
B.~O'donoghue and E.~Candes.
\newblock Adaptive restart for accelerated gradient schemes.
\newblock {\em Found. Comput. Math.}, 15(3):715--732, 2015.

\bibitem{pock2009algorithm}
T.~Pock, D.~Cremers, H.~Bischof, and A.~Chambolle.
\newblock An algorithm for minimizing the {M}umford--{S}hah functional.
\newblock In {\em IEEE Int Conf Comput Vis}, pages 1133--1140. IEEE, 2009.

\bibitem{renegar2016efficient}
J.~Renegar.
\newblock ``efficient'' subgradient methods for general convex optimization.
\newblock {\em SIAM Journal on Optimization}, 26(4):2649--2676, 2016.

\bibitem{renegar2019accelerated}
J.~Renegar.
\newblock Accelerated first-order methods for hyperbolic programming.
\newblock {\em Mathematical Programming}, 173(1):1--35, 2019.

\bibitem{renegar2021simple}
J.~Renegar and B.~Grimmer.
\newblock A simple nearly optimal restart scheme for speeding up first-order
  methods.
\newblock {\em Foundations of Computational Mathematics}, pages 1--46, 2021.

\bibitem{robinson1975application}
S.~M. Robinson.
\newblock An application of error bounds for convex programming in a linear
  space.
\newblock {\em SIAM J. Control}, 13(2):271--273, 1975.

\bibitem{roulet2020computational}
V.~Roulet, N.~Boumal, and A.~d'Aspremont.
\newblock Computational complexity versus statistical performance on sparse
  recovery problems.
\newblock {\em Inf. Inference}, 9(1):1--32, 2020.

\bibitem{roulet2020sharpness}
V.~Roulet and A.~d'Aspremont.
\newblock Sharpness, restart, and acceleration.
\newblock {\em SIAM J. Optim.}, 30(1):262--289, 2020.

\bibitem{optimal_saddle}
O.~Rynkiewicz.
\newblock Lower bounds and primal-dual methods for affinely constrained convex
  optimization under metric subregularity, 2020.

\bibitem{su2014differential}
W.~Su, S.~Boyd, and E.~Candes.
\newblock A differential equation for modeling {N}esterov's accelerated
  gradient method: theory and insights.
\newblock {\em Advances in neural information processing systems}, 27, 2014.

\bibitem{geer2016estimation}
S.~van~de Geer.
\newblock {\em Estimation and Testing Under Sparsity: {\'E}cole d'{\'E}t{\'e}
  de Probabilit{\'e}s de Saint-Flour XLV -- 2015}, volume 2159 of {\em Lecture
  Notes in Math.}
\newblock Springer, Cham, Switzerland, 2016.

\end{thebibliography}

\end{document}